%% file: main.tex
\documentclass[a4paper]{scrartcl}

\pdfoutput=1

\usepackage[T1]{fontenc}
\usepackage[utf8]{inputenc}
\usepackage[english]{babel}
\usepackage{lmodern}

\usepackage{amssymb,amsmath,amsthm,mathrsfs,mathtools,amscd}
\usepackage{float}
\usepackage[]{subfigure}
\usepackage{pgfplots}

\usepackage[colorlinks=true,linkcolor=cyan]{hyperref}

\pgfplotsset{compat=newest}
\newlength\figureheight	
\newlength\figurewidth	

\input{macros/stylefyle}
\input{macros/mathComAbb} 

\begin{document}
\author{Peter Benner${}^*$ \and Jan Heiland\thanks{Computational Methods in Systems and Control Theory Group at the Max Planck Institute for Dynamics of Complex Technical Systems, Sandtorstra{\ss}e~1, D-39106 Magdeburg, Germany
(\href{mailto:heiland@mpi-magdeburg.mpg.de}{\nolinkurl{heiland@mpi-magdeburg.mpg.de}})} ~}
\title{Exponential Stability and Stabilization of Extended Linearizations via  Continuous Updates of Riccati Based Feedback}
\maketitle
{
\begin{abstract}
	\input{chapters/abstract}
\end{abstract}
}
\input{macros/mathEnv}

\input{chapters/introduction}

\input{chapters/stabexlin}
\input{chapters/stabsdrefb}
\input{chapters/numexa}
\input{chapters/conclusion}

\bibliographystyle{abbrv}
\bibliography{extlinstab}  
\end{document}

%% file: macros/stylefyle.tex
\def\trp{{\mathsf T}}
\def\trsp{^\trp}

\def\inix{x_{0}}

\def\xs{\xi}
\def\xst{\xs(t)}
\def\dxst{\dot \xs(t)}
\def\Nu{p}

\def\XOT{\ensuremath{X_{[0,T]}}}
\def\XiOT{\ensuremath{{\Xi_{[0,T]}}}}
\def\XiOt{\ensuremath{{\Xi_{[0,t]}}}}

\def\fbg{\ensuremath{F}}
\def\fbgx{\ensuremath{F(x)}}
\def\fbgxt{\ensuremath{F(\xst)}}

\def\mglob{\ensuremath{K}}
\def\omglob{\ensuremath{\omega}}
\def\omstar{\ensuremath{\omega^*}}
\def\Nx{\ensuremath{n}}
\def\foralle{\quad\text{for all }}

\def\sklmox{\ensuremath{\mathcal S_{(K,L,M,\omglob;X)}}}
\def\sklmoxT{\ensuremath{\mathcal S_{(K,L,M_T,\omglob;X_T)}}}
\def\smo{\ensuremath{\mathcal S_{K,\omglob}}}
\def\stmo{\ensuremath{\mathcal S_{\tilde K,\omglob}}}

\def\sylop{\ensuremath{\mathcal P}}
\DeclareMathOperator{\vecop}{vec}

\def\lipca{\ensuremath{L}}
\def\undecrt{\ensuremath{\omega}}
\def\ungrwthc{\ensuremath{K}}

\def\tA{\ensuremath{\tilde A}}
\def\Rdeltapmo{\ensuremath{R_\Delta}}
\def\Qdelta{\ensuremath{Q_\Delta}}
\def\Pdelta{\ensuremath{P_\Delta}}
\def\Adelta{\ensuremath{A_\Delta}}

\def\tstar{\ensuremath{{t^*}}}

\def\sylrhs{C}
\DeclareMathOperator{\sep}{sep}

\def\sveco{\xi_1}
\def\svecf{\xi_4}
\def\svec{\begin{bmatrix} \sveco \\ \xi_2 \end{bmatrix}}
\def\dsvec{\dot {\begin{bmatrix} \sveco \\ \xi_2 \end{bmatrix}}}
\def\svecwu{\begin{bmatrix} \xi_1 \\ \xi_2\\ \xi_3\\ \xi_4\\ \xi_5 \end{bmatrix}}
\def\dsvecwu{\dot {\begin{bmatrix} \xi_1 \\ \xi_2 \\ \xi_3 \\ \xi_4 \\ \xi_5 \end{bmatrix}}}

\newcommand{\scipy}{\textsf{SciPy}}
\newcommand{\sdre}{\textsf{sdre}}
\newcommand{\sylvup}{\textsf{p-update}}

%% file: macros/mathComAbb.tex
\def\trp{{\mathsf T}}
\def\trsp{^\trp}

\providecommand{\abs}[1]{\lvert#1\rvert}

\providecommand{\norm}[1]{\lVert#1 \rVert}

\providecommand{\inva}[1]{\text{~\textup{d}} #1}

\providecommand{\bbmat}{\begin{bmatrix}}
\providecommand{\ebmat}{\end{bmatrix}}


%% file: chapters/abstract.tex
Many recent works on stabilization of nonlinear systems target the case of locally stabilizing an unstable steady state solutions against small perturbation. In this work we explicitly address the goal of driving a system into a nonattractive steady state starting from a well developed state for which the linearization based local approaches will not work. Considering extended linearizations or state-dependent coefficient representations of nonlinear systems, we develop sufficient conditions for stability of solution trajectories. We find that if the coefficient matrix is uniformly stable in a sufficiently large neighborhood of the current state, then the state will eventually decay. Based on these analytical results we propose an update scheme that is designed to maintain the stabilization property of Riccati based feedback constant during a certain period of the state evolution. We illustrate the general applicability of the resulting algorithm for setpoint stabilization of nonlinear autonomous systems and its numerical efficiency in two examples.


%% file: macros/mathEnv.tex
\newtheorem{thm}{Theorem}[section]
\newtheorem{cor}[thm]{Corollary}
\newtheorem{lem}[thm]{Lemma}
\newtheorem{ass}[thm]{Assumption}

\theoremstyle{remark}
\newtheorem{exa}[thm]{Example}

\theoremstyle{remark}
\newtheorem{rem}[thm]{Remark}

\theoremstyle{definition}
\newtheorem{defin}[thm]{Definition}

\theoremstyle{plain}
\newtheorem{prob}[thm]{Problem}

\newtheoremstyle{named}{}{}{\itshape}{}{\bfseries}{.}{.5em}{#1 \thmnote{#3}}
\theoremstyle{named}
\newtheorem*{namedproblem}{Problem}

\theoremstyle{plain}
\newtheorem{prop}[thm]{Proposition}

\theoremstyle{plain}
\newtheorem{alg}[thm]{Algorithm}

\floatstyle{ruled}
\newfloat{algorithm}{h}{loa}[section]
\floatname{algorithm}{Algorithm}

%% file: chapters/introduction.tex
\section{Introduction}

We consider the general task to find an input $u$ that drives the state $\zeta$ of a nonlinear autonomous input-affine system of type
\begin{equation*}
	\dot \zeta(t) = f(\zeta(t))+Bu(t), \quad \zeta(0) = z \in \mathbb R^{\Nx} ,
\end{equation*}
towards a steady state $z^*$, i.e. a state $z^*$ for which $f(z^*)=0$. This problem is commonly known as \emph{set point stabilization}. It is equivalent to considering $\xi = \zeta - z^*$ and the task to drive the difference state $\xi$, that satisfies 
\begin{equation} \label{eq:nonlsysdiff}
	\dxst = \tilde f(\xst) + Bu ,\quad  \xi(0) = \inix,
\end{equation}
to zero, where $\tilde f(\xst) :=f(\xst + z^*)$ and $\inix=z-z^*$. If $f$ is Lipshitz continuous and since $\tilde f(0) = 0$, there exists \cite{Cim12} a matrix valued function $A\colon \mathbb R^{\Nx} \to \mathbb R^{\Nx, \Nx}$ such that \eqref{eq:nonlsysdiff} can be written as 
\begin{equation}
	\dxst = A(\xs(t))\xs(t) + Bu(t), \quad \xs(0)= \inix.
	\label{eq:extlinsys}
\end{equation}

Thus, \emph{extended linearizations} or \emph{state dependent coefficient} (SDC) systems like \eqref{eq:extlinsys} are a suitable starting point for general nonlinear set point stabilization problems. Then the question is, how to define a feedback gain $\fbg(\xs(t))$ such that solutions of the closed loop system
\begin{align}
	\dot \xs(t) &= [A(\xs(t))-B\fbgxt]\xs(t), \quad \xs(0)= \inix, \label{eq:clsys}
	\intertext{or, equivalently,}
	\dot \zeta(t) &= f(\zeta(t))-BF(\zeta(t)-z^*)[\zeta(t)-z^*], \quad \zeta(0) = z, \notag
\end{align}
decay asymptotically to zero or to $z^*$, respectively.  One approach is to define the feedback gain as $\fbgx = R^{-1}B^TP(x)$ for a given state $x =\xst$, where $P(x)$ is the solution to the \emph{state dependent Riccati equation (SDRE)}
\begin{equation}\label{eq:sdre}
	P(x)A(x) + A\trsp P(x) - P(x)BR^{-1}B\trsp P(x) + Q = 0,
\end{equation}
for given weighting matrices $R\succ0$ and $Q\succcurlyeq0$. 

Known results \cite{BanLT07, Cim12, MraC98} on the stabilization via SDRE feedback base on the assumption that the initial state $\inix$ is close to zero such that the nonlinear terms are but a perturbation of a linear system which can then be stabilized. Precisely, one considers the SDRE \eqref{eq:sdre} for the extended linear system \eqref{eq:extlinsys} and defines $P(x) =: P(\inix) + \Delta P(x)$ and $A(x) =: A(\inix) + \Delta A(x)$. Then, if $\fbg_0:= B\trsp P(\inix)$ is stabilizing for $A_0 := A(\inix)$ and if the considered matrix functions are Lipshitz continuous in $x$, then one can show that the solution to 
\begin{equation*}
	\dxst = (A_0 - B\fbg_0)\xst + h(t), \quad \xs(0) = \inix,
\end{equation*}
where $h(t) := (\Delta A(\xs(t)) + BB^T\Delta P(\xs(t)))\xs(t)$, goes to zero as $t\to \infty$ with an exponential decay rate \cite{BanLT07}, provided that $\inix$ is sufficiently small.

Our goal, however, is to drive a system from a developed state towards the zero state, which contradicts the smallness assumption on the initial value. Once the system's state is close to the origin, stabilization strategies that base on smallness of the deviation from the zero state and that have been proven successful can be applied; see \cite{BaeBSetal15, BenH15, BreK14} for numerical studies considering nonlinear PDEs and \cite{Ray06} for a theoretical analysis. For completeness, we mention the earlier works on feedback synthesis for nonlinear systems based on extended linearizations \cite{BauR86, Rug84}, where families of feedback gains parametrized by set points of the considered plants were considered. There again, the analysis of the stabilizing properties base on smallness of the deviations from the targeted operating points.

The manuscript is organized as follows. In Section \ref{sec:extlinstab}, we extend the results that were reported in \cite{HilI11} on stability of linear time-varying systems like
\begin{equation*}
	\dot \xs(t) = \tilde A(t) \xs(t),
\end{equation*}
to give sufficient conditions for stability of SDC
systems like system \eqref{eq:extlinsys}. The basic idea is that for a given trajectory $\xs$, one can consider $\tilde A(t):=A(\xs(t))$. However, this approach leads to sufficient conditions that are very restrictive and probably not easy to confirm for most applications. In view of practical use, in Section \ref{sec:locextlinstab} we provide localized conditions taking advantage of the observation that with controlling the state $\xs$, one also controls the coefficients. By means of an example, we show the practicability of the derived estimates.

The general result is that one can achieve an exponential decay of the solutions if, at a fixed state $x$, the local transient behavior is well balanced with the decay rate of the current coefficient $A(x)$ and if this balance holds true uniformly in a sufficiently large neighborhood.  In Section \ref{sec:stabsdrefb}, we will introduce conditions and an algorithm for a feedback gain $\fbg$ that ensures uniform bounds on the transitive behavior and a constant decay rate in a neighborhood of the current state via continuously updating an initial feedback. The resulting algorithm is theoretically well founded and generally applicable for set point control of any nonlinear autonomous system that can be written in SDC form. In Section \ref{sec:numexa} we investigate the proposed update scheme for two numerical examples and show its feasibility and efficiency in comparison to the SDRE feedback. We conclude with summarizing remarks and an outlook.

%% file: chapters/stabexlin.tex
\section{Stability of State-dependent Coefficient Systems} \label{sec:extlinstab}

To describe exponential stability for the considered type of SDC systems
\begin{equation}\label{eq:extlinsysplain}
	\dot \xs(t) = A(\xs(t))\xs(t), 
\end{equation}
we adjust the definition for time varying systems as given in \cite[Def. 6.5]{Rug96}.
\begin{defin}\label{def:uniexpstable}
	System \eqref{eq:extlinsysplain} is called \emph{uniformly exponentially stable} if there exist positive constants $\mglob$ and $\omglob$ such that for any $x_0\in \mathbb R$, a solution $\xs$ of \eqref{eq:extlinsysplain} with $\xs(0)=\inix$ satisfies
	\begin{equation} \label{eq:uniexpstab}
		\norm{\xs(t)} \leq \mglob e^{-\omglob t}\norm{\inix}, \quad \text{for }t\geq 0.
	\end{equation}
	It is called \emph{uniformly exponentially stable on $X$}, if, for some $X\subset\mathbb R^{\Nx}$, relation \eqref{eq:uniexpstab} holds for any $\inix \subset X$.
\end{defin}

Note that the definition in \cite{Rug96} is for linear systems but \eqref{eq:uniexpstab} solely bases on solution trajectories and, thus, applies also for nonlinear systems. 

\begin{ass}\label{ass:contiHomoBounded}
	Regarding equation \eqref{eq:extlinsysplain}, we have that

	\begin{itemize}
		\item[(1)] the map $A\colon \mathbb R^{\Nx}\to \mathbb R^{\Nx, \Nx}$ is Lipshitz-continuous,
		\item[(2)] there is a bounded set $X \subset \mathbb R^{\Nx}$ such that $\xs(t)\in X$, for $t\geq 0$, where $\xs$ is a solution to \eqref{eq:extlinsysplain}, with $\xs(0)=\inix \in X$.
	\end{itemize}
\end{ass}

The following lemma states that in order to state exponential stability for trajectories that start in $X$, the existence of a global unique solution is a necessary prerequisite.

\begin{lem}
	Consider equation \eqref{eq:extlinsysplain} and let Assumption \ref{ass:contiHomoBounded} hold. Then, for any $\inix\in X$, there is a unique solution solution $\xs \colon [0, \infty) \to \mathbb R^{\Nx}$ to \eqref{eq:extlinsysplain} with $\xs(0)=\inix$. 
\end{lem}

\begin{proof}
	By Lipshitz-continuity of $A$, it follows that $x\mapsto A(x)x$ is locally Lipshitz continuous. Accordingly, by the \emph{Picard-Lindel\"of} theorem, there exists a unique solution $\xs$ locally in time. Since, by assumption, $\xs$ stays in the bounded set $X$, it can be extended to a global solution. 
\end{proof}

We introduce a class of SDC matrices similar to the class of time-dependent coefficient matrices used in \cite{HilI11} via the following assumption.
\begin{ass}\label{ass:unilipunidecA}
	For a given bounded set $X \subset \mathbb R^{\Nx}$, the function $A\colon X \to \mathbb R^{\Nx,\Nx}$ is \emph{Lipshitz continuous}, i.e. there exists a constant $\lipca\in\mathbb R^{}$ such that
	\begin{equation} \label{eq:lipcontA}
		\norm{A(x_1) - A(x_2)} \leq \lipca \norm{x_1 - x_2}, \foralle x_1,x_2\in X,
	\end{equation}
	and \emph{uniformly stable} on $X$, i.e. there exist constants \undecrt, $\ungrwthc \in \mathbb R^{}_{>0}$ such that 
	\begin{equation} \label{eq:unigrowrateA}
		\norm{e^{A(x)s}} \leq \ungrwthc e^{-\undecrt s}, \foralle x\in X \text{ and for }t>0.
	\end{equation}
\end{ass}

\begin{lem}\label{lem:sollipcont}
	Consider equation \eqref{eq:extlinsysplain} and let Assumption \ref{ass:contiHomoBounded} and Assumption \ref{ass:unilipunidecA} hold. Then $$M:=\sup_{x\in X}\norm{A(x)x} < \infty$$ and any solution to \eqref{eq:extlinsysplain} that starts in $X$ is Lipshitz continuous with Lipshitz constant $M$.
\end{lem}
\begin{proof}
	Since $A$ is Lipshitz continuous and $X$ is bounded, $\norm{A(x)}$ and, thus, $\norm{A(x)x}$ is bounded away from $\infty$ for all $x\in X$.
	By assumption, a solution $\xs$ to \eqref{eq:extlinsysplain} that starts in $X$ stays in $X$ so that we can estimate
	\begin{equation} \label{eq:lcontix}
		\norm{\xs(t_2) - \xs(t_1)} = \norm{\int_{t_1}^{t_2}\dot \xs(s) \inva s} =
		\norm{\int_{t_1}^{t_2}A(\xs(s))\xs(s) \inva s} \leq M \abs{t_2 - t_1},
	\end{equation}
	for $t_1$, $t_2 > 0$.
\end{proof}

By virtue of Lemma \ref{lem:sollipcont}, the following definition, which we use for later reference, is well posed. 

\begin{defin}\label{def:classsklmox}
	The matrix-valued function $A\colon X \subset \mathbb R^{n}\to \mathbb R^{n,n}$ is an element of the class $\sklmox$ for some constants $K$, $L$, $M$, $\omega$ and a bounded set $X$, if $A$ and $K$, $L$, and $\omega$ are such that Assumption \ref{ass:contiHomoBounded} and Assumption \ref{ass:unilipunidecA} hold on $X$ and if $\sup_{x\in X}\norm{A(x)x} \leq M$.
\end{defin}

We can now provide an estimate on the exponential growth of solutions of the SDC system \eqref{eq:extlinsysplain}.

\begin{thm}\label{thm:expstabsolsplain}
	Consider Equation \eqref{eq:extlinsysplain} and let Assumption \ref{ass:contiHomoBounded} and Assumption \ref{ass:unilipunidecA} hold. Then, for any $\inix\in X$, the unique solution solution $\xs \colon [0, \infty) \to \mathbb R^{\Nx}$ to \eqref{eq:extlinsysplain} with $\xs(0)=\inix$ satisfies
	\begin{equation} \label{eq:expstabsolsplain}
		\norm{\xs(t)} \leq K e^{t\cdot (\sqrt{KLM \log 2}-\omega)}\norm{\inix}, \foralle t>0,
	\end{equation}
	where $M:=\sup_{x\in X}\norm{A(x)x}$.
\end{thm}
\begin{cor}
	Under the assumptions of Theorem \ref{thm:expstabsolsplain}, if 
	\begin{equation}\label{eq:condunistab}
		KLM \log 2< \omega^2,
	\end{equation}
	then system \eqref{eq:extlinsysplain} is uniformly exponentially stable on $X$ as defined in Definition \ref{def:uniexpstable}.
\end{cor}

To prove Theorem \ref{thm:expstabsolsplain} we extend the arguments used in \cite{HilI11} to prove this result for linear time-varying systems. The basic idea is that for a given trajectory $x$, the state-dependent coefficient $A$ can be considered as a time-dependent coefficient $\tA(t) := A(x(t))$. We repeat the basic steps of the proof for time-dependent linear systems, to show how the arguments extend to state-dependent coefficient matrices.

\begin{lem}[Lem. 5.2, \cite{HilI11}]\label{lem:solestimate}
	Suppose that $A\in\sklmox$. Then for any $t$, $\rho \geq 0$, every solution $\xs$ of \eqref{eq:extlinsysplain} with $\xs(0) \in X$ satisfies
	\begin{equation} \label{eq:priminineq}
		\norm{\xs(t)} \leq Ke^{-\omega t}\norm{\xs(0)} + KLM\int_0^t \abs{s - \rho} e^{-\omega (t-s)}\norm{\xs(s)}\inva s.
	\end{equation}
\end{lem}
\begin{proof}
	For a given solution $\xs$ and $t$, $\rho \geq 0$, define $A_\rho:=A(\xs(\rho))$ and rewrite \eqref{eq:extlinsysplain} as 
	\begin{equation*}
		\dot \xs(t) = A_\rho \xs(t) + (A(\xs(t)) - A_\rho)\xs(t)
	\end{equation*}
	to get the following representation of $\xs$:
	\begin{equation*}
		\xs(t) = e^{A_\rho t}\xs(0) + \int_0^t e^{A_\rho(t-s)}(A(\xs(s)) - A_\rho)\xs(s)\inva s.
	\end{equation*}
	Then, taking the norm and using the estimates \eqref{eq:lipcontA}, \eqref{eq:unigrowrateA}, and \eqref{eq:lcontix}, namely the Lipshitz continuity of $A$, the stability of $A_\rho$, and the Lipshitz continuity of $\xs$, we estimate that

	\begin{align}
		\xs(t) &\leq \norm{e^{A_\rho t}}\norm{\xs(0)} + \int_0^t \norm{e^{A_\rho(t-s)}}\norm{(A(\xs(s)) - A_\rho)}\norm{\xs(s)}\inva s \notag \\
		&\leq Ke^{-\omega t}\norm{\xs(0)} + \int_0^t K e^{-\omega(t-s)}L\norm{\xs(s) - \xs(\rho)}\norm{\xs(s)}\inva s \notag \\
		&\leq Ke^{-\omega t}\norm{\xs(0)} + KL \int_0^t e^{-\omega(t-s)}M\abs{s - \rho}\norm{\xs(s)}\inva s \label{eq:solestimate}
	\end{align}
	and arrive at inequality \eqref{eq:priminineq}. 
\end{proof}
The resulting inequality \eqref{eq:priminineq} can be parametrized through a function $r\colon \mathbb R^{}_{\geq 0}\to \mathbb R^{}_{\geq 0}$ and a scaling of the solution $\xs$ and the time $t$ to give:

\begin{lem}[Lem. 5.3, \cite{HilI11}]\label{lem:scalethexi}
	Suppose that $A\in\sklmox$ and consider a bounded piecewise continuous function $r\colon \mathbb R^{}_{\geq 0}\to \mathbb R^{}_{\geq 0}$. Then for any solution $\xs$ of \eqref{eq:extlinsysplain} with $0 \neq \xs(0) \in X$, the function
	\begin{equation}
		\zeta \colon \mathbb R^{}_{\geq 0}\to \mathbb R^{}_{\geq 0}, \quad \zeta(t):=e^{\omega t / \alpha}\frac{\norm{\xs(t/\alpha)}}{K\norm{\xs(0)}}, \quad \text{where }\alpha:=\sqrt{KLM},
		\label{eq:scaledsolandt}
	\end{equation}
	satisfies
	\begin{equation} \label{eq:scaledintineq}
		\zeta(t)\leq 1 + \int_0^t\abs{s-r(t)}\zeta(s)\inva s,
	\end{equation}
	for all $t\geq 0$.
\end{lem}
\begin{proof}
	See the proof in \cite{HilI11} and replace $L$ by $LM$.
\end{proof}
Next, one can prove an integral comparison lemma:
\begin{lem}[Lem. 5.5, \cite{HilI11}]\label{lem:intsupsol}
For $A\in\sklmox$, for any $r$, $v\colon \mathbb R_{\geq 0} \to \mathbb R_{\geq 0}$ that are bounded and piecewise continuous and that satisfy
\begin{equation}\label{eq:intsupsol}
	v(t)\geq 1 + \int_0^t\abs{s-r(t)}\zeta(s)\inva s, 
\end{equation}
for some $t_0>0$ and for $t\in[0, t_0]$, the function $\zeta$ defined in \eqref{eq:scaledsolandt} satisfies
\begin{equation*}
	\zeta(t) \leq v(t),
\end{equation*}
for all $t\in[0, t_0]$.
\end{lem}
We can now prove Theorem \ref{thm:expstabsolsplain}:
		\begin{proof}
			If $A\in\sklmox$, then for any solution $\xs$ of \eqref{eq:extlinsysplain}, it holds that 
			$$e^{\omega t / \sqrt{KLM}}\frac{\norm{\xs(t/\sqrt{KLM})}}{K\norm{\xs(0)}}=:\zeta(t),$$
			with $\zeta$ satisfies \eqref{eq:scaledintineq}, cf. Lemma \ref{lem:scalethexi}. Let
\begin{equation*}
	v_2(t):=e^{t\cdot \sqrt{\log 2}}, \quad r_2(t) := \max\{0, t-\sqrt{\log 2}\}.
\end{equation*}
Then, by \cite[Lem. 5.8]{HilI11}, the functions $v_2$ and $r_2$ satisfy \eqref{eq:intsupsol} for all $t\geq 0$, such that, by Lemma \ref{lem:intsupsol}, the function $v_2$ is a supersolution, i.e. $\zeta(t)\leq v_2(t)$ at any time $t\geq 0$. Accordingly 
			$$e^{\omega t / \sqrt{KLM}}\frac{\norm{\xs(t/\sqrt{KLM})}}{K\norm{\xs(0)}}\leq e^{t\cdot \sqrt{\log 2}}$$
or, having undone the scalings,
$$\norm{\xs(t)} \leq K e^{t\cdot (\sqrt{KLM \log 2}-\omega)}\norm{\xi(0)},$$
		for all $t\geq 0$. 
		\end{proof}
\begin{rem}
	For $K<2$, the factor $\log 2$ in \eqref{eq:uniexpstab} can be replaced by $\log K$, see \cite[Thm. 2.1]{HilI11}. 
\end{rem}

\section{Local Conditions for Exponential Stability}
\label{sec:locextlinstab}
Relation \eqref{eq:condunistab} illustrates the nature of the stability results. For the parametrization $\tilde A(t) := A(\xs(t))$, the constant $LM$ is the Lipshitz constant of $t\mapsto \tilde A (t)$. Accordingly, the requirement that $LM$ must not exceed some value defined by the decay rate $\omglob$ and the bound $\mglob$
means that the changes in $\tilde A$, that may trigger new transient phases faster than the overall decay fades them out, should be limited. 

In the linear time varying case, if one considers global constants $\omglob$ and $\mglob$, one also needs a global bound on the $LM$, since the overall decay of the solution can be violated by a sudden change in $\tilde A$ at any time. Also, in the linear time varying case, the function $t\to \tilde A(t)$ is known for all time so that a global bound can be found. Improvements of the results may be obtained by relating $\mglob$, $\omglob$, and $LM$ locally in time. However, due to the arbitrariness of the mapping $t\mapsto \tilde A(t)$, such localizations would be very problem dependent.

Things are different for the extended linearizations. The mapping $t \mapsto A(\xs(t))$ is less arbitrary, since $A(\xs(t))$ will be stabilized together with the solution $\xs$. If the function $x \mapsto A(x)$ is smooth, then, for $\xs(t)\to 0$, the coefficient $A(\xs(t))$ approaches a constant value. In fact, when having reached or when starting from a state close to zero, exponential decay can be established by the results on \emph{local exponential stability} \cite{BreK14} or on \emph{almost linear systems}, cf. the proof for the SDRE stabilization properties in \cite{BanLT07}. On the other hand, for an arbitrary starting value, a global bound on $M(x) = \norm{A(x)x}$ might not be available or too conservative. Thus, the results provided only apply to particular classes of problems for which the existence of the system invariant subspace $X$ is known or to particular given trajectories.

The following results address sufficient conditions for exponential decay of solution trajectories at discrete time instances that can be locally estimated by means of bounds on the growth of the solution in a certain time interval.
This decay at discrete instances will eventually drive the system into a state close to zero from where the linear theory will provide exponential decay. The piecewise in time character of the results that follow can also be used to define feedback laws that act locally.

We drop the global assumption on the existence of a system invariant subspace $X\subset \mathbb R^{\Nx}$, cf. Assumption \ref{ass:contiHomoBounded}(2), and consider a set of initial values and a set that contains all states that evolve from these initial values within a finite time horizon.

\begin{defin}\label{def:X0XT}
	Let $X_0 \subset \mathbb R^{\Nx}$ be a connected closed set that contains the origin and let $T\geq 0$. 
	\begin{itemize}
		\item[a.)] By $\XiOT$ we denote the set of all solution trajectories that start in $X_0$:
			\begin{equation*}
				\XiOT : = \{\xs\colon [0,T] \to \mathbb R^{\Nx} : \xs \text{ solves \eqref{eq:extlinsysplain} and }\xs(0)\in X_0 \}.
			\end{equation*}
		\item[b.)] By $X_T$ we denote the set that contains all final values of the trajectories 
			\begin{equation*}
				X_T := \{\xs(T) : \xs \in \XiOT\}.
			\end{equation*}
		\item[c.)] By $\XOT$ we denote the set that contains all values that are achieved by the solution trajectories within the time interval $[0, T]$:
			\begin{equation*}
				\XOT := \{\xs(t) : \xs \in \XiOT, ~0\leq t \leq T\}.
			\end{equation*}
	\end{itemize}
 If any solution to \eqref{eq:extlinsysplain} that starts in $X_0$ has a finite escape time $t_f<T$, we set $X_T:=\XOT:=\mathbb R^{\Nx}$. 
\end{defin}


The definition of $\sklmox$, cf. Definition \ref{def:classsklmox}, readily extends to $\sklmoxT$, if one assumes that for an element $A\colon X \subset \mathbb R^{n}\to \mathbb R^{n,n}$ and $T>0$, there exist constants $K$, $L$, $M_T$, $\omega$ such that Assumption \ref{ass:contiHomoBounded} and Assumption \ref{ass:unilipunidecA} hold and such that $M_T:=\sup_{x\in \XOT}\norm{A(x)x} < \infty$ is valid on the set $\XOT$. Note that in the case of solutions of finite escape time less than $T$, the set $\XOT$ is not bounded and the latter assumption $M_T<\infty$ does not hold, cf. Definition \ref{def:X0XT}.

\begin{rem}
	We will assume that the pointwise stability constants $\omglob$ and $\mglob$ and the Lipshitz constant $L$ are independent of the state. The uniformity of the stability constants will be used to state global convergence and is going to be a design target of a feedback stabilization. The uniformity of the Lipshitz constant is given for the case that $A$ is affine linear in $x$. Also, a state dependent $L$ can be treated with the same approach illustrated below.
\end{rem}

In the following theorem, we provide a local condition for exponential decay at discrete time instances of trajectories that start in $X_0$. The basic reasoning is that if for a time $t^*$ all trajectories are in a set that is contained in the considered set of initial values $X_0$ then, because of the autonomy of the system, the system states will be contained in $X_{[0,t^*]}$ thereafter. Accordingly, if one can establish exponential decay for the short time horizon, then the decay will hold on for the whole time axis. Having stated the basic result, we refine it by providing a dynamic bound which can replace the static constant $LM_T$, which is sharper, and which can be evaluated numerically.

\begin{thm}\label{thm:locconexpdec}
	For a given $T>0$, let $A \in \sklmoxT$ and for $0 \leq t \leq T$, let $M_t := \sup_{x\in X_t}\norm{A(x)x}$. If for a \tstar, with $0< \tstar \leq T$, \begin{equation*} 
		-\omega_{\tstar} :=  \sqrt{KLM_{\tstar}\log 2}-\omglob
	\end{equation*}
	and 
	\begin{equation*}\label{eq:locexplocomega}
		-\omega^* : =\frac{\log \mglob}{\tstar} - \omega_{\tstar} 
	\end{equation*}
	are negative, then the snapshots $\xs(t)$ of any solution $\xs$ to \eqref{eq:extlinsysplain} with $\xs(0) = x_0 \in X_0$ taken on the discrete grid $\mathcal T^*:=\{t \colon t = N\cdot \tstar,~ N=0,1,\dotsc \}$ decay exponentially in the sense that
	\begin{equation*}
		\norm{\xs(t)} \leq \norm{x_0}e^{-\omega^* t}, \quad \text{for all } t\in \mathcal T^* .
	\end{equation*}
\end{thm}
\begin{proof}
	The assumptions made include that $M_T<\infty$ so that for every $x_0\in X_0$ the associated solution $\xs$ to \eqref{eq:extlinsysplain} that starts in $x_0$ exists on $[0,T]$. Noting that by definition the bound $M_t = \sup_{x\in X_t}\norm{A(x)x}$ grows with $t$ and noting that Theorem \ref{thm:expstabsolsplain} is also valid on a finite time horizon, any such solution $\xs$ fulfills
	\begin{equation*}
		\norm{\xs(t)} \leq K e^{-\omega_t t}\norm{x_0}=e^{(\frac {\log \mglob}{t} - \omega_t)t}\norm{x_0}, \quad \text{for }0<t\leq T,
	\end{equation*}
	with $\omega_t : = \sqrt{KLM_t \log 2} -\omglob$. Thus, if there exists a $\tstar$ such that $\omega_\tstar$ and $\omega^*$ as defined in \eqref{eq:locexplocomega} are negative, then at $t^*$ any such solution $\xs$ fulfills 
	\begin{equation*}
		\norm{\xs(\tstar)} \leq \norm{x_0}e^{-\omega^*\tstar},
	\end{equation*}
	with $e^{-\omega^*\tstar}<1$. Accordingly, the current value $\xs(\tstar)$ is in a ball $X_0^* \subset X_0$. Repeating the previous arguments with $X_0$ replaced by $X_0^*$ and $x_0$ by $x(\tstar)$ and noting that the new constants $K$, $L$, and $M_t$ will be smaller than the previous, we can directly state the estimate
	\begin{equation*}
		\norm{\xs(2\tstar)} \leq \norm{\xs(\tstar)}e^{-\omega^*\tstar} \leq \norm{x_0}e^{-\omega^* 2\tstar},
	\end{equation*}
	which, by induction, holds for any multiple of $t^*$.
\end{proof}

Next, we replace the static constant $LM_t$ by a dynamic estimate that bases on differential and integral mean values.

\begin{lem}\label{lem:meanvldconsts}
	 For a given $T>0$, let $A \in \sklmoxT$ be smoothly differentiable. If also the chosen norm $\norm{\cdot}$ is smoothly differentiable, then the constant $LM_t$ in Theorem \ref{thm:locconexpdec} can be replaced by 
	\begin{equation}
		m_t := \inf_{\rho \in \mathbb R^{}_{\geq 0}} \sup_{\xs \in \XiOt}  
		\frac{\int_0^t e^{-\omega(t-s)} \norm{A(\xs(s)) - A(\xs(\rho))}\norm{\xs(s)}\inva s}{\int_0^t e^{-\omega(t-s)} \abs{s-\rho}\norm{\xs(s)}\inva s}.
		\label{eq:meanvaluemeanint}
	\end{equation}
\end{lem}
\begin{proof}
	Under the given assumptions, for $\xi \in \XiOT$ and $\rho \in \mathbb R_{\geq 0}$, the function $f_\rho \colon (0,T) \to \norm{A(\xi(s)) - A(\xi(\rho))}$ is differentiable so that, by the \emph{Mean-Value Theorem}, there exists an $s_m \in (\min\{s, \rho\}, \max\{s, \rho\})$ such that
	\begin{equation} \label{eq:fminvalue}
		f_\rho (s) - f_\rho(\rho) = \norm{A(\xi(s)) - A(\xi(\rho))} = \dot f_\rho(s_m)\abs{s - \rho}.
	\end{equation}
	Accordingly, we can rewrite the estimate $\eqref{eq:solestimate}$ in Lemma \ref{lem:solestimate} as 
	\begin{equation*}
		\xi(t) \leq Ke^{-\omega t}\norm{\xs(0)} +  K\int_0^t e^{-\omega(t-s)}\dot f_\rho(s_m)\abs{s - \rho}\norm{\xs(s)}\inva s .
	\end{equation*}
	Since the function $s \mapsto e^{-\omega(t-s)}\abs{s - \rho}\norm{\xs(s)}$ is continuous and positive there exists a constant $\tilde m$ such that 
	\begin{equation} \label{eq:gweightedaverage}
		\int_0^t e^{-\omega(t-s)}\dot f_\rho(s_m)\abs{s - \rho}\norm{\xs(s)}\inva s  
		= \tilde m \int_0^t e^{-\omega(t-s)}\abs{s - \rho}\norm{\xs(s)}\inva s.
	\end{equation}
	If $\xi(s)=0$ for all $s$, we set $\tilde m = 0$. For all other cases, we substitute  
	\begin{equation*}
		\dot f_\rho(s_m) =	\frac{\norm{(A(\xi(s)) - A(\xi(\rho))}}{\abs{{f_\rho (s) - f_\rho(\rho)}}},
	\end{equation*}
	cf. \eqref{eq:fminvalue}, which by the differentiability of $f_\rho$ is well defined also for $s=\rho$, to compute
	\begin{equation*}
 \tilde m =
 \frac{\int_0^t e^{-\omega(t-s)}\norm{(A(\xi(s)) - A(\xi(\rho))}\norm{\xs(s)}\inva s  }{
 \int_0^t e^{-\omega(t-s)}\abs{s - \rho}\norm{\xs(s)}\inva s},
	\end{equation*}
	by virtue of \eqref{eq:gweightedaverage}. Finally, the desired estimate \eqref{eq:meanvaluemeanint} holds true, if one takes the worst case estimate with respect to the possible trajectories $\xs \in \XiOT$ for a given $\rho \in (0,T)$ that possibly has been optimized in order to make the estimate as small as possible.
\end{proof}


We illustrate the use and computability of the condition formulated in Theorem \ref{thm:locconexpdec} with the improved bounds introduced in Lemma \ref{lem:meanvldconsts} by means of an example.

\begin{exa}~
	Consider the following parametrized SDC system
	\begin{equation}\label{eq:exa_sdc_osci}
		\dsvec = 
		\begin{bmatrix}
			-1 & -(1 + \sveco^2) \\
			1 + \sveco^2 & \alpha
		\end{bmatrix}\svec, \quad \xi(0)=x_0 \in X_0,
	\end{equation}
	with a system matrix $A(x)$ that for any $x=
	\begin{bmatrix} \xi_1 & \xi_2 \end{bmatrix} ^\trp\in \mathbb R^{2}$ and for $\alpha \in [-1, 1]$ has the two eigenvalues $\lambda_1$, $\lambda_2$ with real part $\Re(\lambda_1)=\Re(\lambda_2)=\frac 12 (-1+\alpha)$. Moreover, since $\Im \lambda_1 \neq \Im \lambda_2$, the matrix is diagonalizable so that the constant $K$ in \eqref{eq:unigrowrateA} can be computed as the condition number of the eigenvector matrix. Finally, given the set of initial values $X_0$, one can estimate $m_t$, cf. \eqref{eq:meanvaluemeanint}, through examining the solution trajectories to \eqref{eq:exa_sdc_osci} that start on a discrete grid in $X_0$. Thus, one can numerically check the existence of a $\tstar$, such, that for given $\alpha$ and $X_0$ it holds that
	\begin{equation}
		-\omega^* : =\frac{\log \mglob}{\tstar} + \sqrt{Km_\tstar\log 2}- \omglob \label{eq:omstar_oneformula} 
	\end{equation}
	is negative, which is a sufficient condition for the stability of the considered system in the considered range of initial values.
	
	For the presented example on how the above estimates can detect stability, we set $\alpha=0.4$, which results in $\omglob=0.3$, and we set $X_0\subset \mathbb R^{2} $ to be the closed ball around the origin of radius $r=0.25$. The grid for $X_0$ uses $12$ equally distributed points on the circle with radius $r=0.25$, another $8$ points on the circle with $r=0.17$, and $4$ points at $r=0.08$.
 
	From the computed trajectories we compute $K(t)$ (Fig. \ref{fig:stabtrejecs}(a)), $m_t$ (Fig. \ref{fig:stabtrejecs}(b)) with the manually optimized $\rho:=0.55t$, and, defining $K:=\max_{0\leq t\leq \tstar}K(t)$ taken over all trajectories, evaluate $-\omstar$ as in \eqref{eq:omstar_oneformula} (Fig. \ref{fig:stabtrejecs}(c)). Since for $\tstar \approx 6.0$, the value of $-\omstar$ becomes negative, the sufficient conditions for stability as described in Theorem \ref{thm:locconexpdec} and Lemma \ref{lem:meanvldconsts} are fulfilled. Obviously, the computed trajectories approach zero as $t\to\infty$ (Fig. \ref{fig:stabtrejecs}(d)).

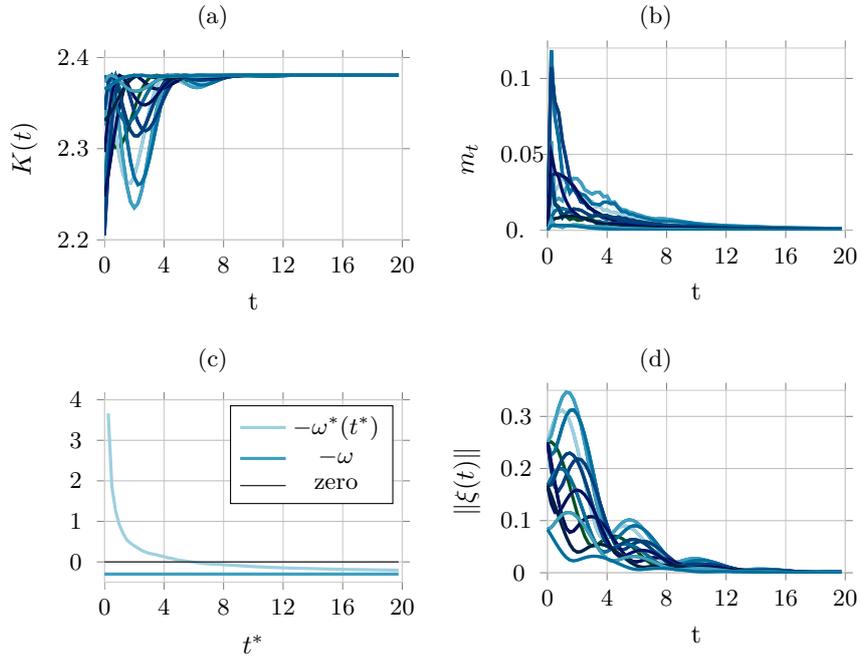
\begin{figure}[tbp]
  \pgfplotsset{footnotesize}
	\pgfplotsset{yticklabel style={text width=1em ,align=right}}
	\setlength\figureheight{4cm}
	\setlength\figurewidth{5.5cm}
	\subfiguretopcaptrue
	\begin{center}
		\subfigure[]{\input{pics/ks-alpha0.4-rho0.55-xzrad0.25.tikz}}
		\subfigure[]{\input{pics/mtlls-alpha0.4-rho0.55-xzrad0.25.tikz}}
		\subfigure[]{\input{pics/oms-alpha0.4-rho0.55-xzrad0.25.tikz}}
		\subfigure[]{\input{pics/trjnrms-alpha0.4-rho0.55-xzrad0.25.tikz}}
	\end{center}
	\caption{Computed bounds $K$ for the transient behavior (a), the estimate $m_t$ (b) and the resulting decay rates $\omega^*$ (c), and the norm of the trajectories over time and for various initial data in $X_0$ (d).}
	\label{fig:stabtrejecs}
\end{figure}

Note that for a larger $X_0$, some trajectories are not stable and also $-\omstar$ does not become negative, see Fig. \ref{fig:notstabtrajecs}.

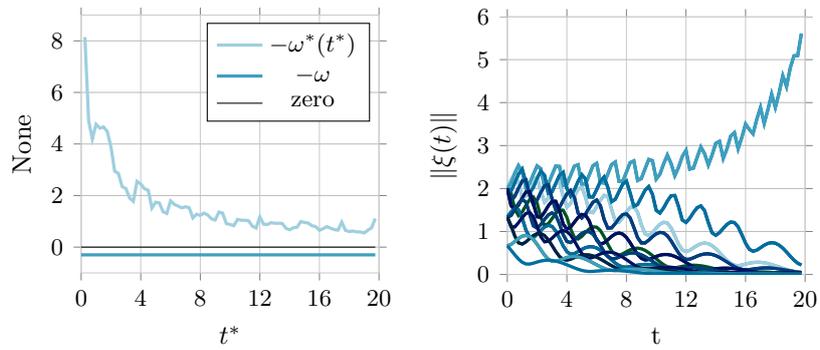
\begin{figure}[tbp]
  \pgfplotsset{footnotesize}
	\setlength\figureheight{5cm}
	\setlength\figurewidth{5.5cm}
	\begin{center}
		\input{pics/oms-alpha0.4-rho0.55-xzrad2.0.tikz}
		\input{pics/trjnrms-alpha0.4-rho0.55-xzrad2.0.tikz}
	\end{center}
	\caption{Estimate of $\omega^*$ and the trajectories for a set of initial values that are not uniformly stable.}
	\label{fig:notstabtrajecs}
\end{figure}

\end{exa}

%% file: pics/ks-alpha0.4-rho0.55-xzrad0.25.tikz
%
%
%
\begin{tikzpicture}

\definecolor{color1}{rgb}{0.247058823529412,0.623529411764706,0.749019607843137}
\definecolor{color0}{rgb}{0.623529411764706,0.811764705882353,0.874509803921569}
\definecolor{color3}{rgb}{0,0.247058823529412,0.498039215686275}
\definecolor{color2}{rgb}{0,0.435294117647059,0.623529411764706}
\definecolor{color5}{rgb}{0,0.129411764705882,0.247058823529412}
\definecolor{color4}{rgb}{0,0.0588235294117647,0.372549019607843}
\definecolor{color6}{rgb}{0,0.317647058823529,0.12156862745098}

\begin{axis}[
xlabel={t},
ylabel={$K(t)$},
xtick={0, 4, 8, 12, 16, 20},
ytick={2.2, 2.3, 2.4},
xmin=0, xmax=20,
ymin=2.2, ymax=2.4,
width=\figurewidth,
height=\figureheight,
tick align=outside,
xmajorgrids,
x grid style={white!80.0!black},
ymajorgrids,
y grid style={white!80.0!black},
axis line style={white!80.0!black}
]
\addplot [very thick, color0]
table {%
0 2.33551218389232
0.25 2.31857637716003
0.5 2.30673557787714
0.75 2.30118223571108
1 2.30190280302505
1.25 2.30800779772438
1.5 2.31804556404166
1.75 2.33030444352876
2 2.3430935930733
2.25 2.35497650402835
2.5 2.36492764856427
2.75 2.37239498092945
3 2.37727254866175
3.25 2.37980730776622
3.5 2.38047378725736
3.75 2.37984803166921
4 2.37850248662629
4.25 2.37693200582007
4.5 2.37551199668644
4.75 2.37448424430475
5 2.37396365729732
5.25 2.37395888261345
5.5 2.37440041903478
5.75 2.37517091993197
6 2.37613354711861
6.25 2.37715545613311
6.5 2.37812472001507
6.75 2.37896018100279
7 2.37961471133778
7.25 2.38007311943634
7.5 2.38034635669166
7.75 2.38046373991724
8 2.38046475130231
8.25 2.38039161301826
8.5 2.38028340946611
8.75 2.38017211713094
9 2.38008055521333
9.25 2.38002201766183
9.5 2.38000119447124
9.75 2.38001593316718
10 2.38005939859517
10.25 2.38012225645337
10.5 2.38019459649792
10.75 2.38026742151822
11 2.38033362396938
11.25 2.38038845952297
11.5 2.3804295880721
11.75 2.38045678951613
12 2.38047147622532
12.25 2.38047611785324
12.5 2.38047367589763
12.75 2.38046711840461
13 2.3804590567192
13.25 2.38045151960183
13.5 2.38044585827602
13.75 2.38044276088724
14 2.38044234648446
14.25 2.38044430592064
14.5 2.38044805997245
14.75 2.38045290991328
15 2.38045816346095
15.25 2.38046322651092
15.5 2.38046765782812
15.75 2.38047118947248
16 2.38047371922262
16.25 2.38047528333671
16.5 2.38047601829081
16.75 2.38047611927925
17 2.38047580168305
17.25 2.38047526971704
17.5 2.38047469443049
17.75 2.38047420148066
18 2.38047386773578
18.25 2.38047372487429
18.5 2.38047376772435
18.75 2.38047396506265
19 2.38047427086232
19.25 2.38047463446767
19.5 2.38047500864342
19.75 2.38047535505385
};
\addplot [very thick, color1]
table {%
0 2.38034635244142
0.25 2.3713741627967
0.5 2.34708760892509
0.75 2.3179462414805
1 2.29148457084162
1.25 2.27225001858756
1.5 2.26234450953926
1.75 2.26201375295101
2 2.2701411534241
2.25 2.28467641042807
2.5 2.30305429783135
2.75 2.32262031607612
3 2.34102785797428
3.25 2.35654103874016
3.5 2.36818633333225
3.75 2.37573827513722
4 2.3795734442371
4.25 2.38045625015695
4.5 2.37932013884005
4.75 2.37708773068153
5 2.3745484135431
5.25 2.37229290652601
5.5 2.37069436130912
5.75 2.36992270625018
6 2.36997988244323
6.25 2.37074575031742
6.5 2.37202664363455
6.75 2.37360042804269
7 2.37525368263351
7.25 2.37680837016353
7.5 2.37813714909093
7.75 2.37916806266724
8 2.37988061400069
8.25 2.38029592592934
8.5 2.38046385130339
8.75 2.38044958568282
9 2.38032170934161
9.25 2.38014287670984
9.5 2.37996367039954
9.75 2.37981957805287
10 2.37973065330916
10.25 2.37970320106625
10.5 2.3797327450015
10.75 2.37980757151164
11 2.37991225421418
11.25 2.38003072015788
11.5 2.38014858826505
11.75 2.38025467109474
12 2.38034166606396
12.25 2.3804061579503
12.5 2.38044811210389
12.75 2.38047005657672
13 2.38047614012493
13.25 2.38047122031109
13.5 2.38046009184325
13.75 2.38044691844504
14 2.38043488884769
14.25 2.38042608358729
14.5 2.38042151440306
14.75 2.38042128845245
15 2.38042484279853
15.25 2.38043120220804
15.5 2.38043922136369
15.75 2.38044778413434
16 2.38045594670192
16.25 2.38046302028893
16.5 2.38046859881509
16.75 2.38047254264443
17 2.38047493148196
17.25 2.38047600102051
17.5 2.38047607543218
17.75 2.38047550564111
18 2.38047461978423
18.25 2.38047368910873
18.5 2.38047290971681
18.75 2.38047239841331
19 2.38047219958177
19.25 2.38047229936692
19.5 2.38047264347885
19.75 2.38047315543412
};
\addplot [very thick, color2]
table {%
0 2.32728970865293
0.25 2.37288193583227
0.5 2.37866870149727
0.75 2.35742127271575
1 2.32354516528974
1.25 2.28853524042796
1.5 2.25988571774808
1.75 2.24162175464273
2 2.23514381138678
2.25 2.23992304994179
2.5 2.25403057725571
2.75 2.2746084668986
3 2.29836544446898
3.25 2.3221003519719
3.5 2.34317961128059
3.75 2.35986456669931
4 2.3714184584711
4.25 2.37799756017344
4.5 2.38039847905859
4.75 2.37975861075569
5 2.3772887684722
5.25 2.37407937190262
5.5 2.37098767143646
5.75 2.36859373635616
6 2.36720630767653
6.25 2.36690053121293
6.5 2.36757312356406
6.75 2.36900398686638
7 2.37091597634938
7.25 2.37302661791935
7.5 2.37508761204513
7.75 2.37691006037835
8 2.37837549406649
8.25 2.3794346378276
8.5 2.38009711514741
8.75 2.38041582072767
9 2.38046950386236
9.25 2.38034642231234
9.5 2.380130986333
9.75 2.37989436765979
10 2.37968923882578
10.25 2.37954821894852
10.5 2.37948524276432
10.75 2.37949891097654
11 2.37957688479829
11.25 2.37970050281565
11.5 2.37984898177104
11.75 2.38000277901611
12 2.38014590669034
12.25 2.38026717548001
12.5 2.38036048521092
12.75 2.380424373959
13 2.38046107647995
13.25 2.3804753422693
13.5 2.38047322872514
13.75 2.38046103234405
14 2.38044446027426
14.25 2.38042808688388
14.5 2.38041509162882
14.75 2.38040723935278
15 2.38040504317381
15.25 2.38040804221088
15.5 2.38041512931389
15.75 2.38042487428464
16 2.3804358027094
16.25 2.38044660648672
16.5 2.38045627714919
16.75 2.38046416515678
17 2.38046997617842
17.25 2.38047372285182
17.5 2.38047564835748
17.75 2.38047613983287
18 2.38047564509956
18.25 2.38047460276463
18.5 2.38047339116398
18.75 2.38047229789753
19 2.38047150852467
19.25 2.38047111098541
19.5 2.38047111114194
19.75 2.38047145455483
};
\addplot [very thick, color3]
table {%
0 2.24063705417903
0.25 2.32267527754159
0.5 2.36871847002652
0.75 2.38032563485102
1 2.36710632793315
1.25 2.34082098734247
1.5 2.3114815500321
1.75 2.28603286487461
2 2.26849465806624
2.25 2.2605497372264
2.5 2.26212252591515
2.75 2.27185793053671
3 2.28754627038093
3.25 2.30654650505169
3.5 2.32621394668342
3.75 2.34428761391864
4 2.35916866996824
4.25 2.37003871778293
4.5 2.37681349079564
4.75 2.37997502806715
5 2.38034913503918
5.25 2.37888959014405
5.5 2.37650780970498
5.75 2.37396190805514
6 2.37180171504615
6.25 2.37035803606209
6.5 2.36976266791226
6.75 2.36998708666635
7 2.37088999035752
7.25 2.37226602802251
7.5 2.3738899127154
7.75 2.37555185986491
8 2.37708209062392
8.25 2.37836389241768
8.5 2.37933629656445
8.75 2.37998859056883
9 2.38034948029942
9.25 2.3804737610237
9.5 2.38042895976745
9.75 2.38028375213192
10 2.38009921981689
10.25 2.37992333949796
10.5 2.37978855835619
10.75 2.37971195201038
11 2.37969726832657
11.25 2.37973810723726
11.5 2.37982154189108
11.75 2.37993161103586
12 2.38005227534668
12.25 2.3801696021666
12.5 2.38027310068927
12.75 2.3803562579396
13 2.38041641429332
13.25 2.38045416705525
13.5 2.38047250253459
13.75 2.38047583929113
14 2.3804691296553
14.25 2.38045712062793
14.5 2.3804438282252
14.75 2.3804322376798
15 2.38042421012421
15.25 2.38042055427783
15.5 2.38042121161652
15.75 2.38042550230497
16 2.38043238500584
16.25 2.3804406939289
16.5 2.38044932894638
16.75 2.38045738695286
17 2.38046423283363
17.25 2.38046951685595
17.5 2.38047315026388
17.75 2.38047525314606
18 2.38047608837039
18.25 2.38047599411309
18.5 2.38047532363323
18.75 2.38047439840448
19 2.38047347715893
19.25 2.38047274071402
19.5 2.38047229045608
19.75 2.38047215721574
};
\addplot [very thick, color4]
table {%
0 2.20510595510158
0.25 2.27413737791313
0.5 2.32852114087029
0.75 2.36333884404746
1 2.37882863786825
1.25 2.37894456692819
1.5 2.3692183666155
1.75 2.35502858611727
2 2.34066663169995
2.25 2.32905055971889
2.5 2.32179657581105
2.75 2.31943419067325
3 2.32165943454637
3.25 2.32758988754593
3.5 2.33601191844941
3.75 2.34561186074024
4 2.35517571289818
4.25 2.36373807862315
4.5 2.37066573943484
4.75 2.37567314013608
5 2.37878073118231
5.25 2.38023661034675
5.5 2.38042434791758
5.75 2.37977609993728
6 2.37870314554346
6.25 2.37754891835269
6.5 2.37656426178298
6.75 2.37590148024899
7 2.37562240272238
7.25 2.37571549017762
7.5 2.37611745775608
7.75 2.37673563171292
8 2.37746815541694
8.25 2.37822010851509
8.5 2.37891455358708
8.75 2.37949837005595
9 2.37994344946129
9.25 2.38024427533094
9.5 2.38041313553833
9.75 2.38047421821899
10 2.38045766187907
10.25 2.3803943533215
10.5 2.38031195763742
10.75 2.38023237312034
11 2.38017056741716
11.25 2.38013458467222
11.5 2.38012641904985
11.75 2.38014342058987
12 2.38017991601455
12.25 2.38022878598915
12.5 2.38028280907993
12.75 2.38033566297364
13 2.38038254447326
13.25 2.38042042833336
13.5 2.38044802631885
13.75 2.38046552973551
14 2.38047422531709
14.25 2.38047606696956
14.5 2.38047327033155
14.75 2.3804679764023
15 2.38046201002265
15.25 2.38045673945442
15.5 2.38045302932128
15.75 2.38045126882823
16 2.38045145239576
16.25 2.38045328886346
16.5 2.38045631809421
16.75 2.38046001823806
17 2.3804638924562
17.25 2.38046752942096
17.5 2.38047063671949
17.75 2.38047304997428
18 2.38047472286139
18.25 2.38047570431009
18.5 2.38047610915141
18.75 2.38047608772098
19 2.38047579860525
19.25 2.38047538722705
19.5 2.38047497152698
19.75 2.38047463474981
};
\addplot [very thick, color5]
table {%
0 2.24771101760347
0.25 2.27162688128558
0.5 2.2980962820727
0.75 2.32367615191632
1 2.34566305472785
1.25 2.36241540566655
1.5 2.37341275585551
1.75 2.37907761802181
2 2.38045857939022
2.25 2.37888922406733
2.5 2.3757049709362
2.75 2.37205269288373
3 2.36879153714045
3.25 2.36646559104901
3.5 2.36532551919079
3.75 2.36537959815041
4 2.36645929714479
4.25 2.36828844472056
4.5 2.37054774490374
4.75 2.37292857602878
5 2.37517214214676
5.25 2.37709240245995
5.5 2.37858353902303
5.75 2.37961466429665
6 2.38021567370314
6.25 2.3804584839927
6.5 2.38043746937695
6.75 2.38025199459174
7 2.37999283958642
7.25 2.3797332670548
7.5 2.37952464889197
7.75 2.3793959925252
8 2.37935639562993
8.25 2.37939934981541
8.5 2.3795078727641
8.75 2.37965961021997
9 2.37983127499081
9.25 2.38000203389087
9.5 2.38015568327426
9.75 2.38028165098394
10 2.38037499940398
10.25 2.38043568991972
10.5 2.38046739612198
10.75 2.38047613736618
11 2.38046895656154
11.25 2.38045280195349
11.5 2.38043370450403
11.75 2.38041628038101
12 2.38040353901529
12.25 2.38039694121051
12.5 2.38039663769291
12.75 2.38040180933489
13 2.3804110408432
13.25 2.38042267159257
13.5 2.38043508430477
13.75 2.38044691209276
14 2.38045715808516
14.25 2.38046523527678
14.5 2.38047094264618
14.75 2.38047439686992
15 2.38047594032195
15.25 2.38047604327361
15.5 2.38047521443575
15.75 2.38047392923648
16 2.38047258047872
16.25 2.38047145196524
16.5 2.38047071255473
16.75 2.38047042620705
17 2.38047057253065
17.25 2.38047107263866
17.5 2.38047181557702
17.75 2.38047268183154
18 2.38047356171812
18.25 2.38047436762267
18.5 2.38047504012693
18.75 2.38047554887286
19 2.38047588937993
19.25 2.38047607733129
19.5 2.38047614168782
19.75 2.38047611784481
};
\addplot [very thick, color6]
table {%
0 2.33551218389232
0.25 2.31857637716003
0.5 2.30673557787713
0.75 2.30118223571106
1 2.30190280302502
1.25 2.30800779772435
1.5 2.31804556404163
1.75 2.33030444352873
2 2.34309359307328
2.25 2.35497650402833
2.5 2.36492764856426
2.75 2.37239498092944
3 2.37727254866175
3.25 2.37980730776622
3.5 2.38047378725736
3.75 2.37984803166921
4 2.37850248662629
4.25 2.37693200582006
4.5 2.37551199668644
4.75 2.37448424430474
5 2.37396365729732
5.25 2.37395888261345
5.5 2.37440041903478
5.75 2.37517091993198
6 2.37613354711863
6.25 2.37715545613313
6.5 2.37812472001509
6.75 2.37896018100281
7 2.3796147113378
7.25 2.38007311943635
7.5 2.38034635669166
7.75 2.38046373991724
8 2.38046475130231
8.25 2.38039161301826
8.5 2.38028340946611
8.75 2.38017211713093
9 2.38008055521333
9.25 2.38002201766183
9.5 2.38000119447124
9.75 2.38001593316718
10 2.38005939859517
10.25 2.38012225645337
10.5 2.38019459649793
10.75 2.38026742151822
11 2.38033362396938
11.25 2.38038845952298
11.5 2.3804295880721
11.75 2.38045678951613
12 2.38047147622532
12.25 2.38047611785324
12.5 2.38047367589763
12.75 2.38046711840461
13 2.3804590567192
13.25 2.38045151960183
13.5 2.38044585827602
13.75 2.38044276088724
14 2.38044234648446
14.25 2.38044430592063
14.5 2.38044805997245
14.75 2.38045290991328
15 2.38045816346095
15.25 2.38046322651092
15.5 2.38046765782812
15.75 2.38047118947248
16 2.38047371922262
16.25 2.38047528333671
16.5 2.38047601829081
16.75 2.38047611927925
17 2.38047580168305
17.25 2.38047526971704
17.5 2.38047469443049
17.75 2.38047420148066
18 2.38047386773578
18.25 2.38047372487429
18.5 2.38047376772435
18.75 2.38047396506265
19 2.38047427086231
19.25 2.38047463446767
19.5 2.38047500864341
19.75 2.38047535505385
};
\addplot [very thick, color0]
table {%
0 2.38034635244142
0.25 2.3713741627967
0.5 2.34708760892509
0.75 2.3179462414805
1 2.29148457084162
1.25 2.27225001858757
1.5 2.26234450953926
1.75 2.26201375295101
2 2.2701411534241
2.25 2.28467641042807
2.5 2.30305429783135
2.75 2.32262031607612
3 2.34102785797428
3.25 2.35654103874016
3.5 2.36818633333225
3.75 2.37573827513722
4 2.3795734442371
4.25 2.38045625015695
4.5 2.37932013884005
4.75 2.37708773068153
5 2.3745484135431
5.25 2.37229290652601
5.5 2.37069436130912
5.75 2.36992270625018
6 2.36997988244323
6.25 2.37074575031742
6.5 2.37202664363455
6.75 2.37360042804269
7 2.37525368263351
7.25 2.37680837016353
7.5 2.37813714909093
7.75 2.37916806266724
8 2.37988061400069
8.25 2.38029592592934
8.5 2.38046385130339
8.75 2.38044958568282
9 2.38032170934161
9.25 2.38014287670984
9.5 2.37996367039954
9.75 2.37981957805287
10 2.37973065330916
10.25 2.37970320106625
10.5 2.3797327450015
10.75 2.37980757151164
11 2.37991225421418
11.25 2.38003072015788
11.5 2.38014858826505
11.75 2.38025467109474
12 2.38034166606396
12.25 2.3804061579503
12.5 2.38044811210389
12.75 2.38047005657672
13 2.38047614012493
13.25 2.38047122031109
13.5 2.38046009184325
13.75 2.38044691844504
14 2.38043488884769
14.25 2.38042608358729
14.5 2.38042151440306
14.75 2.38042128845245
15 2.38042484279853
15.25 2.38043120220804
15.5 2.38043922136369
15.75 2.38044778413434
16 2.38045594670192
16.25 2.38046302028893
16.5 2.38046859881509
16.75 2.38047254264443
17 2.38047493148196
17.25 2.38047600102051
17.5 2.38047607543218
17.75 2.38047550564111
18 2.38047461978423
18.25 2.38047368910873
18.5 2.38047290971681
18.75 2.38047239841331
19 2.38047219958177
19.25 2.38047229936692
19.5 2.38047264347885
19.75 2.38047315543412
};
\addplot [very thick, color1]
table {%
0 2.32728970865293
0.25 2.37288193583226
0.5 2.37866870149727
0.75 2.35742127271573
1 2.3235451652897
1.25 2.28853524042789
1.5 2.25988571774799
1.75 2.24162175464262
2 2.23514381138667
2.25 2.23992304994169
2.5 2.25403057725562
2.75 2.27460846689852
3 2.29836544446891
3.25 2.32210035197185
3.5 2.34317961128056
3.75 2.35986456669929
4 2.37141845847109
4.25 2.37799756017344
4.5 2.38039847905859
4.75 2.37975861075569
5 2.37728876847219
5.25 2.37407937190261
5.5 2.37098767143645
5.75 2.36859373635614
6 2.36720630767652
6.25 2.36690053121291
6.5 2.36757312356405
6.75 2.36900398686637
7 2.37091597634937
7.25 2.37302661791934
7.5 2.37508761204513
7.75 2.37691006037834
8 2.37837549406649
8.25 2.3794346378276
8.5 2.38009711514741
8.75 2.38041582072767
9 2.38046950386236
9.25 2.38034642231234
9.5 2.380130986333
9.75 2.3798943676598
10 2.37968923882579
10.25 2.37954821894852
10.5 2.37948524276433
10.75 2.37949891097654
11 2.3795768847983
11.25 2.37970050281565
11.5 2.37984898177104
11.75 2.38000277901611
12 2.38014590669034
12.25 2.38026717548001
12.5 2.38036048521092
12.75 2.380424373959
13 2.38046107647995
13.25 2.3804753422693
13.5 2.38047322872514
13.75 2.38046103234405
14 2.38044446027426
14.25 2.38042808688388
14.5 2.38041509162882
14.75 2.38040723935278
15 2.38040504317381
15.25 2.38040804221088
15.5 2.38041512931389
15.75 2.38042487428464
16 2.3804358027094
16.25 2.38044660648672
16.5 2.38045627714919
16.75 2.38046416515678
17 2.38046997617842
17.25 2.38047372285182
17.5 2.38047564835748
17.75 2.38047613983287
18 2.38047564509956
18.25 2.38047460276463
18.5 2.38047339116398
18.75 2.38047229789753
19 2.38047150852467
19.25 2.38047111098541
19.5 2.38047111114194
19.75 2.38047145455483
};
\addplot [very thick, color2]
table {%
0 2.24063705417903
0.25 2.32267527754153
0.5 2.36871847002651
0.75 2.38032563485102
1 2.36710632793313
1.25 2.34082098734243
1.5 2.31148155003202
1.75 2.2860328648745
2 2.26849465806611
2.25 2.26054973722626
2.5 2.26212252591502
2.75 2.27185793053659
3 2.28754627038083
3.25 2.30654650505162
3.5 2.32621394668337
3.75 2.3442876139186
4 2.35916866996823
4.25 2.37003871778293
4.5 2.37681349079564
4.75 2.37997502806715
5 2.38034913503918
5.25 2.37888959014405
5.5 2.37650780970497
5.75 2.37396190805513
6 2.37180171504614
6.25 2.37035803606208
6.5 2.36976266791225
6.75 2.36998708666634
7 2.3708899903575
7.25 2.37226602802243
7.5 2.37388991271526
7.75 2.37555185986481
8 2.37708209062388
8.25 2.3783638924177
8.5 2.37933629656452
8.75 2.37998859056891
9 2.38034948029946
9.25 2.38047376102371
9.5 2.38042895976742
9.75 2.38028375213185
10 2.3800992198168
10.25 2.37992333949786
10.5 2.37978855835609
10.75 2.37971195201029
11 2.3796972683265
11.25 2.37973810723721
11.5 2.37982154189104
11.75 2.37993161103584
12 2.38005227534667
12.25 2.38016960216661
12.5 2.38027310068931
12.75 2.3803562579396
13 2.38041641429326
13.25 2.38045416705524
13.5 2.38047250253458
13.75 2.38047583929113
14 2.38046912965532
14.25 2.38045712062795
14.5 2.38044382822513
14.75 2.38043223767978
15 2.38042421012419
15.25 2.38042055427781
15.5 2.38042121161648
15.75 2.38042550230491
16 2.38043238500569
16.25 2.38044069392881
16.5 2.38044932894634
16.75 2.38045738695287
17 2.38046423283367
17.25 2.38046951685598
17.5 2.38047315026393
17.75 2.38047525314606
18 2.38047608837039
18.25 2.38047599411309
18.5 2.38047532363324
18.75 2.3804743984045
19 2.38047347715896
19.25 2.38047274071406
19.5 2.38047229045611
19.75 2.38047215721575
};
\addplot [very thick, color3]
table {%
0 2.20510595510158
0.25 2.27413737791321
0.5 2.32852114087032
0.75 2.36333884404747
1 2.37882863786825
1.25 2.37894456692819
1.5 2.36921836661552
1.75 2.3550285861173
2 2.34066663169998
2.25 2.32905055971893
2.5 2.32179657581111
2.75 2.3194341906733
3 2.32165943454641
3.25 2.32758988754597
3.5 2.33601191844945
3.75 2.34561186074027
4 2.3551757128982
4.25 2.36373807862317
4.5 2.37066573943485
4.75 2.3756731401361
5 2.37878073118231
5.25 2.38023661034675
5.5 2.38042434791758
5.75 2.37977609993728
6 2.37870314554346
6.25 2.37754891835269
6.5 2.37656426178298
6.75 2.37590148024899
7 2.37562240272238
7.25 2.37571549017761
7.5 2.37611745775604
7.75 2.37673563171283
8 2.37746815541683
8.25 2.37822010851498
8.5 2.37891455358692
8.75 2.37949837005585
9 2.37994344946123
9.25 2.38024427533091
9.5 2.38041313553831
9.75 2.38047421821899
10 2.38045766187908
10.25 2.38039435332154
10.5 2.38031195763748
10.75 2.38023237312042
11 2.38017056741726
11.25 2.38013458467233
11.5 2.3801264190499
11.75 2.3801434205899
12 2.38017991601455
12.25 2.38022878598915
12.5 2.38028280907992
12.75 2.38033566297363
13 2.38038254447324
13.25 2.38042042833334
13.5 2.38044802631883
13.75 2.3804655297355
14 2.38047422531709
14.25 2.38047606696956
14.5 2.38047327033156
14.75 2.38046797640229
15 2.38046201002265
15.25 2.38045673945444
15.5 2.38045302932131
15.75 2.38045126882826
16 2.3804514523958
16.25 2.38045328886349
16.5 2.38045631809424
16.75 2.38046001823808
17 2.38046389245621
17.25 2.38046752942097
17.5 2.38047063671949
17.75 2.38047304997428
18 2.38047472286138
18.25 2.38047570431009
18.5 2.38047610915141
18.75 2.38047608772098
19 2.38047579860525
19.25 2.38047538722706
19.5 2.38047497152698
19.75 2.38047463474981
};
\addplot [very thick, color4]
table {%
0 2.24771101760347
0.25 2.27162688128555
0.5 2.29809628207266
0.75 2.3236761519163
1 2.34566305472784
1.25 2.36241540566654
1.5 2.37341275585551
1.75 2.37907761802181
2 2.38045857939022
2.25 2.37888922406733
2.5 2.3757049709362
2.75 2.37205269288372
3 2.36879153714045
3.25 2.366465591049
3.5 2.36532551919078
3.75 2.36537959815041
4 2.36645929714478
4.25 2.36828844472056
4.5 2.37054774490374
4.75 2.37292857602877
5 2.37517214214676
5.25 2.37709240245995
5.5 2.37858353902303
5.75 2.37961466429665
6 2.38021567370314
6.25 2.3804584839927
6.5 2.38043746937695
6.75 2.38025199459174
7 2.37999283958642
7.25 2.3797332670548
7.5 2.37952464889197
7.75 2.3793959925252
8 2.37935639562993
8.25 2.37939934981541
8.5 2.3795078727641
8.75 2.37965961021997
9 2.3798312749908
9.25 2.38000203389087
9.5 2.38015568327426
9.75 2.38028165098394
10 2.38037499940398
10.25 2.38043568991972
10.5 2.38046739612198
10.75 2.38047613736618
11 2.38046895656154
11.25 2.38045280195349
11.5 2.38043370450403
11.75 2.380416280381
12 2.38040353901529
12.25 2.38039694121051
12.5 2.38039663769291
12.75 2.38040180933488
13 2.3804110408432
13.25 2.38042267159257
13.5 2.38043508430477
13.75 2.38044691209276
14 2.38045715808516
14.25 2.38046523527678
14.5 2.38047094264618
14.75 2.38047439686992
15 2.38047594032195
15.25 2.38047604327361
15.5 2.38047521443575
15.75 2.38047392923648
16 2.38047258047872
16.25 2.38047145196524
16.5 2.38047071255473
16.75 2.38047042620705
17 2.38047057253065
17.25 2.38047107263866
17.5 2.38047181557702
17.75 2.38047268183154
18 2.38047356171812
18.25 2.38047436762267
18.5 2.38047504012693
18.75 2.38047554887286
19 2.38047588937993
19.25 2.38047607733129
19.5 2.38047614168782
19.75 2.38047611784481
};
\addplot [very thick, color5]
table {%
0 2.38002966343392
0.25 2.37324518694436
0.5 2.36179322043032
0.75 2.34948969541484
1 2.33904092780951
1.25 2.3320422774569
1.5 2.32912899299176
1.75 2.33018149275215
2 2.33454336528168
2.25 2.34123617862328
2.5 2.34916086688442
2.75 2.35727234878085
3 2.36471228819642
3.25 2.37088839239625
3.5 2.37549746498682
3.75 2.37849982055951
4 2.38006044541659
4.25 2.38047503142952
4.5 2.38009691116373
4.75 2.37927594075413
5 2.37831483112758
5.25 2.377443849422
5.5 2.37681182991214
5.75 2.37648990904541
6 2.37648393169563
6.25 2.3767516505337
6.5 2.37722136940241
6.75 2.37780939453269
7 2.37843445101225
7.25 2.37902803191146
7.5 2.37954039709058
7.75 2.37994255390236
8 2.38022498977372
8.25 2.38039416040084
8.5 2.38046777788525
8.75 2.38046983162633
9 2.38042606506054
9.25 2.38036038021142
9.5 2.38029239344481
9.75 2.38023615929215
10 2.38019992153097
10.25 2.38018665809127
10.5 2.38019514501644
10.75 2.38022126968069
11 2.3802593646097
11.25 2.38030338553951
11.5 2.38034782647409
11.75 2.38038832234518
12 2.3804219446763
12.25 2.38044723210622
12.5 2.38046402125122
12.75 2.38047315182184
13 2.38047611682252
13.25 2.38047471724536
13.5 2.38047076480159
13.75 2.38046585856164
14 2.38046124502748
14.25 2.38045775846151
14.5 2.38045582813214
14.75 2.38045553456695
15 2.38045669496374
15.25 2.38045895944718
15.5 2.38046190307862
15.75 2.38046510306609
16 2.38046819529732
16.25 2.38047090823807
16.5 2.38047307584665
16.75 2.38047463339358
17 2.38047560105317
17.25 2.38047606071498
17.5 2.38047613069661
17.75 2.38047594218706
18 2.38047562000717
18.25 2.38047526903171
18.5 2.38047496660209
18.75 2.38047476032422
19 2.38047467016574
19.25 2.38047469348843
19.5 2.38047481159542
19.75 2.38047499659626
};
\addplot [very thick, color6]
table {%
0 2.34264166244291
0.25 2.37140678437991
0.5 2.38045287844683
0.75 2.37510580039302
1 2.36158805455235
1.25 2.34540305917668
1.5 2.33063000296312
1.75 2.3198252188749
2 2.31420435611069
2.25 2.31391248665408
2.5 2.31830515800433
2.75 2.32622198779718
3 2.33624910236615
3.25 2.34696024441822
3.5 2.35711670643023
3.75 2.36580428542964
4 2.37249431726217
4.25 2.37703148899465
4.5 2.37956600046701
4.75 2.38045563224231
5 2.38016276313237
5.25 2.37916481827922
5.5 2.37788807572407
5.75 2.37666731926394
6 2.37572890386097
6.25 2.37519233486035
6.5 2.37508473252822
6.75 2.37536281182329
7 2.37593773388996
7.25 2.37669912072655
7.5 2.37753554144211
7.75 2.37834982502448
8 2.37906856613229
8.25 2.37964605723068
8.5 2.38006354560017
8.75 2.38032511046502
9 2.38045158059609
9.25 2.38047380662376
9.5 2.3804263468847
9.75 2.38034228884672
10 2.38024958166572
10.25 2.38016896031536
10.5 2.38011331282912
10.75 2.38008819606301
11 2.38009313674224
11.25 2.38012335017903
11.5 2.38017155073271
11.75 2.38022960218272
12 2.38028983975094
12.25 2.38034598151316
12.5 2.38039361946968
12.75 2.38043033735005
13 2.38045553732529
13.25 2.3804700740376
13.5 2.38047579324805
13.75 2.38047505979374
14 2.38047033853956
14.25 2.38046386855783
14.5 2.38045744863375
14.75 2.38045233282899
15 2.38044922123003
15.25 2.38044832257964
15.5 2.38044946236991
15.75 2.38045221086185
16 2.38045600976291
16.25 2.38046028173393
16.5 2.3804645132852
16.75 2.3804683074686
17 2.38047140754294
17.25 2.38047369604139
17.5 2.38047517596648
17.75 2.38047594110528
18 2.38047614216231
18.25 2.38047595414604
18.5 2.38047554890118
18.75 2.38047507499276
19 2.38047464565788
19.25 2.38047433423084
19.5 2.38047417577632
19.75 2.38047417310101
};
\addplot [very thick, color0]
table {%
0 2.2962945105133
0.25 2.33368449831736
0.5 2.3601303143337
0.75 2.37518093993297
1 2.38038627996526
1.25 2.3784149154942
1.5 2.37221392465263
1.75 2.36442411889999
2 2.35707834785233
2.25 2.35151369148274
2.5 2.34841279344275
2.75 2.3479111814435
3 2.34973299947542
3.25 2.35333434722198
3.5 2.35804095616744
3.75 2.36316929337725
4 2.36812157400794
4.25 2.37244796059974
4.5 2.37587399379577
4.75 2.3782966797444
5 2.37975707198631
5.25 2.38039933319177
5.5 2.38042598785049
5.75 2.38005700044755
6 2.37949746838141
6.25 2.37891596087018
6.5 2.37843338694113
6.75 2.37812088603736
7 2.37800451944545
7.25 2.37807434265841
7.5 2.3782955770725
7.75 2.37861995798967
8 2.37899580623792
8.25 2.37937589271565
8.5 2.37972266519136
8.75 2.38001084940632
9 2.38022776101852
9.25 2.38037188311712
9.5 2.3804503432622
9.75 2.38047590655154
10 2.38046400304253
10.25 2.38043016773856
10.5 2.3803881178657
10.75 2.38034855025045
11 2.38031862734551
11.25 2.38030203854921
11.5 2.38029948202154
11.75 2.38030939636332
12 2.38032878537389
12.25 2.38035400734123
12.5 2.38038143707218
12.75 2.38040795098711
13 2.38043121850444
13.25 2.38044981351013
13.5 2.38046317865112
13.75 2.38047148523284
14 2.38047543437012
14.25 2.38047604007535
14.5 2.38047442715035
14.75 2.3804716660526
15 2.38046865648832
15.25 2.38046606212368
15.5 2.38046429165551
15.75 2.38046351673876
16 2.38046371503782
16.25 2.38046472643101
16.5 2.38046631187131
16.75 2.38046820671429
17 2.38047016318469
17.25 2.3804719793996
17.5 2.38047351476551
17.75 2.38047469337004
18 2.38047549793243
18.25 2.38047595781929
18.5 2.38047613396546
18.75 2.38047610358769
19 2.38047594666669
19.25 2.3804757355067
19.5 2.38047552790493
19.75 2.38047536389831
};
\addplot [very thick, color1]
table {%
0 2.33082091461279
0.25 2.33511797034525
0.5 2.34171796932878
0.75 2.3495363736907
1 2.35754186340981
1.25 2.36488723023741
1.5 2.37098754114569
1.75 2.37554286409626
2 2.37851313473379
2.25 2.38006026073652
2.5 2.38047530265634
2.75 2.38010648222501
3 2.37929889973264
3.25 2.37835139747024
3.5 2.37749151307816
3.75 2.3768665175318
4 2.37654702108096
4.25 2.37653915927779
4.5 2.37680153069657
4.75 2.37726358037809
5 2.37784282086586
5.25 2.37845906901362
5.5 2.37904467859362
5.75 2.37955048037155
6 2.37994775717674
6.25 2.38022700997659
6.5 2.38039451153211
6.75 2.38046766448252
7 2.38047009209094
7.25 2.38042717468859
7.5 2.3803624977953
7.75 2.38029543495436
8 2.3802398816946
8.25 2.38020400400946
8.5 2.38019077073641
8.75 2.3801989994732
9 2.38022465125859
9.25 2.38026214522266
9.5 2.38030552193205
9.75 2.38034934718234
10 2.3803893085069
10.25 2.380422508645
10.5 2.38044749721236
10.75 2.38046410512084
11 2.38047315446596
11.25 2.38047611384147
11.5 2.38047475772339
11.75 2.38047087296461
12 2.38046603803727
12.25 2.38046148450439
12.5 2.38045803767609
12.75 2.38045612337286
13 2.38045582310833
13.25 2.38045695815135
13.5 2.38045918438619
13.75 2.38046208311563
14 2.38046523733277
14.25 2.38046828754069
14.5 2.38047096533163
14.75 2.38047310638903
15 2.38047464613476
15.25 2.38047560395007
15.5 2.38047606021316
15.75 2.38047613141281
16 2.38047594688008
16.25 2.3804756297647
16.5 2.3804752836403
16.75 2.38047498495626
17 2.38047478085021
17.25 2.38047469118199
17.5 2.38047471347281
17.75 2.38047482937811
18 2.38047501143521
18.25 2.38047522921445
18.5 2.38047545420892
18.75 2.38047566317158
19 2.38047583990086
19.25 2.38047597562964
19.5 2.38047606834475
19.75 2.38047612140131
};
\addplot [very thick, color2]
table {%
0 2.38002966343392
0.25 2.37324518694436
0.5 2.36179322043032
0.75 2.34948969541482
1 2.33904092780949
1.25 2.33204227745687
1.5 2.32912899299173
1.75 2.33018149275212
2 2.33454336528166
2.25 2.34123617862326
2.5 2.3491608668844
2.75 2.35727234878083
3 2.36471228819642
3.25 2.37088839239624
3.5 2.37549746498681
3.75 2.3784998205595
4 2.38006044541659
4.25 2.38047503142952
4.5 2.38009691116373
4.75 2.37927594075413
5 2.37831483112758
5.25 2.377443849422
5.5 2.37681182991214
5.75 2.37648990904541
6 2.37648393169563
6.25 2.3767516505337
6.5 2.37722136940241
6.75 2.37780939453269
7 2.37843445101225
7.25 2.37902803191145
7.5 2.37954039709058
7.75 2.37994255390236
8 2.38022498977372
8.25 2.38039416040084
8.5 2.38046777788525
8.75 2.38046983162633
9 2.38042606506054
9.25 2.38036038021142
9.5 2.38029239344481
9.75 2.38023615929215
10 2.38019992153097
10.25 2.38018665809127
10.5 2.38019514501644
10.75 2.38022126968069
11 2.3802593646097
11.25 2.38030338553951
11.5 2.38034782647409
11.75 2.38038832234518
12 2.3804219446763
12.25 2.38044723210622
12.5 2.38046402125122
12.75 2.38047315182184
13 2.38047611682252
13.25 2.38047471724536
13.5 2.38047076480159
13.75 2.38046585856164
14 2.38046124502748
14.25 2.38045775846151
14.5 2.38045582813214
14.75 2.38045553456695
15 2.38045669496374
15.25 2.38045895944718
15.5 2.38046190307862
15.75 2.38046510306609
16 2.38046819529732
16.25 2.38047090823807
16.5 2.38047307584665
16.75 2.38047463339358
17 2.38047560105317
17.25 2.38047606071498
17.5 2.38047613069661
17.75 2.38047594218706
18 2.38047562000717
18.25 2.38047526903171
18.5 2.38047496660209
18.75 2.38047476032422
19 2.38047467016574
19.25 2.38047469348843
19.5 2.38047481159542
19.75 2.38047499659626
};
\addplot [very thick, color3]
table {%
0 2.34264166244291
0.25 2.37140678437987
0.5 2.38045287844683
0.75 2.37510580039302
1 2.36158805455233
1.25 2.34540305917663
1.5 2.33063000296304
1.75 2.3198252188748
2 2.31420435611057
2.25 2.31391248665396
2.5 2.3183051580042
2.75 2.32622198779706
3 2.33624910236605
3.25 2.34696024441814
3.5 2.35711670643017
3.75 2.36580428542961
4 2.37249431726216
4.25 2.37703148899464
4.5 2.37956600046701
4.75 2.38045563224231
5 2.38016276313236
5.25 2.37916481827921
5.5 2.37788807572406
5.75 2.37666731926393
6 2.37572890386095
6.25 2.37519233486033
6.5 2.3750847325282
6.75 2.37536281182327
7 2.37593773388992
7.25 2.3766991207265
7.5 2.37753554144205
7.75 2.37834982502443
8 2.37906856613225
8.25 2.37964605723064
8.5 2.38006354560013
8.75 2.38032511046501
9 2.38045158059609
9.25 2.38047380662376
9.5 2.38042634688471
9.75 2.38034228884673
10 2.38024958166572
10.25 2.38016896031536
10.5 2.38011331282912
10.75 2.38008819606301
11 2.38009313674224
11.25 2.38012335017902
11.5 2.3801715507327
11.75 2.38022960218271
12 2.38028983975093
12.25 2.38034598151316
12.5 2.38039361946968
12.75 2.38043033735005
13 2.38045553732529
13.25 2.3804700740376
13.5 2.38047579324805
13.75 2.38047505979374
14 2.38047033853956
14.25 2.38046386855783
14.5 2.38045744863375
14.75 2.38045233282899
15 2.38044922123003
15.25 2.38044832257964
15.5 2.38044946236991
15.75 2.38045221086185
16 2.38045600976291
16.25 2.38046028173393
16.5 2.3804645132852
16.75 2.3804683074686
17 2.38047140754294
17.25 2.38047369604139
17.5 2.38047517596648
17.75 2.38047594110528
18 2.38047614216231
18.25 2.38047595414604
18.5 2.38047554890118
18.75 2.38047507499276
19 2.38047464565788
19.25 2.38047433423084
19.5 2.38047417577632
19.75 2.38047417310101
};
\addplot [very thick, color4]
table {%
0 2.2962945105133
0.25 2.33368449831734
0.5 2.36013031433369
0.75 2.37518093993296
1 2.38038627996526
1.25 2.3784149154942
1.5 2.37221392465263
1.75 2.36442411889998
2 2.35707834785232
2.25 2.35151369148273
2.5 2.34841279344273
2.75 2.34791118144349
3 2.3497329994754
3.25 2.35333434722197
3.5 2.35804095616742
3.75 2.36316929337724
4 2.36812157400794
4.25 2.37244796059973
4.5 2.37587399379576
4.75 2.3782966797444
5 2.37975707198631
5.25 2.38039933319177
5.5 2.38042598785049
5.75 2.38005700044755
6 2.37949746838141
6.25 2.37891596087018
6.5 2.37843338694113
6.75 2.37812088603736
7 2.37800451944545
7.25 2.37807434265841
7.5 2.3782955770725
7.75 2.37861995798967
8 2.37899580623792
8.25 2.37937589271565
8.5 2.37972266519136
8.75 2.38001084940632
9 2.38022776101852
9.25 2.38037188311712
9.5 2.3804503432622
9.75 2.38047590655154
10 2.38046400304253
10.25 2.38043016773856
10.5 2.3803881178657
10.75 2.38034855025045
11 2.38031862734551
11.25 2.38030203854921
11.5 2.38029948202154
11.75 2.38030939636332
12 2.38032878537389
12.25 2.38035400734123
12.5 2.38038143707218
12.75 2.38040795098711
13 2.38043121850444
13.25 2.38044981351013
13.5 2.38046317865112
13.75 2.38047148523284
14 2.38047543437012
14.25 2.38047604007535
14.5 2.38047442715035
14.75 2.3804716660526
15 2.38046865648832
15.25 2.38046606212368
15.5 2.38046429165551
15.75 2.38046351673876
16 2.38046371503782
16.25 2.38046472643101
16.5 2.38046631187131
16.75 2.38046820671429
17 2.38047016318469
17.25 2.3804719793996
17.5 2.38047351476551
17.75 2.38047469337004
18 2.38047549793243
18.25 2.38047595781929
18.5 2.38047613396546
18.75 2.38047610358769
19 2.38047594666669
19.25 2.3804757355067
19.5 2.38047552790493
19.75 2.38047536389831
};
\addplot [very thick, color5]
table {%
0 2.33082091461279
0.25 2.33511797034536
0.5 2.3417179693289
0.75 2.34953637369081
1 2.3575418634099
1.25 2.36488723023745
1.5 2.37098754114574
1.75 2.37554286409631
2 2.37851313473382
2.25 2.38006026073655
2.5 2.38047530265634
2.75 2.38010648222497
3 2.37929889973262
3.25 2.37835139747025
3.5 2.37749151307813
3.75 2.37686651753179
4 2.37654702108116
4.25 2.37653915927818
4.5 2.37680153069723
4.75 2.37726358037887
5 2.37784282086682
5.25 2.37845906901473
5.5 2.37904467859468
5.75 2.37955048037232
6 2.37994775717727
6.25 2.38022700997691
6.5 2.38039451153228
6.75 2.38046766448256
7 2.38047009209091
7.25 2.38042717468848
7.5 2.3803624977951
7.75 2.38029543495408
8 2.38023988169428
8.25 2.38020400400914
8.5 2.38019077073611
8.75 2.38019899947334
9 2.38022465125859
9.25 2.38026214522267
9.5 2.38030552193202
9.75 2.38034934718236
10 2.38038930850699
10.25 2.38042250864508
10.5 2.38044749721242
10.75 2.38046410512093
11 2.38047315446601
11.25 2.38047611384148
11.5 2.38047475772337
11.75 2.38047087296454
12 2.38046603803727
12.25 2.38046148450437
12.5 2.38045803767598
12.75 2.38045612337273
13 2.38045582310822
13.25 2.38045695815127
13.5 2.38045918438611
13.75 2.38046208311552
14 2.38046523733262
14.25 2.38046828754058
14.5 2.38047096533124
14.75 2.38047310638848
15 2.38047464613441
15.25 2.38047560394992
15.5 2.38047606021314
15.75 2.38047613141281
16 2.38047594688008
16.25 2.3804756297646
16.5 2.38047528364019
16.75 2.38047498495617
17 2.38047478085004
17.25 2.38047469118183
17.5 2.3804747134727
17.75 2.380474829378
18 2.38047501143512
18.25 2.38047522921437
18.5 2.38047545420887
18.75 2.38047566317155
19 2.38047583990083
19.25 2.38047597562964
19.5 2.38047606834475
19.75 2.38047612140131
};
\addplot [very thick, color6]
table {%
0 2.3732264927045
0.25 2.37921636895218
0.5 2.38040176137942
0.75 2.37833418317964
1 2.37454568677375
1.25 2.3703337114717
1.5 2.36665470138305
1.75 2.36410038465295
2 2.36292831011642
2.25 2.36312368610892
2.5 2.36447590240731
2.75 2.36665786878918
3 2.36929928706538
3.25 2.37204725027593
3.5 2.37460983012742
3.75 2.37678095921158
4 2.37844757012552
4.25 2.37958225276834
4.5 2.38022604445823
4.75 2.38046629426903
5 2.380414014022
5.25 2.38018397984187
5.5 2.3798795405177
5.75 2.37958287477609
6 2.37935049307945
6.25 2.37921315118444
6.5 2.37917901457383
6.75 2.37923881784557
7 2.37937185597803
7.25 2.37955184098847
7.5 2.37975192461882
7.75 2.37994846768311
8 2.38012340111999
8.25 2.38026524133815
8.5 2.38036897904624
8.75 2.38043514791051
9 2.38046840825568
9.25 2.38047595566879
9.5 2.38046600673317
9.75 2.38044653771252
10 2.38042437401169
10.25 2.38040465671679
10.5 2.38039065745908
10.75 2.38038387559932
11 2.38038433250172
11.25 2.38039097501179
11.5 2.38040210853163
11.75 2.38041579715875
12 2.38043018875087
12.25 2.3804437428881
12.5 2.38045535817177
12.75 2.38046440881624
13 2.38047070954157
13.25 2.3804744321891
13.5 2.38047599693326
13.75 2.38047595946333
14 2.38047490912532
14.25 2.38047338857908
14.5 2.3804718396956
14.75 2.3804705760008
15 2.38046977820793
15.25 2.38046950763324
15.5 2.38046973108713
15.75 2.3804703511662
16 2.3804712367046
16.25 2.38047224945198
16.5 2.38047326459733
16.75 2.38047418410896
17 2.38047494307444
17.25 2.38047551005485
17.5 2.3804758829754
17.75 2.38047608224783
18 2.38047614273234
18.25 2.38047610587781
18.5 2.38047601300666
18.75 2.38047590031528
19 2.38047579579135
19.25 2.38047571794943
19.5 2.38047567608314
19.75 2.38047567160111
};
\addplot [very thick, color0]
table {%
0 2.36526843153546
0.25 2.36783651945941
0.5 2.3706579720858
0.75 2.37340250403451
1 2.37581991077472
1.25 2.3777547100199
1.5 2.37914418929457
1.75 2.38000426457476
2 2.3804081675832
2.25 2.38046285886791
2.5 2.38028705425011
2.75 2.3799934409095
3 2.37967633292069
3.25 2.37940491170461
3.5 2.37922141588274
3.75 2.37914318735002
4 2.37916729589428
4.25 2.37927648809833
4.5 2.37944537412728
4.75 2.37964601419236
5 2.37985235772693
5.25 2.38004326877759
5.5 2.38020411713533
5.75 2.3803271028209
6 2.38041060111689
6.25 2.38045786808829
6.5 2.38047544107106
6.75 2.38047152127004
7 2.38045455278261
7.25 2.38043213127939
7.5 2.38041029698299
7.75 2.38039320451436
8 2.38038311394767
8.25 2.38038062208597
8.5 2.3803850428398
8.75 2.38039484887604
9 2.3804081037999
9.25 2.3804228322546
9.5 2.38043729640071
9.75 2.38045016902067
10 2.38046060706421
10.25 2.38046824298978
10.5 2.38047311616744
10.75 2.38047556892437
11 2.38047612983451
11.25 2.38047540264003
11.5 2.38047397366742
11.75 2.38047234488954
12 2.38047089463625
12.25 2.38046986388098
12.5 2.38046936328815
12.75 2.38046939481811
13 2.38046988132009
13.25 2.38047069839182
13.5 2.38047170379902
13.75 2.38047276138092
14 2.380473757836
14.25 2.38047461210305
14.5 2.38047527805545
14.75 2.38047574195009
15 2.38047601629472
15.25 2.38047613190348
15.5 2.38047612958475
15.75 2.380476052649
16 2.38047594099479
16.25 2.38047582713258
16.5 2.38047573414321
16.75 2.38047567535033
17 2.38047565529495
17.25 2.38047567156339
17.5 2.38047571701503
17.75 2.38047578201825
18 2.3804758564135
18.25 2.38047593102128
18.5 2.38047599862783
18.75 2.38047605445485
19 2.38047609618317
19.25 2.38047612364987
19.5 2.38047613834916
19.75 2.38047614283729
};
\addplot [very thick, color1]
table {%
0 2.3732264927045
0.25 2.37921636895217
0.5 2.38040176137942
0.75 2.37833418317964
1 2.37454568677375
1.25 2.37033371147169
1.5 2.36665470138303
1.75 2.36410038465293
2 2.36292831011639
2.25 2.36312368610888
2.5 2.36447590240728
2.75 2.36665786878915
3 2.36929928706536
3.25 2.37204725027592
3.5 2.37460983012742
3.75 2.37678095921158
4 2.37844757012553
4.25 2.37958225276834
4.5 2.38022604445823
4.75 2.38046629426903
5 2.380414014022
5.25 2.38018397984187
5.5 2.3798795405177
5.75 2.37958287477609
6 2.37935049307945
6.25 2.37921315118443
6.5 2.37917901457382
6.75 2.37923881784556
7 2.37937185597802
7.25 2.37955184098846
7.5 2.37975192461882
7.75 2.37994846768311
8 2.38012340111999
8.25 2.38026524133815
8.5 2.38036897904624
8.75 2.38043514791052
9 2.38046840825568
9.25 2.38047595566879
9.5 2.38046600673318
9.75 2.38044653771252
10 2.3804243740117
10.25 2.38040465671679
10.5 2.38039065745908
10.75 2.38038387559932
11 2.38038433250172
11.25 2.38039097501178
11.5 2.38040210853161
11.75 2.38041579715874
12 2.38043018875087
12.25 2.3804437428881
12.5 2.38045535817177
12.75 2.38046440881624
13 2.38047070954157
13.25 2.38047443218911
13.5 2.38047599693326
13.75 2.38047595946333
14 2.38047490912532
14.25 2.38047338857908
14.5 2.38047183969561
14.75 2.3804705760008
15 2.38046977820793
15.25 2.38046950763324
15.5 2.38046973108714
15.75 2.3804703511662
16 2.3804712367046
16.25 2.38047224945198
16.5 2.38047326459733
16.75 2.38047418410896
17 2.38047494307444
17.25 2.38047551005485
17.5 2.3804758829754
17.75 2.38047608224783
18 2.38047614273234
18.25 2.38047610587781
18.5 2.38047601300666
18.75 2.38047590031528
19 2.38047579579135
19.25 2.38047571794943
19.5 2.38047567608314
19.75 2.38047567160111
};
\addplot [very thick, color2]
table {%
0 2.36526843153546
0.25 2.36783651945936
0.5 2.37065797208576
0.75 2.37340250403448
1 2.3758199107747
1.25 2.37775471001988
1.5 2.37914418929457
1.75 2.38000426457476
2 2.3804081675832
2.25 2.38046285886791
2.5 2.38028705425011
2.75 2.37999344090949
3 2.37967633292069
3.25 2.3794049117046
3.5 2.37922141588274
3.75 2.37914318735001
4 2.37916729589427
4.25 2.37927648809833
4.5 2.37944537412727
4.75 2.37964601419236
5 2.37985235772693
5.25 2.38004326877759
5.5 2.38020411713533
5.75 2.3803271028209
6 2.38041060111689
6.25 2.38045786808829
6.5 2.38047544107106
6.75 2.38047152127004
7 2.38045455278261
7.25 2.38043213127939
7.5 2.38041029698299
7.75 2.38039320451436
8 2.38038311394767
8.25 2.38038062208597
8.5 2.3803850428398
8.75 2.38039484887605
9 2.3804081037999
9.25 2.38042283225461
9.5 2.38043729640071
9.75 2.38045016902067
10 2.38046060706421
10.25 2.38046824298978
10.5 2.38047311616744
10.75 2.38047556892437
11 2.38047612983451
11.25 2.38047540264003
11.5 2.38047397366742
11.75 2.38047234488954
12 2.38047089463625
12.25 2.38046986388098
12.5 2.38046936328815
12.75 2.38046939481811
13 2.38046988132009
13.25 2.38047069839182
13.5 2.38047170379902
13.75 2.38047276138093
14 2.380473757836
14.25 2.38047461210305
14.5 2.38047527805545
14.75 2.38047574195009
15 2.38047601629472
15.25 2.38047613190348
15.5 2.38047612958475
15.75 2.380476052649
16 2.38047594099479
16.25 2.38047582713258
16.5 2.38047573414321
16.75 2.38047567535033
17 2.38047565529496
17.25 2.38047567156339
17.5 2.38047571701504
17.75 2.38047578201825
18 2.3804758564135
18.25 2.38047593102128
18.5 2.38047599862783
18.75 2.38047605445485
19 2.38047609618317
19.25 2.38047612364987
19.5 2.38047613834916
19.75 2.38047614283729
};
\end{axis}

\end{tikzpicture}

%% file: pics/mtlls-alpha0.4-rho0.55-xzrad0.25.tikz
%
%
%
\begin{tikzpicture}

\definecolor{color1}{rgb}{0.247058823529412,0.623529411764706,0.749019607843137}
\definecolor{color0}{rgb}{0.623529411764706,0.811764705882353,0.874509803921569}
\definecolor{color3}{rgb}{0,0.247058823529412,0.498039215686275}
\definecolor{color2}{rgb}{0,0.435294117647059,0.623529411764706}
\definecolor{color5}{rgb}{0,0.129411764705882,0.247058823529412}
\definecolor{color4}{rgb}{0,0.0588235294117647,0.372549019607843}
\definecolor{color6}{rgb}{0,0.317647058823529,0.12156862745098}

\begin{axis}[
xlabel={t},
ylabel={$m_t\phantom{\|{\xi}\|}$},
xtick={0, 4, 8, 12, 16, 20},
ytick={0., 0.05, 0.1},
yticklabels={0., 0.05, 0.1},
xmin=0, xmax=20,
ymin=0, ymax=0.12,
width=\figurewidth,
height=\figureheight,
tick align=outside,
xmajorgrids,
x grid style={white!80.0!black},
ymajorgrids,
y grid style={white!80.0!black},
axis line style={white!80.0!black}
]
\addplot [very thick, color0]
table {%
0 0
0.25 0.02264441288259
0.5 0.017762657902513
0.75 0.0155733628508947
1 0.0108897582507877
1.25 0.00911764941573366
1.5 0.00714184507120932
1.75 0.0080637509638222
2 0.00666424045523726
2.25 0.0067237422971179
2.5 0.00750863356722192
2.75 0.00857679209147852
3 0.00885022062504645
3.25 0.00974030943873171
3.5 0.00964474548784097
3.75 0.00995149136498241
4 0.00947080670476304
4.25 0.00911013091074852
4.5 0.00839966262413411
4.75 0.00772961389885418
5 0.00715248658386109
5.25 0.00658044431710458
5.5 0.00631446996771522
5.75 0.00588806642478481
6 0.00556956058101868
6.25 0.00525784650283604
6.5 0.00486178926134202
6.75 0.00464899808450436
7 0.0042638256360876
7.25 0.00412544641893061
7.5 0.00381418092283493
7.75 0.00356200894060432
8 0.00343938616946656
8.25 0.00322635371443229
8.5 0.00308657043028184
8.75 0.00289738901446614
9 0.00275733279203471
9.25 0.00257898635838065
9.5 0.00245443977684104
9.75 0.00229310389602703
10 0.0021459112489878
10.25 0.00206035347747339
10.5 0.00194654592183679
10.75 0.00188346911066041
11 0.00179941578797495
11.25 0.00175597371312933
11.5 0.00169777703978788
11.75 0.00166544411076029
12 0.00161464095701201
12.25 0.00156139200415716
12.5 0.00152064745142742
12.75 0.0014605601615962
13 0.0014105600000727
13.25 0.00135468924611705
13.5 0.00130150183620735
13.75 0.00124501041484326
14 0.00119538056603384
14.25 0.00114222270648239
14.5 0.00110051879962467
14.75 0.0010552829636047
15 0.00101558184761786
15.25 0.000989390751392984
15.5 0.000960589235468407
15.75 0.000942856194286529
16 0.000921549656597774
16.25 0.000907753051933271
16.5 0.000887051927576442
16.75 0.000870655757743783
17 0.000845484124662137
17.25 0.000824558552217686
17.5 0.000795573633394087
17.75 0.000765680059940904
18 0.000742627576403082
18.25 0.000714009279524928
18.5 0.000692468264490218
18.75 0.000667346851698483
19 0.000649102416278089
19.25 0.000628635871746019
19.5 0.000614583920596835
19.75 0.000598990058094506
};
\addplot [very thick, color1]
table {%
0 0
0.25 0.0110668792268545
0.5 0.0262747875237374
0.75 0.0294981058095674
1 0.0339250586743988
1.25 0.0338582144555423
1.5 0.0321843509372282
1.75 0.0283794235092673
2 0.0238232516056743
2.25 0.0214594370433196
2.5 0.0202316813072324
2.75 0.0186721358192217
3 0.0189975666950048
3.25 0.0155904504979956
3.5 0.0161195717122692
3.75 0.0132348967265009
4 0.0136175414037235
4.25 0.011970909358508
4.5 0.0120801852705919
4.75 0.0111139302947155
5 0.0106188614276642
5.25 0.01010669303524
5.5 0.00947728913566354
5.75 0.008813839337049
6 0.00812344573952423
6.25 0.00750994749786033
6.5 0.00739700988971108
6.75 0.00693753488101436
7 0.00697826367001194
7.25 0.00664974666586487
7.5 0.00646541198382941
7.75 0.00603624539225908
8 0.00585533647330957
8.25 0.00538632364778222
8.5 0.00521935763357348
8.75 0.00479248527292247
9 0.00460208508646726
9.25 0.0042640043300726
9.5 0.00406724191284911
9.75 0.00382207296282355
10 0.00361374964567977
10.25 0.00345469432320637
10.5 0.00326454570728574
10.75 0.00313491339129788
11 0.00296340969636997
11.25 0.00286538220715522
11.5 0.00272783597268594
11.75 0.00265865016142927
12 0.00255693407357406
12.25 0.00247087924835377
12.5 0.00242837266519819
12.75 0.00235522557451684
13 0.00230593101333324
13.25 0.00222654895435538
13.5 0.00216190233381181
13.75 0.00207194975616386
14 0.00199790418744217
14.25 0.00191100277729763
14.5 0.00183719996016563
14.75 0.00176742342599339
15 0.00169904862022804
15.25 0.00164036736383272
15.5 0.00157900180121461
15.75 0.001533924855925
16 0.00148448300806858
16.25 0.00145305186647471
16.5 0.00141670894400877
16.75 0.00139487386798784
17 0.0013646477919967
17.25 0.00134297760549612
17.5 0.0013094929400702
17.75 0.00127281210500595
18 0.00124224265719024
18.25 0.00120075878974729
18.5 0.00116666986607309
18.75 0.00112417398848387
19 0.00109065887274406
19.25 0.0010513471763486
19.5 0.00102166262571162
19.75 0.000988265859069599
};
\addplot [very thick, color2]
table {%
0 0
0.25 0.0586815908827445
0.5 0.0191941808849225
0.75 0.0211873847066961
1 0.0195295930758203
1.25 0.0280886158111274
1.5 0.0314566615784638
1.75 0.034375055324955
2 0.0340033946794081
2.25 0.031384625311843
2.5 0.0267996034469569
2.75 0.0253154776791624
3 0.0238757645703849
3.25 0.0251219553182144
3.5 0.0251587584383362
3.75 0.0222712994151054
4 0.0226015106147602
4.25 0.0184352513193635
4.5 0.0187438977246815
4.75 0.0155379990106876
5 0.0135828359135542
5.25 0.0134367483301261
5.5 0.0120323662581873
5.75 0.011525592672883
6 0.0105738552354364
6.25 0.00991112491615816
6.5 0.00900309808711076
6.75 0.00839445924513979
7 0.00802494688047342
7.25 0.00754410452801224
7.5 0.00780369382496576
7.75 0.0077419191066158
8 0.00746979179082137
8.25 0.0070966146255005
8.5 0.00688806289883701
8.75 0.00634585118390345
9 0.00612547988957441
9.25 0.00559639033635331
9.5 0.0053511207622826
9.75 0.0049367083794387
10 0.00463563619789585
10.25 0.00442296027133793
10.5 0.0042001074282549
10.75 0.00401855362156085
11 0.00382802319466369
11.25 0.00368444858878528
11.5 0.00350577310982343
11.75 0.0034002855394142
12 0.00324641820045489
12.25 0.00310776127570778
12.5 0.00304918108157945
12.75 0.00294568156547159
13 0.00289536602115141
13.25 0.00280213458538828
13.5 0.00273780886917878
13.75 0.00263410366134577
14 0.00255228361037987
14.25 0.00244335683567045
14.5 0.00235353908656352
14.75 0.00225002899902689
15 0.00216427569038815
15.25 0.00208395603519438
15.5 0.00201062909683481
15.75 0.00194431154782708
16 0.00188133951111868
16.25 0.00183186413477604
16.5 0.00177820019814986
16.75 0.00174438368115547
17 0.00170386985659789
17.25 0.00167861131075116
17.5 0.00164118167009981
17.75 0.00160150967169808
18 0.00156933406917611
18.25 0.00152324730020803
18.5 0.00148389807938728
18.75 0.00143382206467362
19 0.00139183150370229
19.25 0.00134218303242309
19.5 0.00130236331037489
19.75 0.00125739031225057
};
\addplot [very thick, color3]
table {%
0 0
0.25 0.118650122370117
0.5 0.074621240185788
0.75 0.0598337021288903
1 0.0371220106185745
1.25 0.0323805267584137
1.5 0.0246331580952699
1.75 0.0281065164362168
2 0.025610880244085
2.25 0.0238921453828708
2.5 0.0232150223721073
2.75 0.0203904795016731
3 0.0175026267665341
3.25 0.0172868693886748
3.5 0.0166785830289532
3.75 0.0185949017437677
4 0.0186665171399293
4.25 0.0176408259064971
4.5 0.0178106982142154
4.75 0.0149476350550678
5 0.0122797974056724
5.25 0.0124378536080678
5.5 0.0106734693236637
5.75 0.0105597934932339
6 0.00935956974679289
6.25 0.00900760472732352
6.5 0.00814074857462883
6.75 0.00769553218462047
7 0.00704339868499762
7.25 0.00663039400423993
7.5 0.00611953242231912
7.75 0.006299589634646
8 0.0060130745011527
8.25 0.00611446426407651
8.5 0.00591647568447978
8.75 0.00575104890974331
9 0.00558013284461318
9.25 0.00519545949460013
9.5 0.00499750442566159
9.75 0.00456209308600357
10 0.00418690663012725
10.25 0.00399120205467249
10.5 0.00372073715002911
10.75 0.0035455598426825
11 0.00336088666354314
11.25 0.00321672310080577
11.5 0.00308003018039782
11.75 0.0029690748683839
12 0.00284685671777764
12.25 0.00272483003645564
12.5 0.00266108921778433
12.75 0.0025630850811103
13 0.0025186668649084
13.25 0.00243676592264708
13.5 0.00239348977161594
13.75 0.00231259890585981
14 0.00225533287293334
14.25 0.00216591137401789
14.5 0.00209582367280752
14.75 0.00200581766649354
15 0.0019178873078594
15.25 0.0018463041371438
15.5 0.00177126810156753
15.75 0.00170772041089762
16 0.00165184419434833
16.25 0.0016005356195749
16.5 0.00155567388101389
16.75 0.00151820912605826
17 0.00147941803872595
17.25 0.00145391200703555
17.5 0.0014225390383096
17.75 0.00139163529019986
18 0.00136925537873318
18.25 0.00133445839094419
18.5 0.00130533717124795
18.75 0.00126574918928211
19 0.00123147192680316
19.25 0.00118994764034362
19.5 0.00115464359761457
19.75 0.00111471298829555
};
\addplot [very thick, color4]
table {%
0 0
0.25 0.10746294297526
0.5 0.0845129456781126
0.75 0.0775687831855565
1 0.0589337903195369
1.25 0.0469395206073455
1.5 0.0341843345209489
1.75 0.0271938204509239
2 0.0235101332627452
2.25 0.0186600981677141
2.5 0.0181323051438493
2.75 0.0147662245695802
3 0.0139080623469363
3.25 0.0109863213308692
3.5 0.00968407911328489
3.75 0.00809394372429306
4 0.00766786550183606
4.25 0.00847903870033453
4.5 0.00843063345219237
4.75 0.00888147019743979
5 0.00804877582040513
5.25 0.00809028910914681
5.5 0.00690725084285398
5.75 0.00691728440836553
6 0.00600201693991511
6.25 0.00592103210646887
6.5 0.00522664329038811
6.75 0.00505962696364756
7 0.00450681972440806
7.25 0.00430508162561817
7.5 0.00392644639121617
7.75 0.00364880816478946
8 0.00346511407016936
8.25 0.00327131177604898
8.5 0.00313435325740657
8.75 0.00318583515809727
9 0.00308557013054945
9.25 0.00313642687370092
9.5 0.0030465970190493
9.75 0.00295047947512044
10 0.00275143489811246
10.25 0.00263862008391905
10.5 0.00242054441809391
10.75 0.00230393532746822
11 0.00213338344412216
11.25 0.00202814926804047
11.5 0.00190831495900408
11.75 0.00182196719062288
12 0.00174177924720312
12.25 0.00167205726148019
12.5 0.00161767340050436
12.75 0.00155787268507415
13 0.00152020569845834
13.25 0.00146874240394752
13.5 0.00144293755250591
13.75 0.00139691660417873
14 0.0013724871445465
14.25 0.00132408753148579
14.5 0.00129222533846669
14.75 0.00123981850433666
15 0.00118742322267403
15.25 0.00114741253617466
15.5 0.00109914326235649
15.75 0.00105925254277384
16 0.00101513252426868
16.25 0.000979549004470781
16.5 0.000944238719909919
16.75 0.000915405727901669
17 0.00089106798717862
17.25 0.000869801746306266
17.5 0.000851637119735655
17.75 0.000833007762507059
18 0.000820716201521843
18.25 0.000803654428483649
18.5 0.000790969003452027
18.75 0.000771214646193943
19 0.000754673837069644
19.25 0.000731834776192984
19.5 0.000712260982875903
19.75 0.000688628704587308
};
\addplot [very thick, color5]
table {%
0 0
0.25 0.0359032739899989
0.5 0.0374728532606358
0.75 0.0371215352625332
1 0.0357592477308794
1.25 0.0343740802900109
1.5 0.0315624286761455
1.75 0.0292948154302485
2 0.0255529346723274
2.25 0.0214717802046784
2.5 0.0185997576739893
2.75 0.0156808762437787
3 0.0136533278608924
3.25 0.0121802688542488
3.5 0.0108777021496819
3.75 0.00944163071805153
4 0.00853965822359852
4.25 0.00725021561722309
4.5 0.00658883025925013
4.75 0.00553916261574571
5 0.00490220086308435
5.25 0.00465211448068379
5.5 0.00446441533957459
5.75 0.00432770339848349
6 0.00422329155140148
6.25 0.00409576677580983
6.5 0.00385004471344685
6.75 0.00369895454455585
7 0.00335478979622272
7.25 0.00319634049898776
7.5 0.00287499147424142
7.75 0.00259857984448361
8 0.002466489193074
8.25 0.00225093281588535
8.5 0.00213999533238547
8.75 0.00198906795541766
9 0.00190241172776241
9.25 0.0017990083750328
9.5 0.00173647710286528
9.75 0.0016637735441867
10 0.00161522736332371
10.25 0.00157632907355166
10.5 0.00156071052243024
10.75 0.00151555887144116
11 0.00147078998707529
11.25 0.00141323422326204
11.5 0.00133945152633269
11.75 0.00127700144763038
12 0.00119671103133388
12.25 0.00112142594191055
12.5 0.00106961109823811
12.75 0.00101093948556634
13 0.00096938777132786
13.25 0.000927377483521258
13.5 0.000896599504272055
13.75 0.000867619321871633
14 0.000846515918242862
14.25 0.000825201386244377
14.5 0.000810608212182192
14.75 0.000790376945870858
15 0.00076767837219599
15.25 0.000751669688060947
15.5 0.00072504433808987
15.75 0.000704816569589592
16 0.000676061162078091
16.25 0.000653724399814516
16.5 0.000625610010305009
16.75 0.000603886085349879
17 0.000578999448820797
17.25 0.000559951657155081
17.5 0.000539472386542666
17.75 0.00052160992704198
18 0.00050872127061376
18.25 0.000495934981795474
18.5 0.000486945250768487
18.75 0.000478094848227296
19 0.00047150731887463
19.25 0.000463277774913407
19.5 0.000456037110622444
19.75 0.000445422402440768
};
\addplot [very thick, color6]
table {%
0 0
0.25 0.022644412882598
0.5 0.0177626579025256
0.75 0.0155733628509068
1 0.010889758250797
1.25 0.00911764941573892
1.5 0.0071418450712158
1.75 0.00806375096382785
2 0.00666424045523965
2.25 0.00672374229712053
2.5 0.00750863356722504
2.75 0.00857679209147975
3 0.00885022062504806
3.25 0.00974030943873345
3.5 0.00964474548784296
3.75 0.009951491364985
4 0.00947080670476571
4.25 0.00911013091075154
4.5 0.00839966262413697
4.75 0.00772961389885696
5 0.00715248658386348
5.25 0.0065804443171068
5.5 0.00631446996771793
5.75 0.00588806642478776
6 0.0055695605810221
6.25 0.00525784650284014
6.5 0.00486178926134659
6.75 0.0046489980845097
7 0.00426382563609326
7.25 0.00412544641893676
7.5 0.00381418092284048
7.75 0.00356200894060887
8 0.00343938616947061
8.25 0.0032263537144354
8.5 0.0030865704302846
8.75 0.00289738901446834
9 0.0027573327920369
9.25 0.00257898635838272
9.5 0.0024544397768433
9.75 0.00229310389602912
10 0.00214591124898979
10.25 0.0020603534774755
10.5 0.00194654592183884
10.75 0.00188346911066251
11 0.00179941578797706
11.25 0.00175597371313141
11.5 0.00169777703978983
11.75 0.00166544411076209
12 0.00161464095701371
12.25 0.00156139200415872
12.5 0.00152064745142897
12.75 0.00146056016159799
13 0.0014105600000742
13.25 0.00135468924611821
13.5 0.00130150183620861
13.75 0.00124501041484449
14 0.001195380566035
14.25 0.00114222270648308
14.5 0.00110051879962525
14.75 0.00105528296360527
15 0.00101558184761835
15.25 0.000989390751393322
15.5 0.000960589235468597
15.75 0.000942856194286558
16 0.000921549656597849
16.25 0.000907753051933506
16.5 0.00088705192757699
16.75 0.000870655757744772
17 0.000845484124663526
17.25 0.000824558552219296
17.5 0.000795573633395415
17.75 0.000765680059941383
18 0.000742627576401963
18.25 0.000714009279521516
18.5 0.000692468264483719
18.75 0.000667346851689179
19 0.000649102416264443
19.25 0.000628635871727975
19.5 0.000614583920574223
19.75 0.000598990058067633
};
\addplot [very thick, color0]
table {%
0 0
0.25 0.0110668792268545
0.5 0.0262747875237374
0.75 0.0294981058095674
1 0.0339250586743988
1.25 0.0338582144555423
1.5 0.032184350937228
1.75 0.0283794235092671
2 0.0238232516056741
2.25 0.0214594370433194
2.5 0.0202316813072321
2.75 0.0186721358192215
3 0.0189975666950046
3.25 0.0155904504979955
3.5 0.0161195717122692
3.75 0.0132348967265009
4 0.0136175414037234
4.25 0.0119709093585079
4.5 0.0120801852705919
4.75 0.0111139302947155
5 0.0106188614276641
5.25 0.01010669303524
5.5 0.00947728913566352
5.75 0.00881383933704898
6 0.00812344573952422
6.25 0.00750994749786033
6.5 0.00739700988971106
6.75 0.00693753488101433
7 0.0069782636700119
7.25 0.00664974666586483
7.5 0.00646541198382936
7.75 0.00603624539225904
8 0.00585533647330953
8.25 0.00538632364778224
8.5 0.00521935763357351
8.75 0.00479248527292244
9 0.00460208508646723
9.25 0.00426400433007257
9.5 0.00406724191284909
9.75 0.00382207296282352
10 0.00361374964567976
10.25 0.00345469432320636
10.5 0.00326454570728573
10.75 0.00313491339129786
11 0.00296340969636995
11.25 0.0028653822071552
11.5 0.00272783597268592
11.75 0.00265865016142925
12 0.00255693407357403
12.25 0.00247087924835375
12.5 0.00242837266519817
12.75 0.00235522557451682
13 0.00230593101333323
13.25 0.00222654895435536
13.5 0.0021619023338118
13.75 0.00207194975616386
14 0.00199790418744217
14.25 0.00191100277729764
14.5 0.00183719996016565
14.75 0.00176742342599341
15 0.00169904862022806
15.25 0.00164036736383273
15.5 0.00157900180121462
15.75 0.001533924855925
16 0.00148448300806858
16.25 0.00145305186647471
16.5 0.00141670894400876
16.75 0.00139487386798783
17 0.00136464779199669
17.25 0.00134297760549611
17.5 0.0013094929400702
17.75 0.00127281210500595
18 0.00124224265719024
18.25 0.0012007587897473
18.5 0.0011666698660731
18.75 0.00112417398848388
19 0.00109065887274409
19.25 0.00105134717634864
19.5 0.00102166262571168
19.75 0.000988265859069665
};
\addplot [very thick, color1]
table {%
0 0
0.25 0.0586815908827338
0.5 0.0191941808849256
0.75 0.0211873847067028
1 0.0195295930758303
1.25 0.0280886158111484
1.5 0.0314566615784942
1.75 0.0343750553249875
2 0.03400339467944
2.25 0.0313846253118724
2.5 0.0267996034469819
2.75 0.0253154776791862
3 0.0238757645704072
3.25 0.0251219553182382
3.5 0.0251587584383598
3.75 0.0222712994151259
4 0.022601510614781
4.25 0.01843525131938
4.5 0.0187438977246982
4.75 0.0155379990107015
5 0.0135828359135664
5.25 0.013436748330138
5.5 0.0120323662581978
5.75 0.0115255926728929
6 0.010573855235445
6.25 0.00991112491616625
6.5 0.00900309808711793
6.75 0.00839445924514645
7 0.00802494688048032
7.25 0.00754410452801874
7.5 0.00780369382497257
7.75 0.00774191910662245
8 0.00746979179082781
8.25 0.00709661462550633
8.5 0.00688806289884271
8.75 0.00634585118390875
9 0.00612547988957964
9.25 0.0055963903363581
9.5 0.00535112076228729
9.75 0.00493670837944314
10 0.00463563619790005
10.25 0.00442296027134197
10.5 0.0042001074282588
10.75 0.0040185536215648
11 0.00382802319466751
11.25 0.003684448588789
11.5 0.00350577310982685
11.75 0.00340028553941748
12 0.00324641820045795
12.25 0.00310776127571065
12.5 0.00304918108158224
12.75 0.00294568156547425
13 0.00289536602115408
13.25 0.00280213458539095
13.5 0.00273780886918146
13.75 0.00263410366134835
14 0.00255228361038233
14.25 0.00244335683567271
14.5 0.00235353908656559
14.75 0.00225002899902875
15 0.00216427569038976
15.25 0.00208395603519576
15.5 0.00201062909683579
15.75 0.00194431154782753
16 0.00188133951111857
16.25 0.00183186413477551
16.5 0.00177820019814887
16.75 0.00174438368115377
17 0.00170386985659588
17.25 0.00167861131074889
17.5 0.00164118167009771
17.75 0.00160150967169623
18 0.00156933406917432
18.25 0.00152324730020645
18.5 0.00148389807938572
18.75 0.00143382206467183
19 0.00139183150369975
19.25 0.00134218303241968
19.5 0.00130236331037034
19.75 0.00125739031224538
};
\addplot [very thick, color2]
table {%
0 0
0.25 0.118650122370037
0.5 0.0746212401858144
0.75 0.0598337021289096
1 0.0371220106185902
1.25 0.032380526758433
1.5 0.0246331580952887
1.75 0.0281065164362444
2 0.0256108802441128
2.25 0.0238921453828989
2.5 0.0232150223721347
2.75 0.0203904795016971
3 0.0175026267665547
3.25 0.0172868693886972
3.5 0.0166785830289752
3.75 0.0185949017437944
4 0.0186665171399561
4.25 0.0176408259065198
4.5 0.0178106982142384
4.75 0.0149476350550855
5 0.0122797974056872
5.25 0.0124378536080828
5.5 0.0106734693236761
5.75 0.010559793493246
6 0.00935956974680385
6.25 0.00900760472733391
6.5 0.00814074857463811
6.75 0.00769553218462912
7 0.00704339868500519
7.25 0.00663039400424597
7.5 0.00611953242232262
7.75 0.00629958963465113
8 0.00601307450115719
8.25 0.00611446426408194
8.5 0.00591647568448503
8.75 0.00575104890974574
9 0.00558013284460788
9.25 0.00519545949458528
9.5 0.00499750442563266
9.75 0.00456209308596819
10 0.00418690663008982
10.25 0.00399120205462897
10.5 0.00372073714998813
10.75 0.00354555984263963
11 0.00336088666350629
11.25 0.00321672310077012
11.5 0.00308003018036865
11.75 0.0029690748683571
12 0.00284685671775658
12.25 0.0027248300364403
12.5 0.00266108921777456
12.75 0.002563085081105
13 0.0025186668648973
13.25 0.00243676592262935
13.5 0.00239348977159281
13.75 0.00231259890583416
14 0.0022553328729045
14.25 0.00216591137398119
14.5 0.00209582367275741
14.75 0.00200581766644114
15 0.00191788730781204
15.25 0.00184630413709057
15.5 0.00177126810150224
15.75 0.00170772041079883
16 0.001651844194173
16.25 0.00160053561931453
16.5 0.00155567388073033
16.75 0.00151820912577775
17 0.00147941803849406
17.25 0.00145391200685628
17.5 0.0014225390381881
17.75 0.00139163529013101
18 0.00136925537871631
18.25 0.00133445839099958
18.5 0.00130533717140134
18.75 0.00126574918953713
19 0.00123147192717709
19.25 0.00118994764082314
19.5 0.00115464359819119
19.75 0.00111471298890021
};
\addplot [very thick, color3]
table {%
0 0
0.25 0.107462942975372
0.5 0.0845129456780985
0.75 0.07756878318554
1 0.0589337903195142
1.25 0.0469395206073282
1.5 0.0341843345209355
1.75 0.0271938204509123
2 0.0235101332627344
2.25 0.018660098167705
2.5 0.0181323051438391
2.75 0.0147662245695713
3 0.0139080623469276
3.25 0.0109863213308626
3.5 0.00968407911327911
3.75 0.00809394372428799
4 0.00766786550183106
4.25 0.00847903870032763
4.5 0.00843063345218546
4.75 0.0088814701974333
5 0.00804877582039979
5.25 0.00809028910914139
5.5 0.00690725084284949
5.75 0.00691728440836094
6 0.00600201693991115
6.25 0.00592103210646487
6.5 0.00522664329038444
6.75 0.00505962696364393
7 0.00450681972440437
7.25 0.00430508162561381
7.5 0.00392644639121058
7.75 0.0036488081647811
8 0.00346511407015718
8.25 0.00327131177603233
8.5 0.00313435325738449
8.75 0.00318583515807389
9 0.00308557013052284
9.25 0.00313642687367642
9.5 0.00304659701902532
9.75 0.00295047947510115
10 0.00275143489809904
10.25 0.00263862008391235
10.5 0.00242054441809407
10.75 0.00230393532747631
11 0.00213338344413574
11.25 0.0020281492680605
11.5 0.0019083149590231
11.75 0.00182196719064024
12 0.00174177924721613
12.25 0.00167205726148916
12.5 0.00161767340051014
12.75 0.00155787268507655
13 0.00152020569845604
13.25 0.00146874240394155
13.5 0.00144293755249814
13.75 0.00139691660417406
14 0.00137248714454927
14.25 0.00132408753149827
14.5 0.00129222533849048
14.75 0.00123981850436332
15 0.00118742322270413
15.25 0.00114741253622006
15.5 0.001099143262417
15.75 0.00105925254284951
16 0.00101513252435458
16.25 0.000979549004567414
16.5 0.000944238720008463
16.75 0.000915405728004496
17 0.000891067987279118
17.25 0.000869801746405699
17.5 0.000851637119827767
17.75 0.000833007762592822
18 0.000820716201607223
18.25 0.000803654428570326
18.5 0.00079096900355409
18.75 0.000771214646310875
19 0.000754673837204376
19.25 0.000731834776332259
19.5 0.000712260983029448
19.75 0.000688628704756086
};
\addplot [very thick, color4]
table {%
0 0
0.25 0.0359032739899563
0.5 0.0374728532606214
0.75 0.0371215352625243
1 0.0357592477308848
1.25 0.0343740802900168
1.5 0.0315624286761522
1.75 0.0292948154302551
2 0.0255529346723339
2.25 0.021471780204684
2.5 0.0185997576739943
2.75 0.0156808762437828
3 0.013653327860896
3.25 0.0121802688542522
3.5 0.010877702149685
3.75 0.00944163071805407
4 0.00853965822360083
4.25 0.00725021561722508
4.5 0.00658883025925195
4.75 0.00553916261574724
5 0.0049022008630857
5.25 0.00465211448068506
5.5 0.00446441533957578
5.75 0.00432770339848464
6 0.0042232915514026
6.25 0.00409576677581092
6.5 0.00385004471344796
6.75 0.00369895454455695
7 0.00335478979622369
7.25 0.00319634049898869
7.5 0.00287499147424223
7.75 0.00259857984448433
8 0.00246648919307469
8.25 0.00225093281588598
8.5 0.00213999533238607
8.75 0.00198906795541823
9 0.00190241172776294
9.25 0.00179900837503327
9.5 0.00173647710286572
9.75 0.00166377354418712
10 0.00161522736332411
10.25 0.001576329073552
10.5 0.00156071052243049
10.75 0.00151555887144126
11 0.00147078998707515
11.25 0.00141323422326171
11.5 0.00133945152633235
11.75 0.00127700144762988
12 0.00119671103133328
12.25 0.00112142594190995
12.5 0.00106961109823748
12.75 0.0010109394855658
13 0.000969387771327356
13.25 0.000927377483520862
13.5 0.000896599504271715
13.75 0.000867619321871401
14 0.000846515918242688
14.25 0.000825201386244295
14.5 0.000810608212182158
14.75 0.000790376945870864
15 0.000767678372195952
15.25 0.000751669688060752
15.5 0.00072504433808951
15.75 0.00070481656958898
16 0.000676061162077216
16.25 0.000653724399813277
16.5 0.000625610010303622
16.75 0.000603886085348439
17 0.000578999448819416
17.25 0.000559951657153717
17.5 0.000539472386541461
17.75 0.000521609927041005
18 0.000508721270612969
18.25 0.000495934981795012
18.5 0.000486945250768272
18.75 0.000478094848227203
19 0.000471507318874507
19.25 0.000463277774913098
19.5 0.000456037110621768
19.75 0.000445422402439716
};
\addplot [very thick, color5]
table {%
0 0
0.25 0.00835780757601245
0.5 0.0129505199057194
0.75 0.0136512745726563
1 0.0144119678936628
1.25 0.0137343254903541
1.5 0.0124792055717259
1.75 0.0106363896609636
2 0.00884344351932545
2.25 0.00808677454362488
2.5 0.00779362071366378
2.75 0.00680613913125056
3 0.00701543400163126
3.25 0.00574204687652563
3.5 0.00597221390874851
3.75 0.00500877395158301
4 0.00516108477987896
4.25 0.00465001421021018
4.5 0.00469187959699746
4.75 0.00449859477800932
5 0.00435046833847625
5.25 0.00414149643781617
5.5 0.00392331476849689
5.75 0.00365422098495744
6 0.00339905831625492
6.25 0.00314809970521291
6.5 0.00307338530325636
6.75 0.00288257737587597
7 0.00288580497507786
7.25 0.0027496429611603
7.5 0.00267717042850124
7.75 0.00250862102041211
8 0.00243432238644345
8.25 0.00224961694792275
8.5 0.0021820570235161
8.75 0.00200855778858879
9 0.00193078942787176
9.25 0.00179036181100533
9.5 0.00170862752617594
9.75 0.00160172890800585
10 0.00151045965830821
10.25 0.00144348695947913
10.5 0.00136354013733997
10.75 0.00130859024197477
11 0.00123839493349201
11.25 0.00119658531447387
11.5 0.00114013426579925
11.75 0.00111051526682358
12 0.00106766545195538
12.25 0.00103114192195722
12.5 0.00101342290067656
12.75 0.00098341691129489
13 0.000963453628511727
13.25 0.000930203631872192
13.5 0.000903985944774949
13.75 0.000866799093037264
14 0.000836499858479834
14.25 0.000799770767239429
14.5 0.000769339454529964
14.75 0.000739733669475807
15 0.000710975270657996
15.25 0.000686512887640516
15.5 0.000661089083060285
15.75 0.000642161994972603
16 0.000621350839878009
16.25 0.000608052319224795
16.5 0.000592660127490843
16.75 0.00058349100018245
17 0.000570822294049311
17.25 0.000561996821633509
17.5 0.000548124344759942
17.75 0.000532945528254426
18 0.000520466764028305
18.25 0.000503270787529664
18.5 0.000489215443701253
18.75 0.000471530413763047
19 0.000457593547084816
19.25 0.000441138340381908
19.5 0.000428704804892155
19.75 0.00041465857081365
};
\addplot [very thick, color6]
table {%
0 0
0.25 0.0365949279690296
0.5 0.017233389849584
0.75 0.014385166773805
1 0.00836234054770592
1.25 0.0115986513382456
1.5 0.0122723151524113
1.75 0.0142136663645401
2 0.0138250078674357
2.25 0.0132978070941289
2.5 0.0118337820680574
2.75 0.0104266299777378
3 0.00936884715062324
3.25 0.0101727743922894
3.5 0.0100535306583886
3.75 0.00983853529873896
4 0.00993420033399071
4.25 0.00844362936344572
4.5 0.00856534792146797
4.75 0.00705783283374704
5 0.0060082546345993
5.25 0.00604403504649659
5.5 0.00535764551480453
5.75 0.0052381346895183
6 0.00479461023245992
6.25 0.00456545453393034
6.5 0.0041645937129625
6.75 0.00391440752337524
7 0.00359445916034647
7.25 0.00338001196976208
7.5 0.00327550569205733
7.75 0.00335488540524415
8 0.00321723829455494
8.25 0.0031905959981104
8.5 0.00309486009922711
8.75 0.00294467630939301
9 0.00285568491041459
9.25 0.00263467202036135
9.5 0.00253122833819203
9.75 0.00231583599605947
10 0.00213780437670563
10.25 0.00203915046373528
10.5 0.00191434722586586
10.75 0.00182690995963621
11 0.00173679263355577
11.25 0.00166532712185543
11.5 0.00159163347202537
11.75 0.00153720393375138
12 0.00147110856745371
12.25 0.0014082144140276
12.5 0.00137741829633156
12.75 0.00132770361248629
13 0.00130549047432431
13.25 0.00126487173036251
13.5 0.00124101773695904
13.75 0.00119853916793292
14 0.00116689044349621
14.25 0.00111983733089059
14.5 0.00108234055976877
14.75 0.00103516827660068
15 0.000989818378563269
15.25 0.000953172130844691
15.5 0.000917406573350211
15.75 0.000885504323245735
16 0.000856480895666308
16.25 0.000831288109949764
16.5 0.000807625917333026
16.75 0.000789626882147883
17 0.000769762084391886
17.25 0.000757443026739094
17.5 0.000741245691650806
17.75 0.000724827603842356
18 0.000712572737985592
18.25 0.000693628341347961
18.5 0.000677800201259686
18.75 0.00065653425383866
19 0.000638421989306628
19.25 0.000616503304749604
19.5 0.000598295420332217
19.75 0.000577529293440937
};
\addplot [very thick, color0]
table {%
0 0
0.25 0.0510287681604495
0.5 0.0387388371252093
0.75 0.0349550896139757
1 0.0255807909300517
1.25 0.0194407131934359
1.5 0.0146845350241058
1.75 0.0122452006580527
2 0.0103140446822659
2.25 0.00826322965843758
2.5 0.00821251145384232
2.75 0.0067242657187699
3 0.00630487820724111
3.25 0.00497696868099944
3.5 0.00433320611055137
3.75 0.00407402097873561
4 0.00392901452988256
4.25 0.00439741510494961
4.5 0.00439282662909755
4.75 0.00444135413645863
5 0.00392849674164851
5.25 0.00396175437599401
5.5 0.003368640361598
5.75 0.00338110311961822
6 0.00289755733299162
6.25 0.00286318053546002
6.5 0.00253999871833392
6.75 0.00245940824452175
7 0.00223241899143705
7.25 0.00213092964098053
7.5 0.0019388687001602
7.75 0.00180876197735013
8 0.00171804550325177
8.25 0.00162905371973884
8.5 0.00156180689589609
8.75 0.00160522755242064
9 0.00155558051776386
9.25 0.00156978644579315
9.5 0.00152469156108228
9.75 0.00146742416255644
10 0.00136508442504507
10.25 0.00130879918025359
10.5 0.00120168186965202
10.75 0.00114401152534951
11 0.00105794648929788
11.25 0.00100636161201853
11.5 0.000946360362233629
11.75 0.000904279394365826
12 0.000865585043926083
12.25 0.000831202356755483
12.5 0.000804792003849327
12.75 0.000774579788210533
13 0.000756350145121257
13.25 0.000730231603153197
13.5 0.000717655488535034
13.75 0.000695336446305167
14 0.000683074697107507
14.25 0.000659799055038454
14.5 0.000643703102494623
14.75 0.000618320646506038
15 0.000592922508352433
15.25 0.000572871569405443
15.5 0.000548149212246925
15.75 0.000528311648444373
16 0.000506196774406948
16.25 0.000488605711252553
16.5 0.000471169858514909
16.75 0.000456991484365696
17 0.000444944979616385
17.25 0.000434546913589536
17.5 0.000425503357122577
17.75 0.000416115149279009
18 0.00041005716445912
18.25 0.000401424125940793
18.5 0.000395012406532333
18.75 0.000384991263297882
19 0.000376634988182649
19.25 0.000365239313713041
19.5 0.000355425112875778
19.75 0.000343625459731017
};
\addplot [very thick, color1]
table {%
0 0
0.25 0.00568848107051778
0.5 0.00790135824266573
0.75 0.00828544221382644
1 0.00920667536844387
1.25 0.00933218434199586
1.5 0.00950513397737747
1.75 0.00931409116362819
2 0.00904151920338737
2.25 0.00845980678637022
2.5 0.00797224977105892
2.75 0.0071157874798611
3 0.0064308679296395
3.25 0.00553661666118532
3.5 0.00486789942449329
3.75 0.00440267797177368
4 0.00394754015918044
4.25 0.00362206394720428
4.5 0.00329638161162081
4.75 0.00295209537256045
5 0.00258595641999027
5.25 0.00240138752335261
5.5 0.00211182465989564
5.75 0.00199900591680287
6 0.00184952114227217
6.25 0.00178464158274082
6.5 0.00170113748521184
6.75 0.00164306733935175
7 0.00156357081908443
7.25 0.00149508893508325
7.5 0.00139332479196282
7.75 0.00127633203674929
8 0.00121058504504948
8.25 0.00110556613161094
8.5 0.0010475110302354
8.75 0.000962861458231326
9 0.000915592187291028
9.25 0.000850996353791314
9.5 0.000815047922757644
9.75 0.000771128668263506
10 0.000735731897453123
10.25 0.000715019945564586
10.5 0.00068994065359879
10.75 0.000674652010009628
11 0.000653714149068608
11.25 0.000636773716729121
11.5 0.000619334638001666
11.75 0.000597395001120582
12 0.000573388109296051
12.25 0.000544470257281568
12.5 0.000520588962292786
12.75 0.000491349564797245
13 0.000469634738193914
13.25 0.000444842797420108
13.5 0.000427025758470609
13.75 0.000408131321698261
14 0.000394721390964186
14.25 0.000381329853461557
14.5 0.000372046195992528
14.75 0.000362740103471568
15 0.000354167695754098
15.25 0.000348470894478912
15.5 0.000339525458131614
15.75 0.000332807264599725
16 0.000321880061575153
16.25 0.000313335496920048
16.5 0.000301156670538295
16.75 0.00029162733149455
17 0.000279514477628632
17.25 0.000270166191482746
17.5 0.000259250242175407
17.75 0.000249253693849195
18 0.000241907720970165
18.25 0.000233893900067436
18.5 0.000228249882519691
18.75 0.000222256818088785
19 0.000218300449238976
19.25 0.000214018120308372
19.5 0.000211166576368195
19.75 0.000207537588605203
};
\addplot [very thick, color2]
table {%
0 0
0.25 0.00835780757601245
0.5 0.0129505199057203
0.75 0.0136512745726637
1 0.0144119678936765
1.25 0.0137343254903651
1.5 0.0124792055717344
1.75 0.0106363896609708
2 0.00884344351933162
2.25 0.00808677454362864
2.5 0.00779362071366746
2.75 0.00680613913125527
3 0.00701543400163607
3.25 0.00574204687652948
3.5 0.00597221390875246
3.75 0.00500877395158609
4 0.00516108477988208
4.25 0.004650014210213
4.5 0.00469187959700028
4.75 0.00449859477801213
5 0.00435046833847867
5.25 0.00414149643781848
5.5 0.00392331476849908
5.75 0.00365422098495947
6 0.00339905831625675
6.25 0.0031480997052146
6.5 0.0030733853032579
6.75 0.00288257737587738
7 0.00288580497507943
7.25 0.00274964296116178
7.5 0.00267717042850266
7.75 0.00250862102041347
8 0.00243432238644476
8.25 0.00224961694792394
8.5 0.00218205702351725
8.75 0.00200855778858984
9 0.00193078942787277
9.25 0.00179036181100631
9.5 0.00170862752617696
9.75 0.00160172890800684
10 0.00151045965830915
10.25 0.00144348695948002
10.5 0.00136354013734081
10.75 0.00130859024197555
11 0.00123839493349272
11.25 0.00119658531447453
11.5 0.00114013426579988
11.75 0.00111051526682416
12 0.00106766545195591
12.25 0.00103114192195772
12.5 0.00101342290067702
12.75 0.000983416911295336
13 0.000963453628512179
13.25 0.000930203631872655
13.5 0.000903985944775408
13.75 0.000866799093037704
14 0.000836499858480234
14.25 0.000799770767239747
14.5 0.000769339454530204
14.75 0.00073973366947599
15 0.000710975270658086
15.25 0.000686512887640749
15.5 0.00066108908306088
15.75 0.000642161994973509
16 0.000621350839879184
16.25 0.000608052319226264
16.5 0.000592660127491926
16.75 0.000583491000182391
17 0.000570822294048328
17.25 0.000561996821632526
17.5 0.000548124344759785
17.75 0.000532945528254985
18 0.000520466764028286
18.25 0.000503270787528413
18.5 0.000489215443698275
18.75 0.000471530413756797
19 0.000457593547074487
19.25 0.000441138340368417
19.5 0.000428704804874128
19.75 0.000414658570789962
};
\addplot [very thick, color3]
table {%
0 0
0.25 0.0365949279689879
0.5 0.0172333898496028
0.75 0.014385166773819
1 0.00836234054771518
1.25 0.0115986513382622
1.5 0.0122723151524286
1.75 0.0142136663645646
2 0.0138250078674616
2.25 0.0132978070941573
2.5 0.0118337820680833
2.75 0.01042662997776
3 0.00936884715064184
3.25 0.0101727743923094
3.5 0.010053530658408
3.75 0.00983853529875964
4 0.00993420033401156
4.25 0.00844362936346427
4.5 0.0085653479214868
4.75 0.00705783283376336
5 0.0060082546346128
5.25 0.00604403504651023
5.5 0.00535764551481535
5.75 0.00523813468952895
6 0.0047946102324695
6.25 0.00456545453393935
6.5 0.00416459371296988
6.75 0.00391440752338149
7 0.00359445916035126
7.25 0.00338001196976536
7.5 0.00327550569206021
7.75 0.00335488540524665
8 0.00321723829455585
8.25 0.00319059599811056
8.5 0.00309486009922594
8.75 0.0029446763093918
9 0.00285568491041354
9.25 0.00263467202036106
9.5 0.00253122833819227
9.75 0.00231583599605984
10 0.00213780437670584
10.25 0.00203915046373502
10.5 0.00191434722586553
10.75 0.00182690995963522
11 0.00173679263355467
11.25 0.00166532712185346
11.5 0.00159163347202329
11.75 0.00153720393374837
12 0.0014711085674505
12.25 0.00140821441402414
12.5 0.00137741829632748
12.75 0.00132770361248214
13 0.00130549047432003
13.25 0.00126487173035859
13.5 0.00124101773695536
13.75 0.00119853916792957
14 0.00116689044349317
14.25 0.00111983733088804
14.5 0.00108234055976652
14.75 0.0010351682765987
15 0.000989818378561767
15.25 0.000953172130843578
15.5 0.000917406573349591
15.75 0.00088550432324598
16 0.000856480895668038
16.25 0.000831288109952755
16.5 0.000807625917336866
16.75 0.000789626882151884
17 0.000769762084396367
17.25 0.000757443026744134
17.5 0.000741245691656025
17.75 0.00072482760384735
18 0.000712572737990543
18.25 0.000693628341352372
18.5 0.000677800201263058
18.75 0.000656534253840711
19 0.000638421989307814
19.25 0.000616503304749908
19.5 0.000598295420331839
19.75 0.000577529293440405
};
\addplot [very thick, color4]
table {%
0 0
0.25 0.0510287681604167
0.5 0.0387388371252188
0.75 0.0349550896139836
1 0.025580790930057
1.25 0.0194407131934399
1.5 0.0146845350241095
1.75 0.0122452006580558
2 0.0103140446822685
2.25 0.00826322965844001
2.5 0.00821251145384518
2.75 0.00672426571877286
3 0.00630487820724421
3.25 0.00497696868100192
3.5 0.00433320611055349
3.75 0.00407402097873734
4 0.0039290145298842
4.25 0.00439741510495147
4.5 0.00439282662909942
4.75 0.00444135413646085
5 0.00392849674165074
5.25 0.00396175437599627
5.5 0.00336864036159982
5.75 0.00338110311962008
6 0.00289755733299317
6.25 0.00286318053546158
6.5 0.00253999871833522
6.75 0.00245940824452303
7 0.00223241899143816
7.25 0.00213092964098164
7.5 0.00193886870016125
7.75 0.00180876197735115
8 0.00171804550325274
8.25 0.0016290537197396
8.5 0.00156180689589681
8.75 0.0016052275524214
9 0.0015555805177646
9.25 0.00156978644579389
9.5 0.00152469156108292
9.75 0.0014674241625569
10 0.00136508442504544
10.25 0.00130879918025373
10.5 0.00120168186965197
10.75 0.00114401152534928
11 0.0010579464892976
11.25 0.00100636161201816
11.5 0.000946360362233325
11.75 0.000904279394365547
12 0.000865585043925935
12.25 0.000831202356755433
12.5 0.000804792003849342
12.75 0.000774579788210631
13 0.000756350145121425
13.25 0.000730231603153433
13.5 0.000717655488535301
13.75 0.000695336446305475
14 0.000683074697107776
14.25 0.000659799055038678
14.5 0.000643703102494797
14.75 0.000618320646506178
15 0.000592922508352552
15.25 0.000572871569405559
15.5 0.000548149212247068
15.75 0.000528311648444545
16 0.000506196774407162
16.25 0.000488605711252833
16.5 0.000471169858515341
16.75 0.000456991484366299
17 0.000444944979617185
17.25 0.000434546913590533
17.5 0.000425503357123773
17.75 0.000416115149280355
18 0.000410057164460592
18.25 0.00040142412594225
18.5 0.000395012406533802
18.75 0.0003849912632993
19 0.000376634988184131
19.25 0.000365239313714558
19.5 0.000355425112877257
19.75 0.000343625459732289
};
\addplot [very thick, color5]
table {%
0 0
0.25 0.00568848107065723
0.5 0.0079013582427159
0.75 0.00828544221386499
1 0.00920667536845459
1.25 0.00933218434199787
1.5 0.00950513397736734
1.75 0.00931409116361669
2 0.00904151920336758
2.25 0.0084598067863555
2.5 0.00797224977104645
2.75 0.00711578747985374
3 0.00643086792963759
3.25 0.00553661666118814
3.5 0.00486789942450847
3.75 0.00440267797179885
4 0.00394754015921592
4.25 0.0036220639472335
4.5 0.00329638161166864
4.75 0.00295209537260412
5 0.00258595642003924
5.25 0.00240138752342902
5.5 0.00211182466001661
5.75 0.00199900591696077
6 0.00184952114243225
6.25 0.0017846415829095
6.5 0.00170113748535724
6.75 0.00164306733948701
7 0.00156357081918143
7.25 0.0014950889351627
7.5 0.00139332479198943
7.75 0.00127633203673375
8 0.00121058504501232
8.25 0.00110556613156251
8.5 0.00104751103017423
8.75 0.000962861458184124
9 0.000915592187255615
9.25 0.000850996353768193
9.5 0.000815047922737879
9.75 0.000771128668253328
10 0.000735731897458337
10.25 0.00071501994558028
10.5 0.000689940653620048
10.75 0.000674652010021877
11 0.000653714149061394
11.25 0.000636773716669401
11.5 0.000619334637892594
11.75 0.000597395000951019
12 0.000573388109082291
12.25 0.000544470257035965
12.5 0.000520588962003533
12.75 0.000491349564500597
13 0.000469634737888828
13.25 0.00044484279713922
13.5 0.000427025758186253
13.75 0.000408131321411705
14 0.000394721390618832
14.25 0.00038132985307479
14.5 0.000372046195417313
14.75 0.000362740102585469
15 0.000354167694625901
15.25 0.000348470893196413
15.5 0.000339525456801227
15.75 0.000332807263071497
16 0.000321880059958559
16.25 0.000313335495006383
16.5 0.00030115666834416
16.75 0.000291627329025416
17 0.00027951447501729
17.25 0.00027016618857942
17.5 0.000259250239245509
17.75 0.000249253690946703
18 0.000241907717959338
18.25 0.000233893897113287
18.5 0.000228249879496828
18.75 0.000222256815169101
19 0.000218300446285813
19.25 0.00021401811747744
19.5 0.000211166573509964
19.75 0.000207537585811638
};
\addplot [very thick, color6]
table {%
0 0
0.25 0.00738391522538073
0.5 0.00284883303620438
0.75 0.00276422397542957
1 0.00199971995582773
1.25 0.00293748941155054
1.5 0.00329222119488674
1.75 0.00365885530984069
2 0.00360137975673081
2.25 0.00338093430872731
2.5 0.00294135922977719
2.75 0.00262386549372723
3 0.00239549263962883
3.25 0.00256388561307524
3.5 0.00254164609215875
3.75 0.00239855928452899
4 0.00242675880239493
4.25 0.00204251634128067
4.5 0.0020745602064006
4.75 0.0017103118376736
5 0.00149006438983471
5.25 0.00149786822474052
5.5 0.00135199127711258
5.75 0.00132092608831444
6 0.00120378458486837
6.25 0.00114705123263385
6.5 0.00106331736797171
6.75 0.00100010528861775
7 0.00091852588884571
7.25 0.000864271925891701
7.5 0.000830894324451363
7.75 0.000848139779758056
8 0.000813225844735463
8.25 0.00080759546862501
8.5 0.000783464151991157
8.75 0.000747704218248783
9 0.000725706190250425
9.25 0.00067169801479003
9.5 0.000646045491528232
9.75 0.000592213808582356
10 0.000546643475430031
10.25 0.000521662003669812
10.5 0.000489073437147372
10.75 0.000466744195821957
11 0.000443197097063594
11.25 0.000424834062974653
11.5 0.00040592632350762
11.75 0.000391866973036176
12 0.000374922746454748
12.25 0.000358852973026087
12.5 0.00035087255777302
12.75 0.00033828906353182
13 0.000332603470749645
13.25 0.00032236006980204
13.5 0.000316413582552743
13.75 0.000305609496863827
14 0.00029773227782825
14.25 0.000285888039112589
14.5 0.000276487935591907
14.75 0.000264465192604716
15 0.000252934971426365
15.25 0.000243664344759893
15.5 0.000234321566927879
15.75 0.000226209687428253
16 0.000218784844322646
16.25 0.000212339843063966
16.5 0.000206288865693481
16.75 0.000201657035921689
17 0.000196547345956034
17.25 0.000193384101070865
17.5 0.000189236307172578
17.75 0.00018503701913591
18 0.000181987227425021
18.25 0.000177194855207029
18.5 0.000173240764040514
18.75 0.000167845318260092
19 0.000163283394238492
19.25 0.000157705574154519
19.5 0.000153086679705173
19.75 0.000147783733592859
};
\addplot [very thick, color0]
table {%
0 0
0.25 0.00321603742278675
0.5 0.00343945766451096
0.75 0.00343422926540932
1 0.00339198288402231
1.25 0.00329986722268502
1.5 0.00311907375423916
1.75 0.00296111392466669
2 0.00268349999730328
2.25 0.00234455556571433
2.5 0.00211024527288297
2.75 0.0017723767746278
3 0.00153435901861179
3.25 0.00136197449558873
3.5 0.00120709979900218
3.75 0.00108405292523009
4 0.00097745676428641
4.25 0.000855111431346161
4.5 0.00077868780478238
4.75 0.000669136952575749
5 0.000575245244432134
5.25 0.000537438857121027
5.5 0.000494203682835583
5.75 0.000474797366505923
6 0.000457219885705835
6.25 0.000443046011946097
6.5 0.000425188000500383
6.75 0.000408979994063991
7 0.000379914085395927
7.25 0.000361829523968885
7.5 0.000328988968288594
7.75 0.000297907280695923
8 0.000282436508616553
8.25 0.000257509809606831
8.5 0.000244351707000296
8.75 0.000225805784280797
9 0.000215377277257729
9.25 0.000201787160358281
9.5 0.000194109278259162
9.75 0.000185337645315976
10 0.000178151634442225
10.25 0.000173750201156929
10.5 0.000169116973868593
10.75 0.000164957144748878
11 0.000161676068950895
11.25 0.000156270789800688
11.5 0.000150329945399884
11.75 0.000143933365718955
12 0.000136080154409572
12.25 0.000127833897617912
12.5 0.000121925740975526
12.75 0.0001149110750283
13 0.000109946710099396
13.25 0.000104590895420931
13.5 0.000100763880974169
13.75 9.70307391167226e-05
14 9.43100118594038e-05
14.25 9.17167044532319e-05
14.5 8.98908915702455e-05
14.75 8.78276435651289e-05
15 8.55806866181765e-05
15.25 8.40374442480857e-05
15.5 8.13767092132775e-05
15.75 7.93782129043614e-05
16 7.63179053172144e-05
16.25 7.39710261261849e-05
16.5 7.08729134531919e-05
16.75 6.84670804293767e-05
17 6.55908410406977e-05
17.25 6.33875786114769e-05
17.5 6.09365155327737e-05
17.75 5.87530313358375e-05
18 5.716782754906e-05
18.25 5.55111618839478e-05
18.5 5.43645858653235e-05
18.75 5.32140873238458e-05
19 5.24235550898771e-05
19.25 5.1538445924043e-05
19.5 5.08286746999662e-05
19.75 4.98298723271291e-05
};
\addplot [very thick, color1]
table {%
0 0
0.25 0.00738391522537007
0.5 0.0028488330362102
0.75 0.00276422397543403
1 0.00199971995582676
1.25 0.00293748941155241
1.5 0.00329222119489073
1.75 0.0036588553098462
2 0.00360137975673721
2.25 0.00338093430873379
2.5 0.00294135922978326
2.75 0.0026238654937321
3 0.00239549263963315
3.25 0.00256388561308055
3.5 0.00254164609216408
3.75 0.00239855928453505
4 0.00242675880240104
4.25 0.00204251634128535
4.5 0.00207456020640533
4.75 0.00171031183767709
5 0.00149006438983776
5.25 0.00149786822474354
5.5 0.00135199127711524
5.75 0.00132092608831703
6 0.00120378458487074
6.25 0.00114705123263596
6.5 0.0010633173679734
6.75 0.00100010528861901
7 0.000918525888846586
7.25 0.000864271925892341
7.5 0.0008308943244526
7.75 0.000848139779759306
8 0.000813225844736632
8.25 0.000807595468626048
8.5 0.000783464151992259
8.75 0.000747704218250006
9 0.00072570619025192
9.25 0.000671698014791782
9.5 0.000646045491530223
9.75 0.000592213808584494
10 0.000546643475432143
10.25 0.000521662003672153
10.5 0.000489073437149506
10.75 0.00046674419582374
11 0.000443197097064382
11.25 0.000424834062973882
11.5 0.00040592632350508
11.75 0.000391866973031949
12 0.000374922746450185
12.25 0.000358852973021809
12.5 0.000350872557769407
12.75 0.00033828906352949
13 0.000332603470748197
13.25 0.000322360069800717
13.5 0.000316413582551356
13.75 0.000305609496862948
14 0.000297732277828647
14.25 0.000285888039114823
14.5 0.000276487935595903
14.75 0.000264465192610219
15 0.000252934971433193
15.25 0.000243664344768212
15.5 0.000234321566937594
15.75 0.000226209687439216
16 0.000218784844334051
16.25 0.00021233984307587
16.5 0.000206288865705165
16.75 0.000201657035933683
17 0.000196547345967847
17.25 0.000193384101082883
17.5 0.000189236307184028
17.75 0.000185037019146456
18 0.000181987227435083
18.25 0.000177194855215894
18.5 0.000173240764048774
18.75 0.000167845318267203
19 0.0001632833942451
19.25 0.000157705574159956
19.5 0.000153086679708666
19.75 0.000147783733593497
};
\addplot [very thick, color2]
table {%
0 0
0.25 0.00321603742273435
0.5 0.00343945766450269
0.75 0.00343422926540263
1 0.0033919828840209
1.25 0.00329986722268613
1.5 0.00311907375424537
1.75 0.00296111392467308
2 0.00268349999730995
2.25 0.00234455556572095
2.5 0.00211024527288916
2.75 0.00177237677463331
3 0.00153435901861654
3.25 0.00136197449559281
3.5 0.00120709979900577
3.75 0.00108405292523342
4 0.000977456764289425
4.25 0.000855111431348912
4.5 0.0007786878047849
4.75 0.000669136952577945
5 0.000575245244434056
5.25 0.000537438857122797
5.5 0.000494203682837092
5.75 0.000474797366507363
6 0.000457219885707245
6.25 0.00044304601194746
6.5 0.000425188000501691
6.75 0.000408979994065249
7 0.000379914085397176
7.25 0.000361829523970008
7.5 0.000328988968289445
7.75 0.000297907280696712
8 0.000282436508617346
8.25 0.000257509809607638
8.5 0.000244351707001157
8.75 0.000225805784281695
9 0.000215377277258715
9.25 0.000201787160359306
9.5 0.000194109278260265
9.75 0.000185337645317147
10 0.000178151634443431
10.25 0.00017375020115818
10.5 0.000169116973869758
10.75 0.000164957144750017
11 0.000161676068951926
11.25 0.000156270789801559
11.5 0.000150329945400493
11.75 0.000143933365719314
12 0.000136080154409666
12.25 0.000127833897617892
12.5 0.00012192574097546
12.75 0.00011491107502824
13 0.000109946710099773
13.25 0.000104590895422142
13.5 0.000100763880976324
13.75 9.70307391195733e-05
14 9.43100118629421e-05
14.25 9.17167044573438e-05
14.5 8.98908915750248e-05
14.75 8.78276435699238e-05
15 8.55806866229095e-05
15.25 8.40374442531314e-05
15.5 8.13767092184052e-05
15.75 7.9378212909928e-05
16 7.6317905323041e-05
16.25 7.39710261326747e-05
16.5 7.08729134599861e-05
16.75 6.84670804369407e-05
17 6.55908410488683e-05
17.25 6.33875786208116e-05
17.5 6.0936515543063e-05
17.75 5.87530313472157e-05
18 5.7167827562197e-05
18.25 5.55111618970273e-05
18.5 5.43645858789717e-05
18.75 5.32140873371723e-05
19 5.24235551024061e-05
19.25 5.15384459360556e-05
19.5 5.08286747124906e-05
19.75 4.98298723394537e-05
};
\end{axis}

\end{tikzpicture}

%% file: pics/oms-alpha0.4-rho0.55-xzrad0.25.tikz
%
%
%
\begin{tikzpicture}

\definecolor{color1}{rgb}{0.247058823529412,0.623529411764706,0.749019607843137}
\definecolor{color0}{rgb}{0.623529411764706,0.811764705882353,0.874509803921569}
\definecolor{color2}{rgb}{0,0.435294117647059,0.623529411764706}

\begin{axis}[
xlabel={$t^*$},
ylabel={$\phantom{K(t)}$},
xtick={0, 4, 8, 12, 16, 20},
xmin=0, xmax=20,
ymin=-0.5, ymax=4,
width=\figurewidth,
height=\figureheight,
tick align=outside,
xmajorgrids,
x grid style={white!80.0!black},
ymajorgrids,
y grid style={white!80.0!black},
axis line style={white!80.0!black},
legend style={at={(0.97,0.97)}, anchor=north east},
legend entries={{$-\omega^*(t^*)$},{$-\omega$},{zero}}
]
\addplot [very thick, color0]
table {%
0 inf
0.25 3.66412663855349
0.5 1.85231235090773
0.75 1.25658320005731
1 0.916116840137061
1.25 0.705143865671132
1.5 0.543861572973755
1.75 0.46200157877735
2 0.39860778865423
2.25 0.340017418691614
2.5 0.282143776881523
2.75 0.243999634985049
3 0.211121807172074
3.25 0.194603823803479
3.5 0.175709031330026
3.75 0.145712055489018
4 0.132840862792375
4.25 0.0991634830938613
4.5 0.0894525820618807
4.75 0.0616972571899039
5 0.0409204084563544
5.25 0.0317574253494483
5.5 0.015304052734875
5.75 0.00509308130308694
6 -0.00769792661862065
6.25 -0.0181851243593969
6.5 -0.0302324900701061
6.75 -0.0398634104032499
7 -0.0473823816396726
7.25 -0.0555706394639268
7.5 -0.0574292022599925
7.75 -0.0616629226536164
8 -0.0674019294410755
8.25 -0.0738289414502147
8.5 -0.0787127704757529
8.75 -0.0864178686601202
9 -0.0911762168434311
9.25 -0.0987471271240633
9.5 -0.103596406870065
9.75 -0.110089331061026
10 -0.11544009948536
10.25 -0.1198259562067
10.5 -0.124279105591642
10.75 -0.128234886454965
11 -0.132254038790666
11.25 -0.135689252699656
11.5 -0.139506256884845
11.75 -0.142400618397245
12 -0.145856056303871
12.25 -0.149098475728445
12.5 -0.151273010325924
12.75 -0.153991690753089
13 -0.155968738393957
13.25 -0.158482498958926
13.5 -0.16057275302165
13.75 -0.163178500727615
14 -0.16545922834748
14.25 -0.1681119803874
14.5 -0.170479004045165
14.75 -0.173042915072193
15 -0.175334336100907
15.25 -0.1775343051056
15.5 -0.17961593252012
15.75 -0.181575563892406
16 -0.183470435829019
16.25 -0.185129325714527
16.5 -0.186845481120328
16.75 -0.188208915323606
17 -0.189671364889802
17.25 -0.190852012383574
17.5 -0.192230311449852
17.75 -0.193636189856585
18 -0.19489538802997
18.25 -0.19639746779786
18.5 -0.197768745981743
18.75 -0.199335772687987
19 -0.200747017161743
19.25 -0.202304613397168
19.5 -0.203668985971994
19.75 -0.205135158952899
};
\addplot [very thick, color1]
table {%
0 -0.3
0.25 -0.3
0.5 -0.3
0.75 -0.3
1 -0.3
1.25 -0.3
1.5 -0.3
1.75 -0.3
2 -0.3
2.25 -0.3
2.5 -0.3
2.75 -0.3
3 -0.3
3.25 -0.3
3.5 -0.3
3.75 -0.3
4 -0.3
4.25 -0.3
4.5 -0.3
4.75 -0.3
5 -0.3
5.25 -0.3
5.5 -0.3
5.75 -0.3
6 -0.3
6.25 -0.3
6.5 -0.3
6.75 -0.3
7 -0.3
7.25 -0.3
7.5 -0.3
7.75 -0.3
8 -0.3
8.25 -0.3
8.5 -0.3
8.75 -0.3
9 -0.3
9.25 -0.3
9.5 -0.3
9.75 -0.3
10 -0.3
10.25 -0.3
10.5 -0.3
10.75 -0.3
11 -0.3
11.25 -0.3
11.5 -0.3
11.75 -0.3
12 -0.3
12.25 -0.3
12.5 -0.3
12.75 -0.3
13 -0.3
13.25 -0.3
13.5 -0.3
13.75 -0.3
14 -0.3
14.25 -0.3
14.5 -0.3
14.75 -0.3
15 -0.3
15.25 -0.3
15.5 -0.3
15.75 -0.3
16 -0.3
16.25 -0.3
16.5 -0.3
16.75 -0.3
17 -0.3
17.25 -0.3
17.5 -0.3
17.75 -0.3
18 -0.3
18.25 -0.3
18.5 -0.3
18.75 -0.3
19 -0.3
19.25 -0.3
19.5 -0.3
19.75 -0.3
};
\addplot [black]
table {%
0 0
0.25 0
0.5 0
0.75 0
1 0
1.25 0
1.5 0
1.75 0
2 0
2.25 0
2.5 0
2.75 0
3 0
3.25 0
3.5 0
3.75 0
4 0
4.25 0
4.5 0
4.75 0
5 0
5.25 0
5.5 0
5.75 0
6 0
6.25 0
6.5 0
6.75 0
7 0
7.25 0
7.5 0
7.75 0
8 0
8.25 0
8.5 0
8.75 0
9 0
9.25 0
9.5 0
9.75 0
10 0
10.25 0
10.5 0
10.75 0
11 0
11.25 0
11.5 0
11.75 0
12 0
12.25 0
12.5 0
12.75 0
13 0
13.25 0
13.5 0
13.75 0
14 0
14.25 0
14.5 0
14.75 0
15 0
15.25 0
15.5 0
15.75 0
16 0
16.25 0
16.5 0
16.75 0
17 0
17.25 0
17.5 0
17.75 0
18 0
18.25 0
18.5 0
18.75 0
19 0
19.25 0
19.5 0
19.75 0
};
\end{axis}

\end{tikzpicture}

%% file: pics/trjnrms-alpha0.4-rho0.55-xzrad0.25.tikz
%
%
%
\begin{tikzpicture}

\definecolor{color1}{rgb}{0.247058823529412,0.623529411764706,0.749019607843137}
\definecolor{color0}{rgb}{0.623529411764706,0.811764705882353,0.874509803921569}
\definecolor{color3}{rgb}{0,0.247058823529412,0.498039215686275}
\definecolor{color2}{rgb}{0,0.435294117647059,0.623529411764706}
\definecolor{color5}{rgb}{0,0.129411764705882,0.247058823529412}
\definecolor{color4}{rgb}{0,0.0588235294117647,0.372549019607843}
\definecolor{color6}{rgb}{0,0.317647058823529,0.12156862745098}

\begin{axis}[
xlabel={t},
xtick={0, 4, 8, 12, 16, 20},
ylabel={$\|\xi(t)\|$},
xmin=0, xmax=20,
ymin=0, ymax=0.35,
width=\figurewidth,
height=\figureheight,
tick align=outside,
xmajorgrids,
x grid style={white!80.0!black},
ymajorgrids,
y grid style={white!80.0!black},
axis line style={white!80.0!black}
]
\addplot [very thick, color0]
table {%
0 0.25
0.25 0.251146383428053
0.5 0.244829319327786
0.75 0.231840883301574
1 0.213352615820394
1.25 0.190761271597675
1.5 0.165571284966153
1.75 0.139316435177546
2 0.113538139367355
2.25 0.0898492596083413
2.5 0.0701061253520786
2.75 0.0565282172600941
3 0.0508820742419407
3.25 0.0522919294014882
3.5 0.0573500089601299
3.75 0.062965469848405
4 0.0674379485607353
4.25 0.070039081643766
4.5 0.0705610524446614
4.75 0.0690715243881196
5 0.0657898940157812
5.25 0.0610202030048773
5.5 0.0551110783275176
5.75 0.0484305452310498
6 0.041351443280204
6.25 0.0342475833621223
6.5 0.0275047313446296
6.75 0.0215542888199426
7 0.0169261114261107
7.25 0.0142104543187652
7.5 0.0136305280220959
7.75 0.0145709760869064
8 0.0160741003633552
8.25 0.017480183778644
8.5 0.0184689077714812
8.75 0.0189193818761878
9 0.0188152417493798
9.25 0.018196127911721
9.5 0.0171326157397494
9.75 0.0157120747237491
10 0.014030076089768
10.25 0.0121851278933212
10.5 0.0102761310550481
10.75 0.00840292734631743
11 0.00667151995523956
11.25 0.00520593616081118
11.5 0.0041593266513935
11.75 0.00367056058960003
12 0.00371150396997559
12.25 0.00405223904857049
12.5 0.00445874802317947
12.75 0.00479573720369808
13 0.00500341114268338
13.25 0.0050630293863408
13.5 0.00497742009794525
13.75 0.00476114040653608
14 0.00443522227233297
14.25 0.00402411567916335
14.5 0.00355379295295664
14.75 0.0030506457522898
15 0.00254111338594819
15.25 0.0020522676983172
15.5 0.00161388942164328
15.75 0.00126217932567703
16 0.00103965690923516
16.25 0.000970932484665234
16.5 0.0010229521273634
16.75 0.00112719477666548
17 0.00123128170050364
17.25 0.00130851996977911
17.5 0.00134823386401606
17.75 0.00134818458351538
18 0.00131062022964859
18.25 0.00124025638945549
18.5 0.00114315902254273
18.75 0.00102607300886405
19 0.000896010201424381
19.25 0.000760025329332828
19.5 0.000625213073115741
19.75 0.00049900358884399
};
\addplot [very thick, color1]
table {%
0 0.25
0.25 0.275127722567112
0.5 0.295996618057603
0.75 0.308706214989358
1 0.311840247943252
1.25 0.305367303843953
1.5 0.290126866203579
1.75 0.26754446647247
2 0.239420810932178
2.25 0.207753438392477
2.5 0.174603251625406
2.75 0.142040819292021
3 0.112217505114048
3.25 0.0875864275198848
3.5 0.0710218979189168
3.75 0.0646260372606558
4 0.0669517928505009
4.25 0.073562706332277
4.5 0.0806541453685901
4.75 0.0861831169169804
5 0.0892891297812527
5.25 0.0897412645009219
5.5 0.0876435082080922
5.75 0.0832857898744142
6 0.0770623984453785
6.25 0.0694225707790633
6.5 0.0608390778191436
6.75 0.0517901906219665
7 0.0427557900783181
7.25 0.0342335207114099
7.5 0.0267851625223407
7.75 0.0211029627275906
8 0.0179279639672403
8.25 0.0174462853278904
8.5 0.0187762387642777
8.75 0.0207116015254779
9 0.022462877837099
9.25 0.0236562147696757
9.5 0.024155764533384
9.75 0.0239503281463002
10 0.0230957950074789
10.25 0.0216851517119112
10.5 0.0198314287106417
10.75 0.017657257827362
11 0.0152885337758132
11.25 0.0128515140373551
11.5 0.0104740328372391
11.75 0.00829297090062553
12 0.0064702387568394
12.25 0.0052043871524027
12.5 0.00466058666028874
12.75 0.00477036136064153
13 0.00522639910949205
13.25 0.00574170623188594
13.5 0.00615643955520637
13.75 0.00640167843086409
14 0.00645736176662069
14.25 0.00632915141984971
14.5 0.00603677349497364
14.75 0.00560768606171059
15 0.00507339249814042
15.25 0.00446714734359602
15.5 0.00382263337006818
15.75 0.00317361436286808
16 0.00255483760615929
16.25 0.00200493070805061
16.5 0.00157121761090837
16.75 0.00130781316654531
17 0.0012401140564023
17.25 0.00131786227957775
17.5 0.00145340543027721
17.75 0.00158370406648732
18 0.00167752165832101
18.25 0.001722754936593
18.5 0.00171734358327062
18.75 0.00166459004330466
19 0.00157075151342925
19.25 0.0014436950806747
19.5 0.00129208472584159
19.75 0.00112488005443954
};
\addplot [very thick, color2]
table {%
0 0.25
0.25 0.263882136858604
0.5 0.290845690824092
0.75 0.317512268139703
1 0.337024424264276
1.25 0.346425831242443
1.5 0.344858667080553
1.75 0.332728080983126
2 0.311300527462983
2.25 0.282440052700061
2.5 0.248377015752218
2.75 0.211505747976784
3 0.174253564972967
3.25 0.139079020145252
3.5 0.108651512195352
3.75 0.0861136737168916
4 0.0745149500402411
4.25 0.0739830635282379
4.5 0.0802776520351605
4.75 0.0884736193114885
5 0.0955580732094744
5.25 0.100149713853952
5.5 0.101772989889851
5.75 0.100429412523194
6 0.096382263109848
6.25 0.0900428569095221
6.5 0.0819037343893893
6.75 0.0724961840596086
7 0.0623639285268887
7.25 0.0520520840947465
7.5 0.0421159382795277
7.75 0.0331601101865143
8 0.0259127167937454
8.25 0.0212361248450644
8.5 0.0196746404334641
8.75 0.0206277953135466
9 0.0227142625555175
9.25 0.0248417448271502
9.5 0.0264449460888646
9.75 0.027294182544333
10 0.027337012437327
10.25 0.026615131152156
10.5 0.0252220059886775
10.75 0.02327955054489
11 0.0209241250703294
11.25 0.018297870621352
11.5 0.015544034928728
11.75 0.0128065255106178
12 0.0102355276156356
12.25 0.0080022822305082
12.5 0.00631877681548622
12.75 0.00540642069516726
13 0.00530351167630609
13.25 0.00572740063034742
13.5 0.0063162146808501
13.75 0.00683943634718114
14 0.00718942612848824
14.25 0.0073280236118087
14.5 0.00725345332506591
14.75 0.00698355233192823
15 0.00654699536911154
15.25 0.00597828208920753
15.5 0.0053146636635632
15.75 0.00459429129982617
16 0.0038554074419478
16.25 0.00313680582368736
16.5 0.00248022667897784
16.75 0.00193531520645246
17 0.00156267706777897
17.25 0.00141018822895391
17.5 0.00145197767038033
17.75 0.0015931374988264
18 0.00174861242300228
18.25 0.00187187081831474
18.5 0.0019431130015253
18.75 0.00195682612497402
19 0.00191503685466517
19.25 0.00182389205536632
19.5 0.00169180694914393
19.75 0.00152837534531661
};
\addplot [very thick, color3]
table {%
0 0.25
0.25 0.226101893225543
0.5 0.233454643673274
0.75 0.256336978359105
1 0.281196468996076
1.25 0.300673988716302
1.5 0.311504620406655
1.75 0.312615649197995
2 0.304186158340311
2.25 0.287200796727768
2.5 0.263187178924637
2.75 0.234011542371676
3 0.201706205028968
3.25 0.16834818384078
3.5 0.136027189881432
3.75 0.106949221178992
4 0.0836789947531167
4.25 0.06912020535722
4.5 0.0648429258608113
4.75 0.06850204797914
5 0.0755070263664472
5.25 0.0824190696396488
5.5 0.0874994217693012
5.75 0.0900551274664221
6 0.0899442840175772
6.25 0.0873208027077949
6.5 0.0825039480916899
6.75 0.0759052577768361
7 0.0679835106892255
7.25 0.0592161924185798
7.5 0.0500841649314752
7.75 0.0410713934820825
8 0.0326867708832262
8.25 0.0255179959206314
8.5 0.0202906826705522
8.75 0.0177012792189473
9 0.0177120196709024
9.25 0.0192720785555949
9.5 0.0212272440266011
9.75 0.0228880762522724
10 0.0239437308360502
10.25 0.024291715626544
10.5 0.0239388485394158
10.75 0.0229512530726388
11 0.0214278663584137
11.25 0.0194850821370783
11.5 0.0172472742244598
11.75 0.0148411749014872
12 0.0123937454728075
12.25 0.0100345123726063
12.5 0.00790481658080812
12.75 0.00617526729928102
13 0.00504945985364227
13.25 0.00466089483629124
13.5 0.00487515974018675
13.75 0.00536624154932272
14 0.00587205474946031
14.25 0.00625596847236414
14.5 0.0064621086659776
14.75 0.00647730496478887
15 0.00631102931957753
15.25 0.00598518307392124
15.5 0.00552849052842479
15.75 0.00497314508913878
16 0.00435273748360812
16.25 0.00370113550634941
16.5 0.00305235291549988
16.75 0.002441821226436
17 0.00190981550310964
17.25 0.0015061706357646
17.5 0.00128370935186308
17.75 0.00125384260170274
18 0.00135159497496948
18.25 0.00149077202765993
18.5 0.00161566048865413
18.75 0.00170005130936375
19 0.00173452824929363
19.25 0.00171846843025274
19.5 0.00165598449348036
19.75 0.00155380148492605
};
\addplot [very thick, color4]
table {%
0 0.25
0.25 0.19615285339002
0.5 0.164463353507215
0.75 0.157688197564087
1 0.168575635532283
1.25 0.185966036831233
1.5 0.202220418182859
1.75 0.213601125414208
2 0.218639360679461
2.25 0.217072511031311
2.5 0.209320814637259
2.75 0.196218547679927
3 0.178851203154395
3.25 0.158440814602572
3.5 0.136264086400039
3.75 0.113607191298337
4 0.0917725094613239
4.25 0.0721628900430594
4.5 0.0564475712433489
4.75 0.0465541775046588
5 0.0435652081224678
5.25 0.0459557617050501
5.5 0.0506465355181336
5.75 0.0553051800035746
6 0.0587473472629387
6.25 0.0605001284298148
6.5 0.0604653926009613
6.75 0.058745154314287
7 0.0555523014198412
7.25 0.0511609388809514
7.5 0.0458762121048131
7.75 0.040015467610858
8 0.0338982316864796
8.25 0.0278459241999237
8.5 0.0221957428263854
8.75 0.0173356002775374
9 0.0137455532210138
9.25 0.0119026333058524
9.5 0.0118237756540172
9.75 0.0128318817533517
10 0.0141415477243174
10.25 0.0152726561940525
10.5 0.0160059072966859
10.75 0.0162668654321123
11 0.0160570212686862
11.25 0.0154189516857554
11.5 0.0144179555252994
11.75 0.0131314555797058
12 0.0116424968704328
12.25 0.0100358502981653
12.5 0.00839645072797355
12.75 0.00681075027697791
13 0.00537257406058968
13.25 0.00419458605592812
13.5 0.00341256273106874
13.75 0.00312280977779259
14 0.0032476437175306
14.25 0.00357108062330492
14.5 0.00391245659750874
14.75 0.0041759100732483
15 0.00432161520821825
15.25 0.0043394925563855
15.5 0.00423520253389833
15.75 0.00402303134559879
16 0.00372200691490654
16.25 0.00335358751316273
16.5 0.0029402332456864
16.75 0.00250461669804475
17 0.00206948887136286
17.25 0.00165846327275496
17.5 0.00129821026913814
17.75 0.00102166467501092
18 0.000864621550147651
18.25 0.000837572562828271
18.5 0.000899576923201578
18.75 0.000992374754812578
19 0.00107722371826817
19.25 0.00113563525585443
19.5 0.00116080879117108
19.75 0.00115206515987335
};
\addplot [very thick, color5]
table {%
0 0.25
0.25 0.210631317912371
0.5 0.171650626349714
0.75 0.135750474946718
1 0.105937635642287
1.25 0.0856966048502439
1.5 0.0776421089438606
1.75 0.0801893296673218
2 0.0880485994866769
2.25 0.0965897424778344
2.5 0.103303196587283
2.75 0.107122421690019
3 0.107750945422212
3.25 0.105302607491472
3.5 0.100120190231215
3.75 0.0926761306990518
4 0.0835120208275118
4.25 0.0731996577595607
4.5 0.0623182486177511
4.75 0.0514488836926583
5 0.0411934523454384
5.25 0.0322303411094656
5.5 0.0253942663362278
5.75 0.0215774941464856
6 0.0210022650731304
6.25 0.022605457611943
6.5 0.024935400016186
6.75 0.0270426779317327
7 0.028477855834
7.25 0.0290776778827929
7.5 0.0288287851766099
7.75 0.0277985255340787
8 0.0260989002785332
8.25 0.0238660516764067
8.5 0.0212476920423759
8.75 0.0183954660740448
9 0.0154614570626201
9.25 0.012599656256415
9.5 0.00997498950416824
9.75 0.00778259283859351
10 0.00626166164844689
10.25 0.00561046888031467
10.5 0.00574510975679854
10.75 0.00629504374591139
11 0.00691527663023831
11.25 0.00741391504340871
11.5 0.00770829462478988
11.75 0.0077744267674633
12 0.00761921783601412
12.25 0.00726646498904139
12.5 0.0067492556436289
12.75 0.00610553369904575
13 0.00537534268411499
13.25 0.00459924013478185
13.5 0.00381787961326837
13.75 0.00307310869923362
14 0.00241146745581689
14.25 0.00188998568585908
14.5 0.00157380555651646
14.75 0.00149320826470072
15 0.00158732088604976
15.25 0.00175061801364268
15.5 0.00190737516529666
15.75 0.00202011051344644
16 0.0020743193601899
16.25 0.00206755671977682
16.5 0.00200381879662361
16.75 0.00189064824526426
17 0.00173752414401027
17.25 0.00155487991397448
17.5 0.00135350394915509
17.75 0.00114423872728946
18 0.000938009617964392
18.25 0.000746347045447815
18.5 0.000582631671056243
18.75 0.00046343034650276
19 0.00040460976852517
19.25 0.000405154312201455
19.5 0.00044095107534273
19.75 0.000485655226510632
};
\addplot [very thick, color6]
table {%
0 0.25
0.25 0.251146383428094
0.5 0.244829319327834
0.75 0.231840883301629
1 0.21335261582045
1.25 0.19076127159773
1.5 0.165571284966204
1.75 0.139316435177592
2 0.113538139367394
2.25 0.0898492596083753
2.5 0.0701061253521042
2.75 0.056528217260113
3 0.0508820742419547
3.25 0.052291929401499
3.5 0.0573500089601409
3.75 0.0629654698484193
4 0.0674379485607533
4.25 0.0700390816437848
4.5 0.0705610524446813
4.75 0.0690715243881374
5 0.0657898940157997
5.25 0.0610202030048758
5.5 0.0551110783274678
5.75 0.0484305452309901
6 0.0413514432800833
6.25 0.0342475833619811
6.5 0.0275047313444987
6.75 0.021554288819799
7 0.0169261114259912
7.25 0.0142104543187092
7.5 0.0136305280221007
7.75 0.0145709760869417
8 0.0160741003633896
8.25 0.0174801837786718
8.5 0.0184689077714969
8.75 0.0189193818762063
9 0.0188152417493327
9.25 0.0181961279116715
9.5 0.0171326157397025
9.75 0.0157120747237059
10 0.014030076089733
10.25 0.012185127893274
10.5 0.0102761310550179
10.75 0.00840292734628621
11 0.00667151995522145
11.25 0.0052059361608045
11.5 0.00415932665140307
11.75 0.00367056058960772
12 0.00371150396997756
12.25 0.00405223904856665
12.5 0.00445874802316794
12.75 0.00479573720367865
13 0.00500341114273431
13.25 0.00506302938633529
13.5 0.00497742009791801
13.75 0.00476114040651602
14 0.0044352222723491
14.25 0.00402411567918301
14.5 0.00355379295295263
14.75 0.00305064575228743
15 0.00254111338595954
15.25 0.00205226769833133
15.5 0.00161388942165668
15.75 0.00126217932569056
16 0.00103965690922703
16.25 0.000970932484650692
16.5 0.00102295212733675
16.75 0.00112719477663137
17 0.00123128170047543
17.25 0.00130851996977123
17.5 0.00134823386404057
17.75 0.00134818458358195
18 0.00131062022976695
18.25 0.00124025638962347
18.5 0.0011431590227486
18.75 0.00102607300905409
19 0.000896010201751487
19.25 0.000760025329607863
19.5 0.000625213073381018
19.75 0.000499003589164135
};
\addplot [very thick, color0]
table {%
0 0.25
0.25 0.275127722567112
0.5 0.295996618057603
0.75 0.308706214989358
1 0.311840247943252
1.25 0.305367303843953
1.5 0.290126866203579
1.75 0.26754446647247
2 0.239420810932177
2.25 0.207753438392477
2.5 0.174603251625406
2.75 0.142040819292021
3 0.112217505114048
3.25 0.0875864275198847
3.5 0.0710218979189167
3.75 0.0646260372606557
4 0.0669517928505008
4.25 0.0735627063322769
4.5 0.0806541453685899
4.75 0.0861831169169802
5 0.0892891297812525
5.25 0.0897412645009218
5.5 0.0876435082080921
5.75 0.0832857898744142
6 0.0770623984453786
6.25 0.0694225707790637
6.5 0.0608390778191441
6.75 0.0517901906219672
7 0.042755790078319
7.25 0.0342335207114107
7.5 0.0267851625223418
7.75 0.0211029627275912
8 0.0179279639672406
8.25 0.0174462853278904
8.5 0.0187762387642775
8.75 0.0207116015254777
9 0.0224628778370989
9.25 0.0236562147696756
9.5 0.0241557645333841
9.75 0.0239503281463004
10 0.0230957950074792
10.25 0.0216851517119116
10.5 0.0198314287106421
10.75 0.0176572578273625
11 0.0152885337758137
11.25 0.0128515140373554
11.5 0.0104740328372394
11.75 0.00829297090062584
12 0.00647023875683966
12.25 0.0052043871524028
12.5 0.0046605866602886
12.75 0.00477036136064115
13 0.00522639910949152
13.25 0.00574170623188535
13.5 0.00615643955520576
13.75 0.00640167843086352
14 0.00645736176662016
14.25 0.00632915141984933
14.5 0.00603677349497339
14.75 0.00560768606171051
15 0.00507339249814051
15.25 0.00446714734359619
15.5 0.00382263337006838
15.75 0.00317361436286836
16 0.00255483760615956
16.25 0.00200493070805086
16.5 0.00157121761090857
16.75 0.00130781316654538
17 0.00124011405640225
17.25 0.00131786227957761
17.5 0.00145340543027704
17.75 0.00158370406648716
18 0.00167752165832087
18.25 0.00172275493659276
18.5 0.00171734358327031
18.75 0.00166459004330421
19 0.00157075151342856
19.25 0.00144369508067367
19.5 0.00129208472584089
19.75 0.00112488005443878
};
\addplot [very thick, color1]
table {%
0 0.25
0.25 0.26388213685871
0.5 0.290845690824212
0.75 0.317512268139835
1 0.337024424264418
1.25 0.34642583124259
1.5 0.344858667080699
1.75 0.332728080983267
2 0.311300527463113
2.25 0.282440052700175
2.5 0.248377015752316
2.75 0.211505747976867
3 0.174253564973038
3.25 0.139079020145309
3.5 0.108651512195398
3.75 0.086113673716931
4 0.0745149500402806
4.25 0.0739830635282827
4.5 0.0802776520352115
4.75 0.0884736193115441
5 0.0955580732095342
5.25 0.100149713854009
5.5 0.101772989889903
5.75 0.100429412523245
6 0.0963822631098907
6.25 0.0900428569095634
6.5 0.0819037343894234
6.75 0.0724961840596411
7 0.0623639285269191
7.25 0.0520520840947656
7.5 0.0421159382795417
7.75 0.0331601101865258
8 0.0259127167937538
8.25 0.0212361248450716
8.5 0.0196746404334655
8.75 0.0206277953135433
9 0.0227142625555148
9.25 0.0248417448271471
9.5 0.0264449460888647
9.75 0.0272941825443341
10 0.02733701243733
10.25 0.0266151311521676
10.5 0.0252220059886538
10.75 0.023279550544874
11 0.0209241250703217
11.25 0.0182978706213571
11.5 0.0155440349287346
11.75 0.0128065255106289
12 0.0102355276156462
12.25 0.00800228223051701
12.5 0.00631877681549057
12.75 0.00540642069516591
13 0.00530351167630097
13.25 0.00572740063033964
13.5 0.00631621468084581
13.75 0.00683943634718442
14 0.00718942612849618
14.25 0.00732802361181917
14.5 0.00725345332507747
14.75 0.00698355233193998
15 0.00654699536912432
15.25 0.00597828208922853
15.5 0.00531466366359909
15.75 0.00459429129987374
16 0.00385540744200384
16.25 0.00313680582370287
16.5 0.00248022667904774
16.75 0.00193531520648917
17 0.00156267706781338
17.25 0.00141018822896242
17.5 0.00145197767038553
17.75 0.0015931374988244
18 0.00174861242300954
18.25 0.00187187081831562
18.5 0.00194311300153582
18.75 0.0019568261250109
19 0.00191503685471438
19.25 0.00182389205543436
19.5 0.00169180694920319
19.75 0.00152837534536541
};
\addplot [very thick, color2]
table {%
0 0.25
0.25 0.22610189322568
0.5 0.23345464367342
0.75 0.256336978359261
1 0.281196468996249
1.25 0.300673988716484
1.5 0.31150462040685
1.75 0.312615649198184
2 0.304186158340492
2.25 0.287200796727942
2.5 0.263187178924787
2.75 0.234011542371802
3 0.201706205029074
3.25 0.168348183840861
3.5 0.136027189881496
3.75 0.106949221179043
4 0.0836789947531568
4.25 0.0691202053572596
4.5 0.064842925860858
4.75 0.068502047979193
5 0.0755070263665067
5.25 0.0824190696397129
5.5 0.087499421769367
5.75 0.0900551274664893
6 0.0899442840176421
6.25 0.0873208027078575
6.5 0.0825039480917486
6.75 0.0759052577768899
7 0.0679835106893016
7.25 0.059216192418815
7.5 0.0500841649318217
7.75 0.041071393482305
8 0.0326867708832648
8.25 0.0255179959205791
8.5 0.0202906826706242
8.75 0.0177012792194983
9 0.0177120196719971
9.25 0.0192720785572101
9.5 0.0212272440283064
9.75 0.0228880762539033
10 0.0239437308374168
10.25 0.024291715627786
10.5 0.0239388485403988
10.75 0.0229512530734555
11 0.0214278663589317
11.25 0.019485082137389
11.5 0.0172472742246037
11.75 0.0148411749014939
12 0.0123937454726775
12.25 0.010034512372218
12.5 0.0079048165799578
12.75 0.00617526729945868
13 0.00504945985498867
13.25 0.00466089483705347
13.5 0.00487515974087098
13.75 0.00536624154959322
14 0.0058720547501565
14.25 0.0062559684737513
14.5 0.00646210866761003
14.75 0.00647730496471261
15 0.00631102931986465
15.25 0.00598518307525376
15.5 0.00552849053089073
15.75 0.00497314509311605
16 0.00435273749401953
16.25 0.00370113551166095
16.5 0.00305235291803632
16.75 0.0024418212256632
17 0.00190981549999808
17.25 0.00150617063282712
17.5 0.00128370934916467
17.75 0.0012538425991542
18 0.00135159497166899
18.25 0.00149077202310702
18.5 0.00161566048233757
18.75 0.00170005130332918
19 0.0017345282417368
19.25 0.00171846842328099
19.5 0.00165598448800868
19.75 0.0015538014820005
};
\addplot [very thick, color3]
table {%
0 0.25
0.25 0.196152853389942
0.5 0.164463353507147
0.75 0.157688197564014
1 0.168575635532211
1.25 0.185966036831152
1.5 0.202220418182775
1.75 0.213601125414119
2 0.21863936067937
2.25 0.217072511031221
2.5 0.209320814637173
2.75 0.196218547679851
3 0.178851203154328
3.25 0.158440814602514
3.5 0.136264086399988
3.75 0.113607191298293
4 0.091772509461287
4.25 0.0721628900430279
4.5 0.0564475712433169
4.75 0.0465541775046324
5 0.0435652081224604
5.25 0.0459557617050589
5.5 0.0506465355181435
5.75 0.0553051800035733
6 0.058747347262935
6.25 0.0605001284298058
6.5 0.0604653926009431
6.75 0.0587451543142851
7 0.0555523014199976
7.25 0.0511609388811949
7.5 0.0458762121051568
7.75 0.0400154676116012
8 0.0338982316873808
8.25 0.0278459242008507
8.5 0.02219574282774
8.75 0.0173356002784456
9 0.0137455532215339
9.25 0.0119026333060809
9.5 0.0118237756540295
9.75 0.0128318817531212
10 0.0141415477237137
10.25 0.0152726561931608
10.5 0.0160059072956037
10.75 0.0162668654309323
11 0.0160570212674729
11.25 0.0154189516845631
11.5 0.0144179555251518
11.75 0.0131314555797597
12 0.0116424968707905
12.25 0.0100358502985067
12.5 0.00839645072846772
12.75 0.00681075027749623
13 0.00537257406153188
13.25 0.00419458605634529
13.5 0.00341256273127707
13.75 0.00312280977705617
14 0.00324764371638032
14.25 0.00357108062196393
14.5 0.00391245659599763
14.75 0.00417591007337021
15 0.00432161520656284
15.25 0.00433949255421457
15.5 0.00423520253167554
15.75 0.0040230313436336
16 0.00372200691301398
16.25 0.00335358751161665
16.5 0.002940233244521
16.75 0.00250461669720798
17 0.00206948887089395
17.25 0.00165846327255267
17.5 0.0012982102691156
17.75 0.00102166467492728
18 0.000864621549907901
18.25 0.000837572561889453
18.5 0.000899576921646498
18.75 0.000992374753074409
19 0.00107722371705294
19.25 0.00113563525474074
19.5 0.00116080878967709
19.75 0.0011520651575879
};
\addplot [very thick, color4]
table {%
0 0.25
0.25 0.210631317912405
0.5 0.171650626349752
0.75 0.135750474946742
1 0.105937635642306
1.25 0.0856966048502568
1.5 0.0776421089438703
1.75 0.0801893296673298
2 0.0880485994866863
2.25 0.0965897424778456
2.5 0.103303196587295
2.75 0.107122421690033
3 0.107750945422226
3.25 0.105302607491485
3.5 0.100120190231228
3.75 0.0926761306990646
4 0.083512020827523
4.25 0.0731996577595697
4.5 0.0623182486177593
4.75 0.051448883692665
5 0.0411934523454439
5.25 0.03223034110947
5.5 0.0253942663362314
5.75 0.021577494146488
6 0.0210022650731313
6.25 0.0226054576119423
6.5 0.0249354000161839
6.75 0.0270426779317305
7 0.028477855834002
7.25 0.0290776778827941
7.5 0.028828785176612
7.75 0.0277985255340802
8 0.0260989002785352
8.25 0.0238660516764085
8.5 0.0212476920423783
8.75 0.0183954660740484
9 0.015461457062626
9.25 0.0125996562564197
9.5 0.00997498950417153
9.75 0.00778259283859511
10 0.0062616616484508
10.25 0.00561046888032495
10.5 0.00574510975681681
10.75 0.00629504374593567
11 0.00691527663026628
11.25 0.00741391504343846
11.5 0.0077082946248193
11.75 0.00777442676749024
12 0.00761921783603755
12.25 0.00726646498905943
12.5 0.00674925564364062
12.75 0.00610553369905147
13 0.00537534268411651
13.25 0.00459924013478143
13.5 0.0038178796132644
13.75 0.00307310869922875
14 0.00241146745581223
14.25 0.00188998568585284
14.5 0.00157380555651314
14.75 0.00149320826470213
15 0.001587320886064
15.25 0.00175061801366322
15.5 0.00190737516532133
15.75 0.00202011051347772
16 0.00207431936022911
16.25 0.00206755671982038
16.5 0.00200381879664077
16.75 0.00189064824527681
17 0.00173752414401947
17.25 0.00155487991398013
17.5 0.00135350394914902
17.75 0.00114423872727879
18 0.000938009617953413
18.25 0.000746347045424372
18.5 0.000582631671045005
18.75 0.000463430346499492
19 0.000404609768532636
19.25 0.000405154312221636
19.5 0.000440951075367578
19.75 0.000485655226536968
};
\addplot [very thick, color5]
table {%
0 0.166666666666667
0.25 0.182147520986688
0.5 0.193684971954423
0.75 0.19960526802389
1 0.199495979422297
1.25 0.193641143499568
1.5 0.182736805020276
1.75 0.167726912882227
2 0.149692924189831
2.25 0.12977607087682
2.5 0.109130714598342
2.75 0.08891912456296
3 0.0703679685416761
3.25 0.0549045230536419
3.5 0.0442508408253367
3.75 0.0397921258870492
4 0.0408631376605293
4.25 0.0448071193382802
4.5 0.0492004844286342
4.75 0.0527068154365768
5 0.0547531820753583
5.25 0.0551757820569647
5.5 0.0540269689227887
5.75 0.051477743549971
6 0.0477649941743197
6.25 0.043160158129278
6.5 0.037949645866031
6.75 0.0324235695921274
7 0.0268727607111092
7.25 0.021597145050189
7.5 0.0169316121813347
7.75 0.0132872728774713
8 0.0111259796186652
8.25 0.0106353571102658
8.5 0.0113496622410606
8.75 0.0125198373548847
9 0.0136233237125895
9.25 0.0144048565128481
9.5 0.0147673207500777
9.75 0.0146965060480103
10 0.0142225828252431
10.25 0.0134002065427104
10.5 0.0122973389080231
10.75 0.0109884741703064
11 0.00955046985921808
11.25 0.00806051286288392
11.5 0.00659647332530197
11.75 0.00524087063678706
12 0.00408999861626078
12.25 0.00326289808725172
12.5 0.00286950940511552
12.75 0.00289262295846755
13 0.00315514710926029
13.25 0.00347280871545011
13.5 0.00373805532749965
13.75 0.00390306986129429
14 0.00395263187575602
14.25 0.00388863144098727
14.5 0.00372225613073029
14.75 0.00346982927959714
15 0.00315038755882036
15.25 0.00278418496036661
15.5 0.00239182049048721
15.75 0.00199393192381473
16 0.001611624693968
16.25 0.0012680568009886
16.5 0.000991322680705374
16.75 0.000814597549318527
17 0.000757919709124613
17.25 0.000796731636266535
17.5 0.000877658193397386
17.75 0.000959251846252577
18 0.00102023888409784
18.25 0.0010520506847234
18.5 0.00105281583423411
18.75 0.00102422171705027
19 0.000969911615133906
19.25 0.000894599848220376
19.5 0.000803542604810272
19.75 0.000702209464485946
};
\addplot [very thick, color6]
table {%
0 0.166666666666667
0.25 0.16902244048375
0.5 0.184718118809403
0.75 0.203171188503504
1 0.218331886983932
1.25 0.227487801136436
1.5 0.229731690828058
1.75 0.225137948005115
2 0.214360591040774
2.25 0.198412623695257
2.5 0.178520105289919
2.75 0.156014793203898
3 0.132261182258096
3.25 0.108628968918796
3.5 0.0865331539818214
3.75 0.0675684674705221
4 0.0536696267295883
4.25 0.0466991415997466
4.5 0.0466144055862203
4.75 0.0506777029063398
5 0.0558300087441971
5.25 0.0602305398725155
5.5 0.0630448053266936
5.75 0.0639935641227152
6 0.0630893308957104
6.25 0.0605035885414161
6.5 0.0564964842180785
6.75 0.0513758646347705
7 0.045471719628673
7.25 0.0391207797786193
7.5 0.0326604624757218
7.75 0.0264348566699085
8 0.0208192359460556
8.25 0.0162663184607201
8.5 0.013314644643544
8.75 0.0123112653400754
9 0.012891494938418
9.25 0.0141926431166425
9.5 0.0155265127992344
9.75 0.0165355463833618
10 0.0170739726174494
10.25 0.0171079337800599
10.5 0.0166629570118409
10.75 0.0157971936773244
11 0.014586704338474
11.25 0.0131166460633566
11.5 0.0114758348145242
11.75 0.00975379636503613
12 0.00804045184549862
12.25 0.00642955090847252
12.5 0.00502781011031237
12.75 0.00396730996413289
13 0.00338715217062434
13.25 0.00331474036888901
13.5 0.00357609682941961
13.75 0.00394408886603761
14 0.00427286272534435
14.25 0.004494031272397
14.5 0.00458316761598146
14.75 0.00453886610110326
15 0.00437212069575014
15.25 0.004100775868131
15.5 0.00374636062615789
15.75 0.00333215571176052
16 0.00288201836939081
16.25 0.00241986589698717
16.5 0.00196994608329083
16.75 0.00155830895039813
17 0.00121588479431653
17.25 0.000980488277110574
17.5 0.000882535395307683
17.75 0.000906914351945672
18 0.000994612725117878
18.25 0.00109202021273648
18.5 0.00116963578754674
18.75 0.00121484872861508
19 0.00122409004996613
19.25 0.0011985643298705
19.5 0.00114208127321173
19.75 0.00105988661879015
};
\addplot [very thick, color0]
table {%
0 0.166666666666667
0.25 0.132503210836291
0.5 0.115672928812155
0.75 0.115869328892679
1 0.126128631260663
1.25 0.138909090027686
1.5 0.149729178279562
1.75 0.15655656951385
2 0.158702073352234
2.25 0.156190863774875
2.5 0.149449980253762
2.75 0.139141274177626
3 0.126059440100868
3.25 0.111063724151481
3.5 0.0950340101984201
3.75 0.0788528442886529
4 0.0634233821216576
4.25 0.0497414082764901
4.5 0.0390160421636797
4.75 0.0326056750457636
5 0.0310884605767357
5.25 0.0331335847839666
5.5 0.0365477052094944
5.75 0.0397865193988117
6 0.0420921184848072
6.25 0.0431740214645825
6.5 0.0429870224470098
6.75 0.0416173368633704
7 0.0392239233164161
7.25 0.0360054927079857
7.5 0.0321803980887088
7.75 0.0279742328252328
8 0.023613634566329
8.25 0.0193271818550863
8.5 0.0153568714243708
8.75 0.011984764336582
9 0.00955938647314081
9.25 0.00840324292630885
9.5 0.00846767154580208
9.75 0.00923503912487955
10 0.0101652317900259
10.25 0.010942633669436
10.5 0.011426816845831
10.75 0.0115729851704937
11 0.0113865453478407
11.25 0.0109001926388321
11.5 0.0101616932390104
11.75 0.00922678283678818
12 0.00815478287976065
12.25 0.00700601931168232
12.5 0.00584093036721494
12.75 0.00472131939358282
13 0.00371498432573117
13.25 0.00290415310316277
13.5 0.00238597289234613
13.75 0.0022193093202027
14 0.00233253940726432
14.25 0.00256939902246537
14.5 0.00280839356798348
14.75 0.0029871305679191
15 0.00308045993515653
15.25 0.00308287270109083
15.5 0.00299929513425506
15.75 0.00284038879159779
16 0.00261995683803372
16.25 0.00235338900760117
16.5 0.00205670061440294
16.75 0.00174602173677695
17 0.00143756333211034
17.25 0.00114827125713866
17.5 0.000897518088907702
17.75 0.000709296921338264
18 0.000608433002335139
18.25 0.000598508989468692
18.5 0.000647079510666427
18.75 0.000713519986652567
19 0.000772190673575103
19.25 0.000811175744295943
19.5 0.000826290641856081
19.75 0.000817395955239165
};
\addplot [very thick, color1]
table {%
0 0.166666666666667
0.25 0.148767600780629
0.5 0.128995023653086
0.75 0.108494106320388
1 0.088417762996474
1.25 0.0699824617642379
1.5 0.0546030376736521
1.75 0.0439878287469163
2 0.0395195001640992
2.25 0.0405545862391189
2.5 0.0444609641455813
2.75 0.0488257177726332
3 0.0523154696829581
3.25 0.0543577250685045
3.5 0.0547879630476172
3.75 0.0536571590349647
4 0.051134536133917
4.25 0.0474549844517245
4.5 0.0428878483990441
4.75 0.0377174269338684
5 0.0322317416855119
5.25 0.0267195228415732
5.5 0.0214784161860759
5.75 0.0168405328209
6 0.013213421665771
6.25 0.0110559648861663
6.5 0.0105581519241859
6.75 0.0112617067837092
7 0.0124225705616239
7.25 0.0135197894080233
7.5 0.0142984571362312
7.75 0.0146613618580301
8 0.014593984373519
8.25 0.0141260536252706
8.5 0.0133117110986956
8.75 0.0122183759444126
9 0.0109199791729718
9.25 0.0094928262498544
9.5 0.00801355257960853
9.75 0.00655946301945453
10 0.00521243389487485
10.25 0.00406793872331906
10.5 0.00324403511070059
10.75 0.0028502776430541
11 0.00287084290396207
11.25 0.00313055621206812
11.5 0.00344603664988448
11.75 0.00370997529977935
12 0.00387458819332964
12.25 0.00392460318652871
12.5 0.00386181176371347
12.75 0.00369727453468023
13 0.00344717297665503
13.25 0.00313039560763549
13.5 0.00276704983468992
13.75 0.00237758763642502
14 0.00198249680967218
14.25 0.00160272552950915
14.5 0.00126124620728517
14.75 0.000985897139280597
15 0.000809627869437354
15.25 0.000752550384175082
15.5 0.00079061234805388
15.75 0.00087084465538458
16 0.000951948372327681
16.25 0.0010126858875691
16.5 0.00104448247048793
16.75 0.00104544907829672
17 0.00101724286363343
17.25 0.000963472295296241
17.5 0.000888815235330068
17.75 0.000798486671233481
18 0.000697917778714054
18.25 0.000592578363683717
18.5 0.000487964301661766
18.75 0.000389820283757692
19 0.000304709324660046
19.25 0.000240760889063913
19.5 0.000206402645780183
19.75 0.000202906854991974
};
\addplot [very thick, color2]
table {%
0 0.166666666666667
0.25 0.182147520986754
0.5 0.193684971954488
0.75 0.19960526802394
1 0.199495979422353
1.25 0.193641143499641
1.5 0.182736805020324
1.75 0.167726912882275
2 0.149692924189873
2.25 0.129776070876857
2.5 0.109130714598373
2.75 0.088919124562985
3 0.0703679685416964
3.25 0.0549045230536592
3.5 0.0442508408253512
3.75 0.0397921258870614
4 0.0408631376605393
4.25 0.0448071193382913
4.5 0.0492004844286467
4.75 0.0527068154365911
5 0.0547531820753742
5.25 0.0551757820569819
5.5 0.0540269689228061
5.75 0.0514777435499878
6 0.0477649941743359
6.25 0.0431601581292927
6.5 0.0379496458660505
6.75 0.0324235695921439
7 0.0268727607111222
7.25 0.0215971450501987
7.5 0.0169316121813422
7.75 0.0132872728774773
8 0.0111259796186702
8.25 0.0106353571102693
8.5 0.0113496622410621
8.75 0.0125198373548862
9 0.0136233237125973
9.25 0.0144048565128368
9.5 0.0147673207500677
9.75 0.0146965060480053
10 0.0142225828252481
10.25 0.0134002065427149
10.5 0.0122973389080312
10.75 0.0109884741703133
11 0.0095504698592221
11.25 0.0080605128628921
11.5 0.0065964733253107
11.75 0.00524087063679397
12 0.00408999861626607
12.25 0.00326289808725608
12.5 0.00286950940512112
12.75 0.00289262295846799
13 0.00315514710925792
13.25 0.00347280871544887
13.5 0.00373805532750077
13.75 0.00390306986129594
14 0.00395263187576384
14.25 0.00388863144099688
14.5 0.00372225613074115
14.75 0.00346982927960715
15 0.00315038755882696
15.25 0.00278418496032947
15.5 0.00239182049044613
15.75 0.00199393192377759
16 0.00161162469392871
16.25 0.00126805680095453
16.5 0.000991322680793474
16.75 0.000814597549450169
17 0.000757919709182297
17.25 0.000796731636221175
17.5 0.000877658193316692
17.75 0.000959251846214199
18 0.0010202388842192
18.25 0.00105205068480776
18.5 0.00105281583442319
18.75 0.0010242217174073
19 0.00096991161543718
19.25 0.000894599848533738
19.5 0.000803542605235926
19.75 0.000702209465068559
};
\addplot [very thick, color3]
table {%
0 0.166666666666667
0.25 0.169022440483915
0.5 0.184718118809573
0.75 0.203171188503694
1 0.218331886984133
1.25 0.227487801136655
1.5 0.229731690828283
1.75 0.22513794800534
2 0.214360591040992
2.25 0.198412623695461
2.5 0.178520105290106
2.75 0.156014793204072
3 0.132261182258249
3.25 0.10862896891892
3.5 0.0865331539819151
3.75 0.0675684674705941
4 0.0536696267296493
4.25 0.0466991415998094
4.5 0.0466144055862927
4.75 0.0506777029064254
5 0.0558300087442881
5.25 0.0602305398726117
5.5 0.0630448053267905
5.75 0.0639935641228097
6 0.063089330895806
6.25 0.0605035885415322
6.5 0.0564964842182726
6.75 0.0513758646349865
7 0.0454717196289121
7.25 0.0391207797789409
7.5 0.0326604624761772
7.75 0.0264348566703024
8 0.0208192359463914
8.25 0.0162663184610501
8.5 0.0133146446437792
8.75 0.012311265340147
9 0.0128914949384121
9.25 0.0141926431166059
9.5 0.015526512799224
9.75 0.0165355463834318
10 0.0170739726175697
10.25 0.0171079337802234
10.5 0.0166629570120451
10.75 0.0157971936775129
11 0.0145867043387134
11.25 0.0131166460636133
11.5 0.011475834814794
11.75 0.00975379636529435
12 0.00804045184573502
12.25 0.00642955090867483
12.5 0.0050278101104753
12.75 0.00396730996423982
13 0.00338715217066566
13.25 0.00331474036890076
13.5 0.00357609682943046
13.75 0.00394408886604126
14 0.00427286272535815
14.25 0.00449403127240714
14.5 0.00458316761599121
14.75 0.00453886610111114
15 0.00437212069575614
15.25 0.00410077586810238
15.5 0.00374636062609987
15.75 0.00333215571166116
16 0.00288201836925435
16.25 0.00241986589685042
16.5 0.00196994608328967
16.75 0.0015583089503503
17 0.00121588479425839
17.25 0.000980488277050459
17.5 0.000882535395287086
17.75 0.000906914351940643
18 0.000994612725109856
18.25 0.00109202021276498
18.5 0.00116963578764535
18.75 0.0012148487286768
19 0.00122409005001815
19.25 0.00119856432993635
19.5 0.00114208127324451
19.75 0.00105988661880027
};
\addplot [very thick, color4]
table {%
0 0.166666666666667
0.25 0.132503210836326
0.5 0.115672928812186
0.75 0.115869328892711
1 0.126128631260697
1.25 0.138909090027724
1.5 0.149729178279603
1.75 0.156556569513891
2 0.158702073352278
2.25 0.15619086377492
2.5 0.149449980253807
2.75 0.139141274177669
3 0.126059440100909
3.25 0.111063724151518
3.5 0.0950340101984524
3.75 0.07885284428868
4 0.0634233821216799
4.25 0.0497414082765087
4.5 0.0390160421636944
4.75 0.0326056750457741
5 0.0310884605767421
5.25 0.0331335847839725
5.5 0.0365477052095013
5.75 0.0397865193988209
6 0.0420921184848177
6.25 0.0431740214646035
6.5 0.0429870224470273
6.75 0.0416173368633831
7 0.0392239233164198
7.25 0.0360054927079842
7.5 0.0321803980887057
7.75 0.0279742328252387
8 0.0236136345663386
8.25 0.0193271818550961
8.5 0.0153568714243785
8.75 0.0119847643365855
9 0.00955938647314493
9.25 0.00840324292632134
9.5 0.00846767154582721
9.75 0.00923503912491633
10 0.0101652317900704
10.25 0.0109426336694848
10.5 0.0114268168458809
10.75 0.0115729851705356
11 0.0113865453478844
11.25 0.0109001926388628
11.5 0.0101616932390275
11.75 0.00922678283679196
12 0.00815478287975728
12.25 0.00700601931167625
12.5 0.00584093036720182
12.75 0.0047213193935694
13 0.00371498432571593
13.25 0.00290415310315188
13.5 0.00238597289234323
13.75 0.00221930932020701
14 0.00233253940727395
14.25 0.00256939902247538
14.5 0.00280839356799182
14.75 0.00298713056792561
15 0.00308045993516076
15.25 0.00308287270109071
15.5 0.00299929513425211
15.75 0.00284038879159334
16 0.00261995683802708
16.25 0.00235338900758843
16.5 0.00205670061438118
16.75 0.00174602173674805
17 0.00143756333207667
17.25 0.00114827125711441
17.5 0.000897518088870192
17.75 0.000709296921315909
18 0.000608433002321792
18.25 0.000598508989465765
18.5 0.000647079510658985
18.75 0.000713519986646074
19 0.000772190673560512
19.25 0.000811175744281045
19.5 0.000826290641864005
19.75 0.000817395955243031
};
\addplot [very thick, color5]
table {%
0 0.166666666666667
0.25 0.148767600780394
0.5 0.128995023652849
0.75 0.108494106320167
1 0.0884177629962829
1.25 0.0699824617641361
1.5 0.0546030376735408
1.75 0.0439878287467907
2 0.0395195001640267
2.25 0.0405545862389652
2.5 0.0444609641452849
2.75 0.0488257177720404
3 0.0523154696821472
3.25 0.0543577250674784
3.5 0.054787963046061
3.75 0.0536571590325913
4 0.0511345361315156
4.25 0.0474549844472327
4.5 0.042887848392567
4.75 0.0377174269272024
5 0.0322317416775568
5.25 0.0267195228323054
5.5 0.0214784161768768
5.75 0.0168405328138098
6 0.0132134216607905
6.25 0.011055964883565
6.5 0.0105581519235211
6.75 0.0112617067839609
7 0.0124225705619411
7.25 0.013519789409964
7.5 0.0142984571395072
7.75 0.0146613618619987
8 0.0145939843773109
8.25 0.0141260536285722
8.5 0.0133117111011653
8.75 0.0122183759396672
9 0.0109199791712727
9.25 0.00949282624793069
9.5 0.00801355257819299
9.75 0.00655946301684586
10 0.00521243389144727
10.25 0.00406793872071849
10.5 0.00324403510917479
10.75 0.00285027764693265
11 0.00287084291226923
11.25 0.00313055622257638
11.5 0.00344603666117129
11.75 0.00370997531126418
12 0.00387458820245931
12.25 0.00392460319791749
12.5 0.00386181177411106
12.75 0.00369727454193009
13 0.00344717298092561
13.25 0.00313039561100139
13.5 0.00276704983983371
13.75 0.00237758764609224
14 0.00198249682550877
14.25 0.00160272554239226
14.5 0.00126124625740359
14.75 0.000985897203099257
15 0.000809627905910767
15.25 0.000752550402080788
15.5 0.000790612373699871
15.75 0.000870844692769684
16 0.000951948398118546
16.25 0.00101268595075594
16.5 0.00104448252521052
16.75 0.00104544911858025
17 0.00101724291944996
17.25 0.000963472342436907
17.5 0.000888815265448992
17.75 0.000798486704259677
18 0.000697917804426043
18.25 0.00059257838720104
18.5 0.000487964319538758
18.75 0.000389820296008251
19 0.000304709336711597
19.25 0.000240760895032908
19.5 0.000206402655785684
19.75 0.000202906867203793
};
\addplot [very thick, color6]
table {%
0 0.0833333333333333
0.25 0.0869181378872508
0.5 0.0956273273601995
0.75 0.104704455300165
1 0.111646612920039
1.25 0.115414228367849
1.5 0.115743915807324
1.75 0.112788299387332
2 0.106933738583888
2.25 0.0986992801126037
2.5 0.0886742270060263
2.75 0.077477710853335
3 0.0657355099058985
3.25 0.0540761411271089
3.5 0.0431548237036214
3.75 0.0337184601209449
4 0.0266877219588392
4.25 0.0229951835208589
4.5 0.0227277215803404
4.75 0.0246183883822644
5 0.0271395397357479
5.25 0.0293419980905626
5.5 0.0307882358154858
5.75 0.0313268386799522
6 0.03095633925221
6.25 0.0297563749901317
6.5 0.0278514924975321
6.75 0.0253903439515142
7 0.022532863474709
7.25 0.0194425617176837
7.5 0.0162832978933449
7.75 0.0132216273046335
8 0.010437719470006
8.25 0.00814718259838662
8.5 0.00661098953915945
8.75 0.00602147888606076
9 0.00624191438194702
9.25 0.00685907341915034
9.5 0.00751949708552333
9.75 0.00803376418022746
10 0.00832253833422725
10.25 0.00836503695186933
10.5 0.00817143060158233
10.75 0.00776884265960454
11 0.0071937246404125
11.25 0.00648732765221645
11.5 0.00569290423209192
11.75 0.00485415023066347
12 0.00401489443904385
12.25 0.0032205417946068
12.5 0.00252219996315081
12.75 0.00198289341385732
13 0.00167195949126522
13.25 0.00161244860295614
13.5 0.0017282261761264
13.75 0.00190656120475561
14 0.0020712983793396
14.25 0.00218581120359011
14.5 0.00223647062500242
14.75 0.00222168844470971
15 0.00214634091332632
15.25 0.00201886951141223
15.5 0.00184964706560882
15.75 0.00164998196824094
16 0.00143151070876724
16.25 0.00120591050463012
16.5 0.000984981660746828
16.75 0.000781290008067445
17 0.000609595584824068
17.25 0.000488092809129141
17.5 0.00043284022635957
17.75 0.000439506437665307
18 0.000480458091028636
18.25 0.000528396473019645
18.5 0.000567741584000879
18.75 0.00059166409492505
19 0.00059807265270638
19.25 0.000587367237163937
19.5 0.000561302198972031
19.75 0.00052238265304437
};
\addplot [very thick, color0]
table {%
0 0.0833333333333333
0.25 0.0715729785351204
0.5 0.0596342897863484
0.75 0.0481658768465763
1 0.0378768194406023
1.25 0.0296237542417313
1.5 0.0244082709591373
1.75 0.0228058473683607
2 0.0240348177857876
2.25 0.0264850497481116
2.5 0.0289284563950579
2.75 0.030739722052009
3 0.0316688331565067
3.25 0.0316634677694136
3.5 0.0307765573297317
3.75 0.0291190303454968
4 0.0268336325587998
4.25 0.0240791449354813
4.5 0.0210206209489796
4.75 0.0178242265277207
5 0.0146570582706499
5.25 0.0116941879023126
5.5 0.00913652365232378
5.75 0.00723294579365018
6 0.00623561849908572
6.25 0.00616649664970463
6.5 0.00668088526362887
6.75 0.00736484408177853
7 0.00796162067418466
7.25 0.00835299203841739
7.5 0.00849817967288654
7.75 0.00839695776198252
8 0.00807102004967496
8.25 0.00755417547254469
8.5 0.00688670155877663
8.75 0.00611190480515313
9 0.00527404584383613
9.25 0.00441746914139336
9.5 0.0035872605687193
9.75 0.00283218645790631
10 0.00221061271668576
10.25 0.00179324285174374
10.5 0.00163242412138618
10.75 0.00169151517178747
11 0.00185860759299931
11.25 0.00203771457271677
11.5 0.00217733431751635
11.75 0.00225587570138114
12 0.00226766152332441
12.25 0.00221542621130458
12.5 0.00210651019702769
12.75 0.00195078701006918
13 0.00175943367673238
13.25 0.00154417147156559
13.5 0.00131684263127501
13.75 0.00108932561913561
14 0.000873922344258769
14.25 0.000684472254644548
14.5 0.000538041817390107
14.75 0.000453440407526693
15 0.000437027072533266
15.25 0.000468268821274232
15.5 0.000516590664782294
15.75 0.00056129185370746
16 0.000592407572417347
16.25 0.000606221228910452
16.5 0.000602292975662648
16.75 0.000581938141319333
17 0.000547442322936444
17.25 0.000501616025917511
17.5 0.000447523588177736
17.75 0.000388319122955686
18 0.00032716778188854
18.25 0.000267268973234452
18.5 0.000212027542233985
18.75 0.000165438137845993
19 0.000132426972983047
19.25 0.000117359865361199
19.5 0.000119099559311531
19.75 0.000130174391546519
};
\addplot [very thick, color1]
table {%
0 0.0833333333333333
0.25 0.0869181378873094
0.5 0.0956273273602655
0.75 0.104704455300243
1 0.111646612920126
1.25 0.11541422836794
1.5 0.115743915807426
1.75 0.112788299387435
2 0.106933738583992
2.25 0.0986992801126961
2.5 0.0886742270061123
2.75 0.0774777108533936
3 0.0657355099059505
3.25 0.0540761411271385
3.5 0.0431548237036317
3.75 0.0337184601209346
4 0.0266877219588391
4.25 0.0229951835208729
4.5 0.0227277215803644
4.75 0.0246183883822942
5 0.0271395397357715
5.25 0.0293419980905808
5.5 0.0307882358154911
5.75 0.031326838679958
6 0.0309563392522085
6.25 0.0297563749902229
6.5 0.0278514924976399
6.75 0.025390343951625
7 0.0225328634747896
7.25 0.0194425617177506
7.5 0.0162832978933857
7.75 0.0132216273046724
8 0.0104377194700338
8.25 0.008147182598372
8.5 0.00661098953913605
8.75 0.00602147888597863
9 0.00624191438186627
9.25 0.00685907341907756
9.5 0.00751949708542125
9.75 0.00803376418012645
10 0.00832253833412075
10.25 0.00836503695172027
10.5 0.00817143060151987
10.75 0.00776884265978888
11 0.00719372464085529
11.25 0.00648732765294656
11.5 0.00569290423303452
11.75 0.0048541502312479
12 0.00401489443929284
12.25 0.0032205417945761
12.5 0.00252219996290494
12.75 0.00198289341347657
13 0.00167195949113436
13.25 0.00161244860306137
13.5 0.00172822617611078
13.75 0.00190656120457915
14 0.00207129837890656
14.25 0.00218581120309766
14.5 0.0022364706245706
14.75 0.0022216884442369
15 0.00214634091289073
15.25 0.00201886951093591
15.5 0.00184964706517257
15.75 0.00164998196783613
16 0.00143151070849386
16.25 0.00120591050442159
16.5 0.000984981660572283
16.75 0.000781290007923738
17 0.000609595584705351
17.25 0.000488092809045245
17.5 0.000432840226314914
17.75 0.000439506437679323
18 0.000480458091057999
18.25 0.000528396473052415
18.5 0.000567741584029801
18.75 0.000591664094956457
19 0.000598072652708956
19.25 0.000587367237274661
19.5 0.000561302199269214
19.75 0.000522382653379172
};
\addplot [very thick, color2]
table {%
0 0.0833333333333333
0.25 0.0715729785352591
0.5 0.0596342897864702
0.75 0.0481658768467015
1 0.0378768194406848
1.25 0.0296237542417976
1.5 0.0244082709591833
1.75 0.0228058473683921
2 0.024034817785817
2.25 0.026485049748144
2.5 0.0289284563950946
2.75 0.0307397220520507
3 0.0316688331565549
3.25 0.0316634677694666
3.5 0.0307765573297878
3.75 0.029119030345553
4 0.0268336325588538
4.25 0.024079144935523
4.5 0.0210206209490198
4.75 0.0178242265277551
5 0.0146570582706801
5.25 0.0116941879023368
5.5 0.00913652365234308
5.75 0.00723294579366501
6 0.00623561849909859
6.25 0.00616649664971779
6.5 0.00668088526364393
6.75 0.00736484408179479
7 0.00796162067421723
7.25 0.00835299203848312
7.5 0.00849817967293889
7.75 0.00839695776200441
8 0.00807102004965887
8.25 0.00755417547250116
8.5 0.00688670155874246
8.75 0.0061119048050835
9 0.00527404584377435
9.25 0.00441746914134017
9.5 0.00358726056864905
9.75 0.0028321864578465
10 0.00221061271664031
10.25 0.00179324285171274
10.5 0.0016324241213719
10.75 0.00169151517179145
11 0.00185860759302877
11.25 0.00203771457277306
11.5 0.00217733431759295
11.75 0.00225587570147089
12 0.0022676615234197
12.25 0.00221542621132821
12.5 0.00210651019706039
12.75 0.00195078701004492
13 0.00175943367645457
13.25 0.00154417147117347
13.5 0.00131684263088525
13.75 0.00108932561880897
14 0.000873922343966913
14.25 0.000684472254261324
14.5 0.000538041817200439
14.75 0.000453440407409208
15 0.000437027072411277
15.25 0.000468268821108811
15.5 0.000516590664601462
15.75 0.000561291853504899
16 0.000592407572146448
16.25 0.000606221228641498
16.5 0.000602292975388604
16.75 0.000581938140971335
17 0.000547442322548101
17.25 0.000501616025468771
17.5 0.000447523587683846
17.75 0.000388319122423186
18 0.000327167781330952
18.25 0.000267268973265938
18.5 0.000212027541762113
18.75 0.000165438137992928
19 0.000132426972997967
19.25 0.000117359865218688
19.5 0.000119099559171026
19.75 0.000130174391430609
};
\end{axis}

\end{tikzpicture}

%% file: pics/oms-alpha0.4-rho0.55-xzrad2.0.tikz
%
%
%
\begin{tikzpicture}

\definecolor{color1}{rgb}{0.247058823529412,0.623529411764706,0.749019607843137}
\definecolor{color0}{rgb}{0.623529411764706,0.811764705882353,0.874509803921569}
\definecolor{color2}{rgb}{0,0.435294117647059,0.623529411764706}

\begin{axis}[
xlabel={$t^*$},
xtick={0, 4, 8, 12, 16, 20},
ylabel={None},
xmin=0, xmax=20,
ymin=-1, ymax=9,
width=\figurewidth,
height=\figureheight,
tick align=outside,
xmajorgrids,
x grid style={white!80.0!black},
ymajorgrids,
y grid style={white!80.0!black},
axis line style={white!80.0!black},
legend style={at={(0.97,0.97)}, anchor=north east},
legend entries={{$-\omega^*(t^*)$},{$-\omega$},{zero}}
]
\addplot [very thick, color0]
table {%
0 inf
0.25 8.14876748597566
0.5 4.9014917783816
0.75 4.19642980504987
1 4.76533560959222
1.25 4.617021748501
1.5 4.66326430860978
1.75 4.48611745835033
2 3.88749633092808
2.25 2.94231660089252
2.5 2.85896986779729
2.75 2.34728579279374
3 2.27284231480366
3.25 1.9028554259871
3.5 1.77556751178588
3.75 2.54989485645652
4 2.44859775576181
4.25 2.26452128769242
4.5 2.20931009942094
4.75 1.51697483508731
5 1.74313407263261
5.25 1.71193700374242
5.5 1.37248971740569
5.75 1.3068799348453
6 1.79293095640463
6.25 1.65972645360381
6.5 1.59571272515389
6.75 1.52267473492922
7 1.56601660670552
7.25 1.52885270759432
7.5 1.1332640126776
7.75 1.2732548655622
8 1.23187506625459
8.25 1.31933258452377
8.5 1.24591259456215
8.75 1.09591277406781
9 1.05026139999017
9.25 1.36175530963344
9.5 1.31727976501131
9.75 0.910013536064911
10 0.879331106960951
10.25 0.833121917260301
10.5 1.07356474911384
10.75 1.04317523323984
11 1.0415725024497
11.25 1.01952760159181
11.5 0.810217630680474
11.75 0.742694638930909
12 1.17454356030052
12.25 0.888912308139466
12.5 0.87518461177475
12.75 0.948431445774409
13 0.924995973090703
13.25 0.926036322206632
13.5 0.917141243367143
13.75 0.73244476197865
14 0.672833075253658
14.25 0.751652451221986
14.5 0.770247877198308
14.75 0.994572985514296
15 0.957536037818044
15.25 0.932603613754684
15.5 0.827303478319219
15.75 0.797988178214079
16 0.651981262135684
16.25 0.715457922256564
16.5 0.797902253229285
16.75 0.733480450463841
17 0.679969199801724
17.25 0.660108056185824
17.5 0.969459129357811
17.75 0.646986632782335
18 0.60211393663106
18.25 0.625542122700861
18.5 0.591092657754973
18.75 0.589406731113264
19 0.546561274274513
19.25 0.656627188555858
19.5 0.750367363667089
19.75 1.1090257720936
};
\addplot [very thick, color1]
table {%
0 -0.3
0.25 -0.3
0.5 -0.3
0.75 -0.3
1 -0.3
1.25 -0.3
1.5 -0.3
1.75 -0.3
2 -0.3
2.25 -0.3
2.5 -0.3
2.75 -0.3
3 -0.3
3.25 -0.3
3.5 -0.3
3.75 -0.3
4 -0.3
4.25 -0.3
4.5 -0.3
4.75 -0.3
5 -0.3
5.25 -0.3
5.5 -0.3
5.75 -0.3
6 -0.3
6.25 -0.3
6.5 -0.3
6.75 -0.3
7 -0.3
7.25 -0.3
7.5 -0.3
7.75 -0.3
8 -0.3
8.25 -0.3
8.5 -0.3
8.75 -0.3
9 -0.3
9.25 -0.3
9.5 -0.3
9.75 -0.3
10 -0.3
10.25 -0.3
10.5 -0.3
10.75 -0.3
11 -0.3
11.25 -0.3
11.5 -0.3
11.75 -0.3
12 -0.3
12.25 -0.3
12.5 -0.3
12.75 -0.3
13 -0.3
13.25 -0.3
13.5 -0.3
13.75 -0.3
14 -0.3
14.25 -0.3
14.5 -0.3
14.75 -0.3
15 -0.3
15.25 -0.3
15.5 -0.3
15.75 -0.3
16 -0.3
16.25 -0.3
16.5 -0.3
16.75 -0.3
17 -0.3
17.25 -0.3
17.5 -0.3
17.75 -0.3
18 -0.3
18.25 -0.3
18.5 -0.3
18.75 -0.3
19 -0.3
19.25 -0.3
19.5 -0.3
19.75 -0.3
};
\addplot [black]
table {%
0 0
0.25 0
0.5 0
0.75 0
1 0
1.25 0
1.5 0
1.75 0
2 0
2.25 0
2.5 0
2.75 0
3 0
3.25 0
3.5 0
3.75 0
4 0
4.25 0
4.5 0
4.75 0
5 0
5.25 0
5.5 0
5.75 0
6 0
6.25 0
6.5 0
6.75 0
7 0
7.25 0
7.5 0
7.75 0
8 0
8.25 0
8.5 0
8.75 0
9 0
9.25 0
9.5 0
9.75 0
10 0
10.25 0
10.5 0
10.75 0
11 0
11.25 0
11.5 0
11.75 0
12 0
12.25 0
12.5 0
12.75 0
13 0
13.25 0
13.5 0
13.75 0
14 0
14.25 0
14.5 0
14.75 0
15 0
15.25 0
15.5 0
15.75 0
16 0
16.25 0
16.5 0
16.75 0
17 0
17.25 0
17.5 0
17.75 0
18 0
18.25 0
18.5 0
18.75 0
19 0
19.25 0
19.5 0
19.75 0
};
\end{axis}

\end{tikzpicture}

%% file: pics/trjnrms-alpha0.4-rho0.55-xzrad2.0.tikz
%
%
%
\begin{tikzpicture}

\definecolor{color1}{rgb}{0.247058823529412,0.623529411764706,0.749019607843137}
\definecolor{color0}{rgb}{0.623529411764706,0.811764705882353,0.874509803921569}
\definecolor{color3}{rgb}{0,0.247058823529412,0.498039215686275}
\definecolor{color2}{rgb}{0,0.435294117647059,0.623529411764706}
\definecolor{color5}{rgb}{0,0.129411764705882,0.247058823529412}
\definecolor{color4}{rgb}{0,0.0588235294117647,0.372549019607843}
\definecolor{color6}{rgb}{0,0.317647058823529,0.12156862745098}

\begin{axis}[
xlabel={t},
xtick={0, 4, 8, 12, 16, 20},
ylabel={$\|\xi(t)\|$},
xmin=0, xmax=20,
ymin=0, ymax=6,
width=\figurewidth,
height=\figureheight,
tick align=outside,
xmajorgrids,
x grid style={white!80.0!black},
ymajorgrids,
y grid style={white!80.0!black},
axis line style={white!80.0!black}
]
\addplot [very thick, color0]
table {%
0 2
0.25 1.8992018757998
0.5 1.5220798327959
0.75 1.37134849324544
1 1.45547987743251
1.25 1.60549378096618
1.5 1.74333496542061
1.75 1.80116935772249
2 1.65895093802534
2.25 1.32089282922945
2.5 1.14431006060887
2.75 1.17960565945477
3 1.29697243027527
3.25 1.42018698980374
3.5 1.5035781779881
3.75 1.49390837336461
4 1.33035469042142
4.25 1.05823656124623
4.5 0.863368469948278
4.75 0.824802054782146
5 0.885552551764683
5.25 0.976867479136022
5.5 1.05896634186931
5.75 1.10719328579094
6 1.10042570018481
6.25 1.02252497995273
6.5 0.877632864242954
6.75 0.70248936630379
7 0.549990115841229
7.25 0.461335862634559
7.5 0.446410286411127
7.75 0.479992511493066
8 0.529480149202556
8.25 0.574236455932358
8.5 0.603651145237904
8.75 0.612370216485775
9 0.598226276790354
9.25 0.561853727165193
9.5 0.506673256325849
9.75 0.438441147854154
10 0.364256274683389
10.25 0.291566715871031
10.5 0.227837733005449
10.75 0.180845241020235
11 0.157370902618332
11.25 0.15722182491589
11.5 0.170995026320153
11.75 0.188361448490092
12 0.203143003899726
12.25 0.212506886664152
12.5 0.215456166131766
12.75 0.211980884308547
13 0.202642782681108
13.25 0.188352577372519
13.5 0.170227908713134
13.75 0.149491478604275
14 0.127401300219106
14.25 0.105220231722996
14.5 0.084242844365263
14.75 0.0659055603494507
15 0.0519503518187336
15.25 0.0442134098878865
15.5 0.0431174889486605
15.75 0.0464483382680414
16 0.0512333514268591
16.25 0.055541832764398
16.5 0.0584623495667464
16.75 0.0596647694604513
17 0.0591237656610766
17.25 0.0569787648103139
17.5 0.0534609006841628
17.75 0.048851196450829
18 0.0434547130299437
18.25 0.0375847050135909
18.5 0.0315553546355734
18.75 0.0256849951143816
19 0.0203152436035373
19.25 0.0158513569639479
19.5 0.0127878665038542
19.75 0.0115200399391295
};
\addplot [very thick, color1]
table {%
0 2
0.25 2.20024480999268
0.5 2.33413920219613
0.75 2.1451924602404
1 1.77436583719064
1.25 1.82566972831982
1.5 2.01017688083746
1.75 2.19417375814557
2 2.26186065497232
2.25 1.93674966702299
2.5 1.68039390416051
2.75 1.77484915109977
3 1.95765318587887
3.25 2.12422532847596
3.5 2.15518042153196
3.75 1.81334058779695
4 1.57083133748261
4.25 1.65075469006335
4.5 1.82021368950896
4.75 1.9799509969367
5 2.0372944945901
5.25 1.79625370781003
5.5 1.45621517400484
5.75 1.43882371324782
6 1.57144333220796
6.25 1.72783754341632
6.5 1.841482625008
6.75 1.8137979520616
7 1.52080877578175
7.25 1.23893960860285
7.5 1.21524331768592
7.75 1.32228757762114
8 1.45591206906173
8.25 1.56198887595707
8.5 1.58526212560763
8.75 1.447960288334
9 1.16238347606028
9.25 0.946855207528995
9.5 0.908221133704209
9.75 0.977752737699162
10 1.07841785890924
10.25 1.16719972786525
10.5 1.21540403816931
10.75 1.19506027088909
11 1.08466078506164
11.25 0.899291105295118
11.5 0.704779097642548
11.75 0.572573498445842
12 0.532536984902443
12.25 0.562078120496891
12.5 0.619563243377333
12.75 0.676115257384351
13 0.71626893530262
13.25 0.731407365174657
13.5 0.716849704865954
13.75 0.671879559687888
14 0.600610360522179
14.25 0.511787045440663
14.5 0.416917821411634
14.75 0.328099489517571
15 0.257175290205785
15.25 0.214992896115826
15.5 0.205540830236345
15.75 0.219456690347021
16 0.242088112663667
16.25 0.26335755466648
16.5 0.278287878718339
16.75 0.284858891250874
17 0.282636952176444
17.25 0.272120567740938
17.5 0.254408646978149
17.75 0.230988166906607
18 0.203567989386511
18.25 0.173948900677432
18.5 0.143947998949873
18.75 0.115409160349521
19 0.0903377424317644
19.25 0.0711240895859672
19.5 0.0602983443433042
19.75 0.0585449377500088
};
\addplot [very thick, color2]
table {%
0 2
0.25 2.13584915615013
0.5 2.35579609905341
0.75 2.53410155827188
1 2.39100316510984
1.25 2.01126678725578
1.5 2.11487605228342
1.75 2.33270039463367
2 2.52611676490998
2.25 2.46146004998116
2.5 2.03439611749862
2.75 2.10168377867263
3 2.31616742329988
3.25 2.52019446953103
3.5 2.5138278303346
3.75 2.06710720850191
4 2.09792878896686
4.25 2.30886601045993
4.5 2.52033699124571
4.75 2.55321569559285
5 2.10449551804655
5.25 2.10548856679963
5.5 2.31413681524107
5.75 2.53083890939694
6 2.58576134741554
6.25 2.13988056850398
6.5 2.12686420432853
6.75 2.33632306914008
7 2.5565932573887
7.25 2.61579278678272
7.5 2.16754525359807
7.75 2.16615757236431
8 2.38125396107143
8.25 2.60321799093613
8.5 2.64220975402947
8.75 2.18836562956561
9 2.23073822585489
9.25 2.45685131757353
9.5 2.67649414166273
9.75 2.64827433075669
10 2.2193916928035
10.25 2.33306095569345
10.5 2.57352329225434
10.75 2.77870378321207
11 2.58016561822948
11.25 2.29963563061727
11.5 2.49120804876945
11.75 2.74287762573652
12 2.88703609693504
12.25 2.45114393832963
12.5 2.4753715165992
12.75 2.72647886846334
13 2.96613415495279
13.25 2.80948889755995
13.5 2.54733676580613
13.75 2.77860200480636
14 3.05189092170717
14.25 3.09851525511755
14.5 2.70222436751727
14.75 2.92425328350425
15 3.21835171699393
15.25 3.30008464728811
15.5 2.92653612227749
15.75 3.18853920646161
16 3.50060003352891
16.25 3.37599976116821
16.5 3.27375831648812
16.75 3.60758673206416
17 3.88429232557727
17.25 3.51492900557011
17.5 3.83601741758806
17.75 4.19808699977924
18 3.88805989390892
18.25 4.22533206195297
18.5 4.62466976199252
18.75 4.37493296309666
19 4.80751554602272
19.25 5.08845058128495
19.5 5.09470690528909
19.75 5.60699674494108
};
\addplot [very thick, color3]
table {%
0 2
0.25 1.94440790282757
0.5 2.1272048289507
0.75 2.33570397298819
1 2.45018160849311
1.25 2.13499543699197
1.5 1.89234547567133
1.75 2.02840100649441
2 2.23700779523804
2.25 2.40528854837448
2.5 2.30359257658516
2.75 1.88679342192584
3 1.92110228372849
3.25 2.11364027257632
3.5 2.30948714890845
3.75 2.37757515171746
4 2.0025382051949
4.25 1.82281033946729
4.5 1.9649806667017
4.75 2.16614843407009
5 2.32285487812885
5.25 2.22357647363678
5.5 1.80522115950592
5.75 1.8012181883879
6 1.97577634356342
6.25 2.16626073059107
6.5 2.27005024000423
6.75 2.0254329993753
7 1.69155734001156
7.25 1.74884137860054
7.5 1.92613210688825
7.75 2.1016849789566
8 2.17203628739221
8.25 1.89554360456451
8.5 1.58347682817453
8.75 1.62879826486441
9 1.79244764772533
9.25 1.95917141536228
9.5 2.04453752186861
9.75 1.86562862132875
10 1.49248469050217
10.25 1.429373190079
10.5 1.54957471592642
10.75 1.70730956387727
11 1.83300865662134
11.25 1.83803849626822
11.5 1.58315890275535
11.75 1.26840941449153
12 1.21126516119666
12.25 1.30775267719362
12.5 1.44189286515989
12.75 1.55544071620053
13 1.59533514572593
13.25 1.48418823119463
13.5 1.20735659694166
13.75 0.970120471660994
14 0.911125720868621
14.25 0.972649556692825
14.5 1.0729869823253
14.75 1.16545652166764
15 1.22031896919692
15.25 1.20914291152692
15.5 1.10787888415921
15.75 0.92608502547897
16 0.727267770666813
16.25 0.586673334015284
16.5 0.539289441209734
16.75 0.56544608039304
17 0.62276761934505
17.25 0.680754007719227
17.5 0.72296547670948
17.75 0.740231750816254
18 0.727431930068784
18.25 0.68345039889805
18.5 0.612154124629829
18.75 0.522337679160268
19 0.425843688921849
19.25 0.335218712474485
19.5 0.262718000108686
19.75 0.219510616990316
};
\addplot [very thick, color4]
table {%
0 2
0.25 1.72330159490896
0.5 1.81744147737989
0.75 2.00456098377807
1 2.17536869982029
1.25 2.20128711799587
1.5 1.83441631148244
1.75 1.62831190043368
2 1.73250668724801
2.25 1.91120688467294
2.5 2.06916926413198
2.75 2.09160403606232
3 1.76268165349927
3.25 1.50302948326438
3.5 1.56138150617319
3.75 1.71985879706552
4 1.87704431959278
4.25 1.95539931283965
4.5 1.79875235746786
4.75 1.43096691333808
5 1.31629175457771
5.25 1.40609179476834
5.5 1.55114034508916
5.75 1.68079968024206
6 1.73216874804126
6.25 1.60072965900404
6.5 1.28092427429476
6.75 1.07805570897571
7 1.08319327654255
7.25 1.18451305708884
7.5 1.30206154212544
7.75 1.39197838037994
8 1.41336304317053
8.25 1.31710071776466
8.5 1.09528911316594
8.75 0.859085828417136
9 0.733917527109006
9.25 0.733270807147533
9.5 0.799115108525764
9.75 0.87969274197229
10 0.945854037279664
10.25 0.979485201345112
10.5 0.966694286582291
10.75 0.899443008207154
11 0.782913277245382
11.25 0.63951778705592
11.5 0.501301541518057
11.75 0.39851468311791
12 0.350935065686392
12.25 0.355500665662348
12.5 0.38849239211023
12.75 0.427299836625198
13 0.459050009901147
13.25 0.477556266907539
13.5 0.480200053810813
13.75 0.466557434434788
14 0.437898363830076
14.25 0.396886129450477
14.5 0.347187213462782
14.75 0.292992819817044
15 0.238627172428179
15.25 0.188449982945679
15.5 0.147143268604885
15.75 0.119862992208112
16 0.110024964303603
16.25 0.114688881307705
16.5 0.126167566262796
16.75 0.138160023801567
17 0.147343099347604
17.25 0.152326568072673
17.5 0.152746828691761
17.75 0.148795142612492
18 0.140979990847418
18.25 0.129994197539103
18.5 0.116630648439679
18.75 0.101725834497337
19 0.0861264284262228
19.25 0.0706830810612483
19.5 0.0562840207310409
19.75 0.0439455718491514
};
\addplot [very thick, color5]
table {%
0 2
0.25 1.60337656970976
0.5 1.49793791332297
0.75 1.61482675011284
1 1.7805007614912
1.25 1.9175389330136
1.5 1.92775971615254
1.75 1.64609716455904
2 1.34302141082498
2.25 1.32965671693007
2.5 1.45154036851545
2.75 1.59643816487371
3 1.70507449281441
3.25 1.70297449556849
3.5 1.48600913309952
3.75 1.1732461797453
4 1.0518855954514
4.25 1.10220725254949
4.5 1.21442365978073
4.75 1.32576063548776
5 1.39696005301878
5.25 1.38499509025006
5.5 1.24604972355281
5.75 1.00621574559065
6 0.796414885726778
6.25 0.710996988030749
6.5 0.734811367573155
6.75 0.807506335273722
7 0.885019251668793
7.25 0.942909768753572
7.5 0.965092489907307
7.75 0.9394506940731
8 0.861044034784276
8.25 0.738956613362092
8.5 0.597427548645314
8.75 0.467011819362047
9 0.375109395405015
9.25 0.337876253592485
9.5 0.348332398966868
9.75 0.382381676529423
10 0.419532862514437
10.25 0.448599536653417
10.5 0.464315399498245
10.75 0.464587011668993
11 0.449322826048483
11.25 0.419980753595182
11.5 0.379260943541301
11.75 0.330720403120963
12 0.278335361901994
12.25 0.226172824244656
12.5 0.178348545242574
12.75 0.139333437419155
13 0.114046822107715
13.25 0.105526278306472
13.5 0.11056973841768
13.75 0.121742855161356
14 0.133162647831558
14.25 0.141774547995819
14.5 0.146319245737882
14.75 0.146488485696625
15 0.142488750379459
15.25 0.134820545557214
15.5 0.124154310792147
15.75 0.111251670380697
16 0.0969130341760543
16.25 0.0819472900278711
16.5 0.0671676996699991
16.75 0.0534263369396353
17 0.0417031528461875
17.25 0.0332012810156904
17.5 0.0290552690508032
17.75 0.0291591208770749
18 0.0317594043589027
18.25 0.0349721914492144
18.5 0.0376826541431436
18.75 0.0393907284618913
19 0.0399334553229962
19.25 0.0393249051073278
19.5 0.0376754204618406
19.75 0.0351487290717581
};
\addplot [very thick, color6]
table {%
0 2
0.25 1.8992018757998
0.5 1.5220798327959
0.75 1.37134849324543
1 1.45547987743251
1.25 1.60549378096617
1.5 1.74333496542061
1.75 1.80116935772248
2 1.65895093802534
2.25 1.32089282922945
2.5 1.14431006060886
2.75 1.17960565945476
3 1.29697243027527
3.25 1.42018698980374
3.5 1.5035781779881
3.75 1.4939083733646
4 1.33035469042141
4.25 1.05823656124623
4.5 0.863368469948272
4.75 0.824802054782142
5 0.885552551764681
5.25 0.976867479136019
5.5 1.0589663418693
5.75 1.10719328579094
6 1.1004257001848
6.25 1.02252497995273
6.5 0.877632864242952
6.75 0.702489366303792
7 0.54999011584123
7.25 0.461335862634557
7.5 0.446410286411122
7.75 0.479992511493059
8 0.529480149202549
8.25 0.574236455932351
8.5 0.603651145237897
8.75 0.61237021648577
9 0.598226276790351
9.25 0.561853727165191
9.5 0.506673256325849
9.75 0.438441147854154
10 0.36425627468339
10.25 0.291566715871032
10.5 0.22783773300545
10.75 0.180845241020235
11 0.157370902618331
11.25 0.157221824915887
11.5 0.17099502632015
11.75 0.188361448490089
12 0.203143003899723
12.25 0.212506886664149
12.5 0.215456166131763
12.75 0.211980884308544
13 0.202642782681106
13.25 0.188352577372517
13.5 0.170227908713133
13.75 0.149491478604275
14 0.127401300219107
14.25 0.105220231722997
14.5 0.0842428443652644
14.75 0.065905560349452
15 0.0519503518187343
15.25 0.0442134098878862
15.5 0.0431174889486593
15.75 0.0464483382680396
16 0.0512333514268571
16.25 0.0555418327643961
16.5 0.0584623495667448
16.75 0.0596647694604499
17 0.0591237656610754
17.25 0.0569787648103126
17.5 0.0534609006841615
17.75 0.048851196450826
18 0.0434547130299406
18.25 0.0375847050135871
18.5 0.0315553546355676
18.75 0.0256849951143743
19 0.0203152436035313
19.25 0.0158513569639439
19.5 0.012787866503852
19.75 0.0115200399391285
};
\addplot [very thick, color0]
table {%
0 2
0.25 2.20024480999269
0.5 2.3341392021961
0.75 2.14519246024036
1 1.7743658371906
1.25 1.82566972831978
1.5 2.01017688083742
1.75 2.19417375814552
2 2.26186065497228
2.25 1.93674966702299
2.5 1.68039390416046
2.75 1.77484915109971
3 1.9576531858788
3.25 2.1242253284759
3.5 2.15518042153193
3.75 1.81334058779696
4 1.57083133748255
4.25 1.65075469006324
4.5 1.82021368950884
4.75 1.97995099693658
5 2.03729449459001
5.25 1.79625370781009
5.5 1.4562151740048
5.75 1.43882371324766
6 1.57144333220775
6.25 1.72783754341611
6.5 1.84148262500782
6.75 1.81379795206156
7 1.5208087757819
7.25 1.23893960860283
7.5 1.21524331768569
7.75 1.32228757762083
8 1.45591206906141
8.25 1.56198887595678
8.5 1.58526212560745
8.75 1.44796028833412
9 1.16238347606053
9.25 0.946855207528987
9.5 0.908221133703901
9.75 0.977752737698704
10 1.07841785890875
10.25 1.16719972786479
10.5 1.21540403816896
10.75 1.19506027088895
11 1.08466078506177
11.25 0.899291105295467
11.5 0.704779097642868
11.75 0.572573498445892
12 0.532536984902166
12.25 0.562078120496409
12.5 0.619563243376779
12.75 0.676115257383809
13 0.716268935302144
13.25 0.731407365174291
13.5 0.716849704865735
13.75 0.671879559687836
14 0.60061036052229
14.25 0.5117870454409
14.5 0.416917821411938
14.75 0.328099489517864
15 0.257175290205975
15.25 0.214992896115825
15.5 0.205540830236142
15.75 0.219456690346697
16 0.242088112663305
16.25 0.263357554666133
16.5 0.278287878718034
16.75 0.284858891250625
17 0.282636952176259
17.25 0.272120567740818
17.5 0.254408646978092
17.75 0.230988166906608
18 0.203567989386561
18.25 0.173948900677519
18.5 0.143947998949984
18.75 0.115409160349638
19 0.0903377424318678
19.25 0.0711240895860293
19.5 0.0602983443432992
19.75 0.0585449377499382
};
\addplot [very thick, color1]
table {%
0 2
0.25 2.13584915615013
0.5 2.35579609905341
0.75 2.53410155827188
1 2.39100316510984
1.25 2.01126678725578
1.5 2.11487605228342
1.75 2.33270039463366
2 2.52611676490999
2.25 2.46146004998117
2.5 2.03439611749863
2.75 2.10168377867263
3 2.31616742329989
3.25 2.52019446953101
3.5 2.51382783033458
3.75 2.0671072085019
4 2.09792878896686
4.25 2.30886601045992
4.5 2.5203369912457
4.75 2.55321569559285
5 2.10449551804654
5.25 2.10548856679963
5.5 2.31413681524108
5.75 2.53083890939694
6 2.58576134741553
6.25 2.13988056850395
6.5 2.12686420432854
6.75 2.33632306914009
7 2.5565932573887
7.25 2.6157927867827
7.5 2.16754525359805
7.75 2.16615757236429
8 2.38125396107141
8.25 2.60321799093611
8.5 2.64220975402945
8.75 2.18836562956557
9 2.23073822585488
9.25 2.45685131757352
9.5 2.67649414166272
9.75 2.64827433075665
10 2.21939169280347
10.25 2.33306095569345
10.5 2.57352329225434
10.75 2.77870378321205
11 2.58016561822931
11.25 2.29963563061724
11.5 2.49120804876947
11.75 2.74287762573653
12 2.88703609693498
12.25 2.45114393832951
12.5 2.4753715165992
12.75 2.72647886846334
13 2.96613415495279
13.25 2.80948889755975
13.5 2.54733676580614
13.75 2.77860200480638
14 3.05189092170716
14.25 3.09851525511742
14.5 2.70222436751728
14.75 2.92425328350433
15 3.21835171699409
15.25 3.30008464728797
15.5 2.92653612227758
15.75 3.18853920646176
16 3.50060003352905
16.25 3.37599976116803
16.5 3.27375831648824
16.75 3.60758673206428
17 3.88429232557734
17.25 3.51492900557026
17.5 3.83601741758832
17.75 4.19808699977932
18 3.88805989391009
18.25 4.2253320619551
18.5 4.62466976199816
18.75 4.37493296309705
19 4.80751554602349
19.25 5.0884505812851
19.5 5.09470690528666
19.75 5.60699674494131
};
\addplot [very thick, color2]
table {%
0 2
0.25 1.94440790282757
0.5 2.1272048289507
0.75 2.33570397298819
1 2.4501816084931
1.25 2.13499543699196
1.5 1.89234547567133
1.75 2.0284010064944
2 2.23700779523803
2.25 2.40528854837447
2.5 2.30359257658515
2.75 1.88679342192582
3 1.92110228372847
3.25 2.11364027257629
3.5 2.30948714890842
3.75 2.37757515171745
4 2.00253820519491
4.25 1.82281033946726
4.5 1.96498066670166
4.75 2.16614843407003
5 2.32285487812881
5.25 2.22357647363681
5.5 1.80522115950594
5.75 1.80121818838785
6 1.97577634356335
6.25 2.166260730591
6.5 2.2700502400042
6.75 2.02543299937537
7 1.69155734001156
7.25 1.74884137860048
7.5 1.92613210688818
7.75 2.10168497895653
8 2.17203628739219
8.25 1.89554360456462
8.5 1.58347682817453
8.75 1.62879826486431
9 1.79244764772521
9.25 1.95917141536217
9.5 2.04453752186856
9.75 1.86562862132889
10 1.49248469050228
10.25 1.42937319007892
10.5 1.54957471592626
10.75 1.7073095638771
11 1.83300865662122
11.25 1.83803849626835
11.5 1.58315890275569
11.75 1.2684094144917
12 1.21126516119661
12.25 1.30775267719349
12.5 1.44189286515976
12.75 1.55544071620045
13 1.59533514572597
13.25 1.48418823119496
13.5 1.20735659694196
13.75 0.97012047166119
14 0.911125720868695
14.25 0.97264955669286
14.5 1.07298698232534
14.75 1.16545652166769
15 1.22031896919702
15.25 1.20914291152706
15.5 1.10787888415937
15.75 0.926085025479106
16 0.72726777066692
16.25 0.586673334015378
16.5 0.53928944120983
16.75 0.565446080393147
17 0.622767619345171
17.25 0.680754007719359
17.5 0.722965476709616
17.75 0.740231750816384
18 0.727431930068899
18.25 0.683450398898138
18.5 0.612154124629884
18.75 0.522337679160292
19 0.425843688921848
19.25 0.335218712474473
19.5 0.262718000108685
19.75 0.219510616990346
};
\addplot [very thick, color3]
table {%
0 2
0.25 1.72330159490896
0.5 1.81744147737989
0.75 2.00456098377807
1 2.17536869982029
1.25 2.20128711799587
1.5 1.83441631148245
1.75 1.62831190043368
2 1.73250668724802
2.25 1.91120688467295
2.5 2.06916926413198
2.75 2.09160403606233
3 1.76268165349927
3.25 1.50302948326438
3.5 1.56138150617319
3.75 1.71985879706552
4 1.87704431959278
4.25 1.95539931283965
4.5 1.79875235746786
4.75 1.43096691333808
5 1.31629175457772
5.25 1.40609179476836
5.5 1.55114034508917
5.75 1.68079968024207
6 1.73216874804128
6.25 1.60072965900404
6.5 1.28092427429476
6.75 1.07805570897571
7 1.08319327654255
7.25 1.18451305708885
7.5 1.30206154212545
7.75 1.39197838038
8 1.41336304317105
8.25 1.31710071776489
8.5 1.09528911316599
8.75 0.859085828417257
9 0.733917527109354
9.25 0.733270807148098
9.5 0.799115108526453
9.75 0.879692741973024
10 0.945854037280383
10.25 0.979485201345746
10.5 0.966694286582751
10.75 0.899443008207358
11 0.782913277245341
11.25 0.639517787055693
11.5 0.501301541517816
11.75 0.398514683117855
12 0.350935065686657
12.25 0.355500665662887
12.5 0.388492392110904
12.75 0.427299836625901
13 0.459050009901813
13.25 0.477556266908129
13.5 0.480200053811293
13.75 0.466557434435142
14 0.437898363830287
14.25 0.396886129450562
14.5 0.34718721346274
14.75 0.292992819816918
15 0.238627172427997
15.25 0.188449982945489
15.5 0.147143268604746
15.75 0.119862992208104
16 0.110024964303752
16.25 0.11468888130797
16.5 0.126167566263111
16.75 0.138160023801883
17 0.147343099347897
17.25 0.15232656807293
17.5 0.152746828691974
17.75 0.148795142612656
18 0.140979990847537
18.25 0.129994197539173
18.5 0.116630648439706
18.75 0.101725834497327
19 0.0861264284261856
19.25 0.0706830810611956
19.5 0.0562840207309842
19.75 0.043945571849104
};
\addplot [very thick, color4]
table {%
0 2
0.25 1.60337656970975
0.5 1.49793791332297
0.75 1.61482675011284
1 1.78050076149119
1.25 1.9175389330136
1.5 1.92775971615255
1.75 1.64609716455904
2 1.34302141082497
2.25 1.32965671693006
2.5 1.45154036851544
2.75 1.5964381648737
3 1.7050744928144
3.25 1.70297449556849
3.5 1.48600913309953
3.75 1.1732461797453
4 1.0518855954514
4.25 1.10220725254948
4.5 1.21442365978072
4.75 1.32576063548775
5 1.39696005301877
5.25 1.38499509025006
5.5 1.24604972355281
5.75 1.00621574559066
6 0.796414885726784
6.25 0.710996988030744
6.5 0.734811367573144
6.75 0.807506335273709
7 0.88501925166878
7.25 0.942909768753562
7.5 0.965092489907299
7.75 0.939450694073095
8 0.861044034784275
8.25 0.738956613362095
8.5 0.597427548645319
8.75 0.46701181936205
9 0.375109395405015
9.25 0.33787625359248
9.5 0.348332398966861
9.75 0.382381676529414
10 0.419532862514428
10.25 0.448599536653408
10.5 0.464315399498237
10.75 0.464587011668987
11 0.449322826048479
11.25 0.419980753595182
11.5 0.379260943541305
11.75 0.330720403120968
12 0.278335361901998
12.25 0.226172824244659
12.5 0.178348545242576
12.75 0.139333437419157
13 0.114046822107718
13.25 0.105526278306477
13.5 0.110569738417687
13.75 0.121742855161364
14 0.133162647831568
14.25 0.141774547995829
14.5 0.146319245737892
14.75 0.146488485696635
15 0.142488750379468
15.25 0.134820545557222
15.5 0.124154310792152
15.75 0.111251670380704
16 0.0969130341760615
16.25 0.0819472900278792
16.5 0.0671676996700039
16.75 0.0534263369396367
17 0.0417031528461883
17.25 0.0332012810156915
17.5 0.0290552690508104
17.75 0.0291591208770872
18 0.0317594043589178
18.25 0.0349721914492309
18.5 0.0376826541431593
18.75 0.0393907284619023
19 0.039933455323008
19.25 0.0393249051073391
19.5 0.0376754204618534
19.75 0.0351487290717711
};
\addplot [very thick, color5]
table {%
0 1.33333333333333
0.25 1.4560956495965
0.5 1.53245314712511
0.75 1.50379731603625
1 1.31057748571767
1.25 1.03293515694509
1.5 0.867243190154051
1.75 0.858170665159961
2 0.933238898854764
2.25 1.02787923909883
2.5 1.10622140711137
2.75 1.14367900148838
3 1.11761398289248
3.25 1.0135633018762
3.5 0.846131956737649
3.75 0.66587874518223
4 0.530269292661442
4.25 0.472457976313134
4.5 0.484419264742663
4.75 0.531195871988436
5 0.583219051862249
5.25 0.624144658818
5.5 0.645519935172242
5.75 0.643046807505759
6 0.615557842298936
6.25 0.565129232946226
6.5 0.497018392775546
6.75 0.418789595206562
7 0.338961635446648
7.25 0.266112745794755
7.5 0.208934976999327
7.75 0.175716509555769
8 0.169162900878038
8.25 0.181203871547733
8.5 0.199901843923456
8.75 0.217210483456479
9 0.229213542160697
9.25 0.234375760959632
9.5 0.232431395854919
9.75 0.223841308969053
10 0.20951345975964
10.25 0.190630527043585
10.5 0.168525402973562
10.75 0.144590720238888
11 0.120229172657695
11.25 0.0968632913865512
11.5 0.0760332716223842
11.75 0.059578342472851
12 0.0495723259043393
12.25 0.04698838881155
12.5 0.049925453415594
12.75 0.0550596465366071
13 0.0599988604881423
13.25 0.0635564363851669
13.5 0.0652707145989367
13.75 0.0650610135529253
14 0.0630509880701559
14.25 0.0594778440035837
14.5 0.0546413357124483
14.75 0.0488725748787317
15 0.0425145434004483
15.25 0.0359121010194148
15.5 0.0294127703251123
15.75 0.0233835560398084
16 0.0182508224543312
16.25 0.0145412727553583
16.5 0.0127487250377275
16.75 0.0128156014784222
17 0.013966103164769
17.25 0.0153765112188209
17.5 0.0165618713706233
17.75 0.0173054174048567
18 0.0175371769811333
18.25 0.0172642305006247
18.5 0.0165355008664384
18.75 0.0154230438001709
19 0.0140111737543443
19.25 0.0123897476104494
19.5 0.0106502146570029
19.75 0.00888418609214183
};
\addplot [very thick, color6]
table {%
0 1.33333333333333
0.25 1.37136344242607
0.5 1.50806182788954
0.75 1.65061652757186
1 1.74016470414848
1.25 1.68635193759095
1.5 1.40502906412831
1.75 1.13080152408012
2 1.07911097179247
2.25 1.16243287791215
2.5 1.28199881918469
2.75 1.38588589396306
3 1.43382565607265
3.25 1.37582321988044
3.5 1.17986701781437
3.75 0.928279353689756
4 0.76782913199492
4.25 0.742359360540478
4.5 0.80002513185207
4.75 0.88235954586961
5 0.95516933787748
5.25 0.998057115063421
5.5 0.995404870778251
5.75 0.936514517033637
6 0.823324465154042
6.25 0.677043851082699
6.5 0.531844843667286
6.75 0.421240600662585
7 0.367689288920345
7.25 0.369747392145956
7.5 0.403179254469686
7.75 0.443800306918757
8 0.477601407749477
8.25 0.497762543493477
8.5 0.501291735232961
8.75 0.487551441025651
9 0.45775752656474
9.25 0.414694368370792
9.5 0.362303951867963
9.75 0.305145484862941
10 0.247937596391416
10.25 0.195434768385993
10.5 0.152730778584261
10.75 0.125324988215987
11 0.11644786014323
11.25 0.122336419546496
11.5 0.134752084289142
11.75 0.147298407692983
12 0.156679626680215
12.25 0.161543064088255
12.5 0.161565059501357
12.75 0.15698230285627
13 0.148354687960009
13.25 0.136430788952204
13.5 0.122061098153761
13.75 0.106139423377014
14 0.0895688776434538
14.25 0.0732580656451009
14.5 0.0581619976751719
14.75 0.0453846837116181
15 0.0362760315022368
15.25 0.0320474074003583
15.5 0.0324360294112399
15.75 0.0354243118795358
16 0.0389736742442145
16.25 0.0419091376792891
16.5 0.0437119890192098
16.75 0.0442200119704215
17 0.0434585336375324
17.25 0.0415551828457876
17.5 0.0386943424002145
17.75 0.0350904501526252
18 0.0309714124856536
18.25 0.0265688566040066
18.5 0.0221148498472916
18.75 0.0178471586888527
19 0.014027738796337
19.25 0.0109753684577755
19.5 0.00906212572951529
19.75 0.00849396519413963
};
\addplot [very thick, color0]
table {%
0 1.33333333333333
0.25 1.10897133793824
0.5 1.10400397735719
0.75 1.20471940995286
1 1.32538436950458
1.25 1.41974848376162
1.5 1.44489587726052
1.75 1.34735020060415
2 1.11705255927696
2.25 0.877418471043943
2.5 0.760491200954004
2.75 0.770364356214634
3 0.842872409051082
3.25 0.92645127841193
3.5 0.992387470593828
3.75 1.02174946343775
4 0.999237202268744
4.25 0.916780873282226
4.5 0.783522372299336
4.75 0.629098579864818
5 0.491490807288469
5.25 0.402624287328594
5.5 0.376175988312594
5.75 0.397250421817878
6 0.43787823299307
6.25 0.477898489650631
6.5 0.506828111095109
6.75 0.519690714199204
7 0.514551865396701
7.25 0.491642508560323
7.5 0.453058768625281
7.75 0.402436651374057
8 0.344416281131226
8.25 0.284043321818875
8.5 0.226409688281563
8.75 0.176764124319497
9 0.140864737830598
9.25 0.123717898542854
9.5 0.12462662125469
9.75 0.135913864078865
10 0.149605468814663
10.25 0.161044964200614
10.5 0.168138558911685
10.75 0.170189086661382
11 0.167245233371341
11.25 0.15977876836087
11.5 0.148510778438374
11.75 0.134304056159733
12 0.118090012273423
12.25 0.100821459538456
12.5 0.0834541530466146
12.75 0.0669688024667476
13 0.052453425136584
13.25 0.0412328215458014
13.5 0.0347557512903275
13.75 0.0335069784295223
14 0.0359072207258106
14.25 0.0396124227528488
14.5 0.043037613607688
14.75 0.0454197888356975
15 0.0464741649328693
15.25 0.0461666576129181
15.5 0.0445979921314986
15.75 0.0419436399405047
16 0.0384197877292105
16.25 0.034262475035362
16.5 0.0297147853334294
16.75 0.0250206964677382
17 0.0204267747171215
17.25 0.016195691618005
17.5 0.0126361999368058
17.75 0.0101282165788681
18 0.00900248463982077
18.25 0.00915890029664696
18.5 0.010017898788552
18.75 0.0110147605177291
19 0.0118291023375456
19.25 0.0123210084781461
19.5 0.0124481263586965
19.75 0.0122194008180443
};
\addplot [very thick, color1]
table {%
0 1.33333333333333
0.25 1.13864057894584
0.5 0.895598767118943
0.75 0.736872108069325
1 0.706057267654292
1.25 0.757842881041762
1.5 0.836007046328338
1.75 0.906755018044068
2 0.950585369035355
2.25 0.953462839428586
2.5 0.905885042771651
2.75 0.808208616786654
3 0.676041838525658
3.25 0.53686453226939
3.5 0.420029465943163
3.75 0.349235351945236
4 0.332722201617699
4.25 0.354887210684257
4.5 0.391472752259253
4.75 0.4259626353536
5 0.45000788057527
5.25 0.459850051537158
5.5 0.45428339584605
5.75 0.433850054346583
6 0.400484113697654
6.25 0.357188368931807
6.5 0.307629506906283
6.75 0.2557364919131
7 0.205471115980927
7.25 0.160918845992359
7.5 0.126648227921484
7.75 0.107308849223171
8 0.104153640626879
8.25 0.111972324046827
8.5 0.123519383462064
8.75 0.134023060364353
9 0.141204887144536
9.25 0.144214725567617
9.5 0.142956048930748
9.75 0.137747855163976
10 0.129148764255821
10.25 0.117853804728869
10.5 0.104627186880286
10.75 0.0902584227470689
11 0.0755406301582089
11.25 0.0612775235507718
11.5 0.0483338472343856
11.75 0.037738564043216
12 0.0307223617001792
12.25 0.0281469585321874
12.5 0.0292960884981101
12.75 0.0322188141366926
13 0.0352928777846953
13.25 0.0376595769592289
13.5 0.038962677357106
13.75 0.0391124996890221
14 0.0381606331672649
14.25 0.0362363251360039
14.5 0.0335116652932318
14.75 0.0301807739293563
15 0.0264468713567253
15.25 0.0225151128396865
15.5 0.0185913969870678
15.75 0.014889593411059
16 0.0116516383202508
16.25 0.00917634433966673
16.5 0.00778554200244884
16.75 0.00756471538385746
17 0.00813580194841542
17.25 0.00897467454724576
17.5 0.00973643258626518
17.75 0.0102573337237793
18 0.0104776460209324
18.25 0.0103920940643343
18.5 0.010024701238339
18.75 0.00941565863170676
19 0.00861389321533347
19.25 0.00767251727508024
19.5 0.00664607009028465
19.75 0.00558925566736356
};
\addplot [very thick, color2]
table {%
0 1.33333333333333
0.25 1.45609564959651
0.5 1.53245314712512
0.75 1.50379731603627
1 1.31057748571769
1.25 1.03293515694511
1.5 0.867243190154066
1.75 0.858170665159977
2 0.933238898854782
2.25 1.02787923909885
2.5 1.10622140711139
2.75 1.1436790014884
3 1.1176139828925
3.25 1.01356330187621
3.5 0.846131956737647
3.75 0.665878745182225
4 0.530269292661443
4.25 0.472457976313146
4.5 0.484419264742686
4.75 0.531195871988464
5 0.583219051862277
5.25 0.624144658818027
5.5 0.645519935172265
5.75 0.643046807505777
6 0.615557842298947
6.25 0.56512923294623
6.5 0.497018392775544
6.75 0.418789595206555
7 0.338961635446637
7.25 0.266112745794745
7.5 0.208934976999321
7.75 0.175716509555771
8 0.169162900878048
8.25 0.181203871547749
8.5 0.199901843923473
8.75 0.217210483456496
9 0.229213542160713
9.25 0.234375760959645
9.5 0.232431395854929
9.75 0.22384130896906
10 0.209513459759644
10.25 0.190630527043586
10.5 0.168525402973561
10.75 0.144590720238886
11 0.120229172657692
11.25 0.0968632913865473
11.5 0.0760332716223804
11.75 0.0595783424728486
12 0.0495723259043394
12.25 0.0469883888115529
12.5 0.0499254534155988
12.75 0.0550596465366125
13 0.0599988604881475
13.25 0.0635564363851717
13.5 0.0652707145989408
13.75 0.0650610135529285
14 0.0630509880701583
14.25 0.0594778440035852
14.5 0.0546413357124491
14.75 0.0488725748787319
15 0.0425145434004479
15.25 0.0359121010194139
15.5 0.0294127703251112
15.75 0.0233835560398072
16 0.0182508224543302
16.25 0.0145412727553578
16.5 0.0127487250377278
16.75 0.0128156014784233
17 0.0139661031647705
17.25 0.0153765112188225
17.5 0.0165618713706248
17.75 0.0173054174048581
18 0.0175371769811345
18.25 0.0172642305006256
18.5 0.016535500866439
18.75 0.0154230438001712
19 0.0140111737543441
19.25 0.0123897476104492
19.5 0.0106502146570026
19.75 0.00888418609214143
};
\addplot [very thick, color3]
table {%
0 1.33333333333333
0.25 1.37136344242607
0.5 1.50806182788953
0.75 1.65061652757185
1 1.74016470414847
1.25 1.68635193759094
1.5 1.40502906412831
1.75 1.13080152408012
2 1.07911097179247
2.25 1.16243287791214
2.5 1.28199881918468
2.75 1.38588589396304
3 1.43382565607262
3.25 1.3758232198804
3.5 1.17986701781432
3.75 0.928279353689719
4 0.767829131994861
4.25 0.742359360540387
4.5 0.800025131851957
4.75 0.882359545869487
5 0.955169337877357
5.25 0.998057115063306
5.5 0.995404870778153
5.75 0.936514517033578
6 0.823324465153995
6.25 0.677043851082727
6.5 0.531844843667328
6.75 0.421240600662601
7 0.367689288920283
7.25 0.369747392145802
7.5 0.403179254469524
7.75 0.443800306918588
8 0.477601407749293
8.25 0.497762543493325
8.5 0.50129173523286
8.75 0.487551441025597
9 0.457757526564855
9.25 0.414694368371158
9.5 0.362303951868462
9.75 0.305145484863328
10 0.247937596391646
10.25 0.195434768386075
10.5 0.152730778584263
10.75 0.125324988216002
11 0.116447860143249
11.25 0.122336419546505
11.5 0.134752084289164
11.75 0.147298407693011
12 0.156679626680245
12.25 0.161543064088285
12.5 0.161565059501381
12.75 0.156982302856228
13 0.148354687959949
13.25 0.136430788952191
13.5 0.122061098153735
13.75 0.106139423376945
14 0.0895688776434116
14.25 0.0732580656451171
14.5 0.0581619976752343
14.75 0.0453846837117011
15 0.0362760315022724
15.25 0.0320474074002901
15.5 0.0324360294111108
15.75 0.0354243118793714
16 0.0389736742440443
16.25 0.041909137679138
16.5 0.0437119890191027
16.75 0.0442200119702967
17 0.043458533637398
17.25 0.0415551828456153
17.5 0.0386943424000055
17.75 0.0350904501523889
18 0.0309714124854175
18.25 0.026568856603798
18.5 0.0221148498471355
18.75 0.0178471586887597
19 0.0140277387962937
19.25 0.0109753684577323
19.5 0.00906212572938762
19.75 0.00849396519385899
};
\addplot [very thick, color4]
table {%
0 1.33333333333333
0.25 1.10897133793824
0.5 1.10400397735719
0.75 1.20471940995286
1 1.32538436950458
1.25 1.41974848376162
1.5 1.44489587726053
1.75 1.34735020060415
2 1.11705255927696
2.25 0.877418471043942
2.5 0.760491200954008
2.75 0.770364356214643
3 0.842872409051094
3.25 0.926451278411942
3.5 0.992387470593839
3.75 1.02174946343775
4 0.999237202268744
4.25 0.916780873282222
4.5 0.783522372299327
4.75 0.629098579864809
5 0.491490807288461
5.25 0.402624287328591
5.5 0.376175988312598
5.75 0.397250421817885
6 0.437878232993078
6.25 0.477898489650637
6.5 0.506828111095115
6.75 0.519690714199208
7 0.514551865396703
7.25 0.491642508560324
7.5 0.453058768625281
7.75 0.402436651374056
8 0.344416281131224
8.25 0.284043321818873
8.5 0.22640968828156
8.75 0.176764124319495
9 0.140864737830597
9.25 0.123717898542855
9.5 0.124626621254692
9.75 0.135913864078867
10 0.149605468814666
10.25 0.161044964200617
10.5 0.168138558911687
10.75 0.170189086661384
11 0.167245233371342
11.25 0.159778768360871
11.5 0.148510778438374
11.75 0.134304056159732
12 0.118090012273422
12.25 0.100821459538455
12.5 0.0834541530466133
12.75 0.0669688024667462
13 0.0524534251365827
13.25 0.0412328215458005
13.5 0.0347557512903272
13.75 0.0335069784295228
14 0.0359072207258115
14.25 0.0396124227528498
14.5 0.0430376136076889
14.75 0.0454197888356982
15 0.0464741649328698
15.25 0.0461666576129185
15.5 0.0445979921314988
15.75 0.0419436399405048
16 0.0384197877292104
16.25 0.0342624750353618
16.5 0.0297147853334289
16.75 0.0250206964677377
17 0.0204267747171211
17.25 0.0161956916180049
17.5 0.0126361999368058
17.75 0.0101282165788678
18 0.00900248463981962
18.25 0.00915890029664466
18.5 0.0100178987885489
18.75 0.0110147605177256
19 0.0118291023375425
19.25 0.0123210084781423
19.5 0.012448126358693
19.75 0.0122194008180415
};
\addplot [very thick, color5]
table {%
0 1.33333333333333
0.25 1.13864057894584
0.5 0.895598767118942
0.75 0.736872108069327
1 0.706057267654296
1.25 0.757842881041768
1.5 0.836007046328345
1.75 0.906755018044074
2 0.950585369035361
2.25 0.95346283942859
2.5 0.905885042771654
2.75 0.808208616786656
3 0.676041838525659
3.25 0.536864532269391
3.5 0.420029465943165
3.75 0.349235351945239
4 0.332722201617703
4.25 0.354887210684262
4.5 0.391472752259258
4.75 0.425962635353606
5 0.450007880575275
5.25 0.459850051537163
5.5 0.454283395846053
5.75 0.433850054346586
6 0.400484113697656
6.25 0.357188368931807
6.5 0.307629506906282
6.75 0.255736491913099
7 0.205471115980924
7.25 0.160918845992357
7.5 0.126648227921482
7.75 0.107308849223171
8 0.104153640626882
8.25 0.111972324046831
8.5 0.123519383462068
8.75 0.134023060364357
9 0.141204887144539
9.25 0.14421472556762
9.5 0.14295604893075
9.75 0.137747855163978
10 0.129148764255822
10.25 0.11785380472887
10.5 0.104627186880286
10.75 0.090258422747069
11 0.0755406301582088
11.25 0.0612775235507714
11.5 0.0483338472343851
11.75 0.0377385640432157
12 0.0307223617001793
12.25 0.0281469585321881
12.5 0.0292960884981113
12.75 0.032218814136694
13 0.0352928777846967
13.25 0.0376595769592303
13.5 0.0389626773571072
13.75 0.0391124996890232
14 0.0381606331672658
14.25 0.0362363251360044
14.5 0.0335116652932321
14.75 0.0301807739293563
15 0.0264468713567251
15.25 0.0225151128396862
15.5 0.0185913969870674
15.75 0.0148895934110587
16 0.0116516383202505
16.25 0.00917634433966651
16.5 0.00778554200244887
16.75 0.00756471538385783
17 0.00813580194841606
17.25 0.00897467454724653
17.5 0.00973643258626597
17.75 0.0102573337237799
18 0.0104776460209332
18.25 0.0103920940643348
18.5 0.0100247012383395
18.75 0.00941565863170701
19 0.0086138932153333
19.25 0.00767251727508002
19.5 0.0066460700902842
19.75 0.00558925566736309
};
\addplot [very thick, color6]
table {%
0 0.666666666666667
0.25 0.696972243623234
0.5 0.767386479143135
0.75 0.839182028633896
1 0.890983744158792
1.25 0.909066407979485
1.5 0.883857341987444
1.75 0.812310900639847
2 0.70243939389574
2.25 0.57359751225472
2.5 0.450263536953166
2.75 0.355645890697928
3 0.306959844773221
3.25 0.304982725803748
3.5 0.331155457406604
3.75 0.364911688292743
4 0.393880116877223
4.25 0.412124155247673
4.5 0.41719404203763
4.75 0.408642011899784
5 0.387403892127019
5.25 0.355468309987754
5.5 0.315571317577429
5.75 0.2708677888173
6 0.224650614676923
6.25 0.180230535079378
6.5 0.141074070498345
6.75 0.111136532773886
7 0.09445573122849
7.25 0.0919976031457687
7.5 0.0990544694971457
7.75 0.109261947841983
8 0.11847543016971
8.25 0.124729088897579
8.5 0.127299359554271
8.75 0.126117182382541
9 0.12147281320754
9.25 0.113861880534262
9.5 0.103895469599279
9.75 0.0922423001669487
10 0.0795916562785021
10.25 0.0666355553149251
10.5 0.0540756003656494
10.75 0.0426672203804992
11 0.0333100183391626
11.25 0.0270838936963072
11.5 0.0247595841854096
11.75 0.0257320725585365
12 0.0282910702956955
12.25 0.0309996558857847
12.5 0.0330937681944376
12.75 0.0342552704812339
13 0.0344028073513661
13.25 0.033580418480199
13.5 0.0319009992065689
13.75 0.0295154170776166
14 0.0265940732017099
14.25 0.0233154787900065
14.5 0.0198598943530342
14.75 0.0164082008456014
15 0.013148097526482
15.25 0.0102915411227266
15.5 0.00810018162525665
15.75 0.00685791649159375
16 0.0066466796635992
16.25 0.00714035307523481
16.5 0.00787688766904726
16.75 0.00854956286888096
17 0.00901212806284219
17.25 0.00921085823801919
17.5 0.0091404874066144
17.75 0.00882178116532861
18 0.00828988496962581
18.25 0.00758770808178814
18.5 0.00676190403286495
18.75 0.00586041384241648
19 0.00493132746954906
19.25 0.00402329238113324
19.5 0.00318826211692326
19.75 0.00248749307298945
};
\addplot [very thick, color0]
table {%
0 0.666666666666667
0.25 0.566086921164244
0.5 0.458191614393522
0.75 0.358989275188627
1 0.283698299799896
1.25 0.244564005386767
1.5 0.242384530677569
1.75 0.262902250492057
2 0.289760574772655
2.25 0.312991727471245
2.5 0.327884319475055
2.75 0.332610706184534
3 0.326939377374278
3.25 0.311644668229748
3.5 0.288191447444343
3.75 0.258501892505963
4 0.22475030025761
4.25 0.189200397539567
4.5 0.15412801673581
4.75 0.121881919269403
5 0.0951134070559335
5.25 0.0769236334734424
5.5 0.0696584921280749
5.75 0.071915336423644
6 0.0789565085601653
6.25 0.0866218129366916
6.5 0.0926537889117456
6.75 0.0960932423456188
7 0.0966739232901262
7.25 0.0944980895115461
7.5 0.0898727954141869
7.75 0.0832204636755182
8 0.0750247143727918
8.25 0.0657957675496196
8.5 0.0560503092180667
8.75 0.0463066469192849
9 0.0371013189198322
9.25 0.0290382073128025
9.5 0.0228604179863291
9.75 0.0193704544507949
10 0.0187925146391682
10.25 0.0201974881583823
10.5 0.0222805696585052
10.75 0.0241786793111966
11 0.0254809451367577
11.25 0.0260367625337623
11.5 0.0258318589697465
11.75 0.0249252759244246
12 0.0234165979482998
12.25 0.021427348420002
12.5 0.019089610454927
12.75 0.0165391012411149
13 0.0139119573048373
13.25 0.0113459264741087
13.5 0.00898823221483529
13.75 0.00701262685467136
14 0.00563254335684007
14.25 0.00502890749705039
14.5 0.00513503863511127
14.75 0.00562235178429873
15 0.00617880282428393
15.25 0.00662932310270858
15.5 0.00689804280967618
15.75 0.00696251653575633
16 0.0068284212812195
16.25 0.00651676802869665
16.5 0.00605702992888076
16.75 0.00548310353521565
17 0.00483081236838919
17.25 0.00413647350782472
17.5 0.00343648087539352
17.75 0.00276824915499509
18 0.00217328287107311
18.25 0.00170235111471443
18.5 0.00141386240611944
18.75 0.00133656728053065
19 0.00141800533911633
19.25 0.00156365138509774
19.5 0.00170470864418951
19.75 0.00180689974305069
};
\addplot [very thick, color1]
table {%
0 0.666666666666667
0.25 0.696972243623248
0.5 0.767386479143152
0.75 0.839182028633913
1 0.890983744158809
1.25 0.9090664079795
1.5 0.883857341987458
1.75 0.812310900639854
2 0.702439393895743
2.25 0.573597512254725
2.5 0.450263536953168
2.75 0.355645890697931
3 0.306959844773229
3.25 0.304982725803761
3.5 0.331155457406619
3.75 0.364911688292759
4 0.393880116877239
4.25 0.412124155247688
4.5 0.417194042037643
4.75 0.408642011899796
5 0.387403892127028
5.25 0.355468309987761
5.5 0.315571317577434
5.75 0.270867788817302
6 0.224650614676923
6.25 0.180230535079377
6.5 0.141074070498344
6.75 0.111136532773886
7 0.094455731228492
7.25 0.0919976031457733
7.5 0.0990544694971519
7.75 0.10926194784199
8 0.118475430169716
8.25 0.124729088897586
8.5 0.127299359554277
8.75 0.126117182382546
9 0.121472813207544
9.25 0.113861880534265
9.5 0.103895469599281
9.75 0.09224230016695
10 0.0795916562785026
10.25 0.066635555314925
10.5 0.054075600365649
10.75 0.0426672203804986
11 0.0333100183391622
11.25 0.0270838936963074
11.5 0.0247595841854104
11.75 0.0257320725585379
12 0.0282910702956971
12.25 0.0309996558857866
12.5 0.0330937681944395
12.75 0.0342552704812359
13 0.0344028073513684
13.25 0.0335804184802014
13.5 0.0319009992065693
13.75 0.0295154170776174
14 0.0265940732017102
14.25 0.0233154787900068
14.5 0.0198598943530341
14.75 0.016408200845601
15 0.0131480975264814
15.25 0.010291541122726
15.5 0.00810018162525626
15.75 0.00685791649159348
16 0.00664667966359879
16.25 0.00714035307523649
16.5 0.00787688766905004
16.75 0.00854956286888541
17 0.00901212806284751
17.25 0.00921085823802533
17.5 0.00914048740662065
17.75 0.00882178116533471
18 0.00828988496962574
18.25 0.00758770808178852
18.5 0.00676190403286489
18.75 0.00586041384241741
19 0.00493132746954917
19.25 0.00402329238113264
19.5 0.0031882621169221
19.75 0.00248749307298843
};
\addplot [very thick, color2]
table {%
0 0.666666666666667
0.25 0.56608692116418
0.5 0.458191614393476
0.75 0.358989275188587
1 0.283698299799869
1.25 0.244564005386755
1.5 0.242384530677569
1.75 0.262902250492061
2 0.289760574772659
2.25 0.312991727471245
2.5 0.327884319475052
2.75 0.332610706184526
3 0.326939377374267
3.25 0.311644668229734
3.5 0.288191447444326
3.75 0.258501892505946
4 0.224750300257593
4.25 0.189200397539549
4.5 0.154128016735793
4.75 0.121881919269388
5 0.0951134070559218
5.25 0.0769236334734347
5.5 0.0696584921280708
5.75 0.0719153364236423
6 0.0789565085601638
6.25 0.0866218129366893
6.5 0.0926537889117424
6.75 0.0960932423456141
7 0.0966739232901208
7.25 0.0944980895115394
7.5 0.089872795414179
7.75 0.0832204636755097
8 0.075024714372784
8.25 0.0657957675496121
8.5 0.0560503092180593
8.75 0.0463066469192776
9 0.0371013189198259
9.25 0.0290382073127973
9.5 0.0228604179863253
9.75 0.0193704544507928
10 0.0187925146391671
10.25 0.0201974881583813
10.5 0.0222805696585039
10.75 0.0241786793111949
11 0.025480945136756
11.25 0.0260367625337601
11.5 0.0258318589697442
11.75 0.0249252759244221
12 0.0234165979482974
12.25 0.0214273484199996
12.5 0.0190896104549247
12.75 0.0165391012411128
13 0.0139119573048356
13.25 0.0113459264741088
13.5 0.00898823221483402
13.75 0.00701262685467131
14 0.00563254335684035
14.25 0.00502890749705112
14.5 0.00513503863511083
14.75 0.00562235178429736
15 0.00617880282428291
15.25 0.0066293231027092
15.5 0.00689804280967929
15.75 0.00696251653576114
16 0.00682842128122596
16.25 0.0065167680287044
16.5 0.00605702992888939
16.75 0.00548310353522481
17 0.00483081236839915
17.25 0.0041364735078324
17.5 0.00343648087539915
17.75 0.00276824915499858
18 0.00217328287107471
18.25 0.00170235111471521
18.5 0.00141386240611925
18.75 0.00133656728053023
19 0.00141800533911567
19.25 0.00156365138509641
19.5 0.00170470864418713
19.75 0.00180689974304861
};
\end{axis}

\end{tikzpicture}

%% file: chapters/stabsdrefb.tex
\section{Stabilization by Updating Riccati Based Feedback} \label{sec:stabsdrefb}
As can be inferred from the sufficient conditions in Theorem \ref{thm:expstabsolsplain} and \ref{thm:locconexpdec} for exponential decay of solutions, a feedback designed for stabilization should be such that the closed loop matrix $A(x)-B\fbg(x)$, cf. \eqref{eq:clsys}, is uniformly stable with respect to the state $x$. In this section we show how one can continuously update an SDRE feedback so that the bounds on the transient behavior and the decay for the closed loop matrix stay constant in a neighborhood. More precisely, if for a given state $x$, an SDRE based feedback renders the system stable with certain stability constants $\mglob$ and $\omglob$, the introduced approach can maintain these constants for small changes in $x$ in the course of the time evolution of the system.

For further reference, we define an abbreviation for the class of considered matrices.
\begin{defin}\label{def:classsmo}
	We say that $A\in\mathbb R^{n,n}$ is in class \smo~for given constants $K$ and $\omega$, if 
	\begin{equation*}
		\norm{e^{ A\tau}}\leq Ke^{-\omega \tau},
	\end{equation*}
	for $\tau > 0$.
\end{defin}


Assume that at the current state $x$, we have $A(x) - BF(x) \in \smo$, where $F(x) = R^{-1}B^\trp P$ and where $P=P(x)$ solves the Riccati equation \eqref{eq:sdre} for given $B\in \mathbb R^{n,\Nu}$, $ R\succ 0\in\mathbb R^{\Nu,\Nu}$, and $Q\succcurlyeq0\in \mathbb R^{n,n} $. Then, we have that
\begin{equation}\label{eq:ricninvarisubs}
	\begin{bmatrix}
		A(x) & -BR^{-1}B^\trp  \\ -Q & A(x)^\trp 
	\end{bmatrix}
	\begin{bmatrix} I \\ P \end{bmatrix} 
=
	\begin{bmatrix} I \\ P \end{bmatrix} Z,
\end{equation}
where $Z = A(x) - BR^{-1}B^\trp P \in \smo$. The following lemma proposes an update of $F$ to account for changes in the system matrix $A(x+x_\Delta)=:A(x)+\Adelta$ induced by a change $x_\Delta$ in the current state $x$.

\begin{thm} \label{thm:updric}
	Consider relation \eqref{eq:ricninvarisubs} with $Z\in \smo$. If for a $\Adelta \in \mathbb R^{n,n}$, there exist $\Qdelta  \in \mathbb R^{n,n}$, $\Rdeltapmo$, and $E \in \mathbb R^{n,n}$ such that  
\begin{equation}\label{eq:updricninvarisubs}
	\begin{bmatrix}
		A(x)+\Adelta & -B[R^{-1}+\Rdeltapmo]B^\trp  \\ -Q-\Qdelta  & A(x)^\trp+\Adelta^\trp   
	\end{bmatrix}
	\begin{bmatrix} I+E \\ P \end{bmatrix} 
=
	\begin{bmatrix} I+E \\ P \end{bmatrix} Z,
\end{equation}
and if $\norm{E}<1$, then $(I+E)$ is invertible and with $\Pdelta := P(I+E)^{-1}$ it holds that
\begin{equation*}
	A(x)+\Adelta - B[R^{-1}+\Rdeltapmo]B^\trp \Pdelta \in \stmo,
\end{equation*}
with $\tilde K = \frac{1+\norm{E}}{1-\norm{E}}K$.
\end{thm}
\begin{proof}
	Using the \emph{Neumann series} \cite[Exa. I.4.5]{Kat66}, one can infer from $\norm{E}<1$ that $(I+E)$ is invertible and that $\norm{(I+E)^{-1}}\leq \frac 1{1-\norm{E}}$. By multiplying the first block line in \eqref{eq:updricninvarisubs} by $(I+E)^{-1}$ from the left, taking the norm on both sides, recalling that $Z\in\smo$, and estimating $\norm{I+E}\leq 1 + \norm{E}$, we prove the lemma.
\end{proof}

As a consequence of Theorem \ref{thm:updric}, as long as for given $\Adelta$, 
one can find $\Qdelta$, $\Rdeltapmo$, and $E$, with $\norm{E}<c<1$ small enough, one can stabilize $A(x)$ in a neighborhood of $A(x)$ with a constant decay rate $\omega$ and a constant bound on the transient behavior. 

We will use the result of Theorem \ref{thm:updric} to define updates for a given feedback. For further reference, we formulate the situation as a problem.

\begin{prob}\label{prb:updateric}
	Consider the SDC system \eqref{eq:extlinsys} at time $t\geq 0$ and $\xi(t) =: x$. Let $R\succ0$ and $Q\succcurlyeq0$ be given and $P$ satisfy the SDRE \eqref{eq:sdre} so that 
\begin{equation}\label{eq:inprobricninvarisubs}
	\begin{bmatrix}
		A(x) & -BR^{-1}B^\trp  \\ -Q & A(x)^\trp 
	\end{bmatrix}
	\begin{bmatrix} I \\ P \end{bmatrix} 
=
	\begin{bmatrix} I \\ P \end{bmatrix} Z,
\end{equation}
holds for a $Z \in \smo$, and let $0< c < 1$. For a given $\Adelta \in \mathbb R^{\Nx, \Nx}$, find $\Qdelta $, $\Rdeltapmo$, and a corresponding $E$ so that 
\begin{equation}\label{eq:inprobupdricninvarisubs}
	\begin{bmatrix}
		A(x)+\Adelta & -B[R^{-1}+\Rdeltapmo]B^\trp  \\ -Q-\Qdelta  & A(x)^\trp+\Adelta^\trp   
	\end{bmatrix}
	\begin{bmatrix} I+E \\ P \end{bmatrix} 
=
\begin{bmatrix} I+E \\ P \end{bmatrix} Z, \quad \text{and} \quad\norm{E}<c.
\end{equation}
\end{prob}

In what follows, we will address sufficient conditions for the existence of such updates $E$ and how they can be computed.

\begin{lem}\label{lem:epsexistsylvester}
	Consider Problem \ref{prb:updateric}. Any solution $(\Qdelta , \Rdeltapmo, E)$ satisfies 
	\begin{equation}\label{eq:epsexistsylvester}
		(A(x)+\Adelta)E - E Z = -\Adelta + B\Rdeltapmo B^\trp P.
	\end{equation}
	and
	\begin{equation} \label{eq:defqdbyeps}
		-QE - \Qdelta (I + E) - \Adelta^\trp P = 0.
	\end{equation}
	Conversely, for given $0<c<1$, if there exist $\Rdeltapmo $ and $E$ with $\norm{E}<c$ that fulfill \eqref{eq:epsexistsylvester}, then \eqref{eq:defqdbyeps} can be solved for $\Qdelta $ and $(\Qdelta , \Rdeltapmo , E)$ satisfy \eqref{eq:inprobupdricninvarisubs}.
\end{lem}
\begin{proof}
	With $P$ solving \eqref{eq:ricninvarisubs}, the updated system \eqref{eq:updricninvarisubs} is equivalent to \eqref{eq:epsexistsylvester} and \eqref{eq:defqdbyeps}. Conversely, if there is a solution $E$ to \eqref{eq:epsexistsylvester}, then the first block line in \eqref{eq:inprobupdricninvarisubs} is satisfied. If also $\norm{E}<1$, then $1+E$ is invertible and there is a unique $\Qdelta $ so that \eqref{eq:defqdbyeps} and, thus, the second block line of \eqref{eq:inprobupdricninvarisubs} are fulfilled.
\end{proof}
According to Lemma \ref{lem:epsexistsylvester}, a desired solution $E$ to \eqref{eq:updricninvarisubs}, namely an $E$ with $\norm{E} < 1$, is always solely defined by \eqref{eq:epsexistsylvester}. Thus, solvability of \eqref{eq:epsexistsylvester} is the key for applying the approach of updating the initial Riccati based feedback.

Equation \eqref{eq:epsexistsylvester} is a \emph{Sylvester} equation \cite[Ch. 16]{Hig02} that can be written as
\begin{equation}\label{eq:sylvestlineq}
	\sylop(A(x)+\Adelta, -Z)\vecop(E) = \vecop(-\Adelta + B\Rdeltapmo B^\trp P), 
\end{equation}
where $\sylop(A_1, A_2) : = A_1 \otimes I - I \otimes A_2$ and $\vecop$ is the operator that stacks the columns of a matrix into a long vector. For given $A_1$ and $A_2$, the \emph{Sylvester} operator $\sylop$ is invertible, if and only if the spectra of $A_1$ and $A_2$ do not have a common eigenvalue.

In the considered case, there is no guarantee that the spectra of $A(x)+\Adelta$ and $-Z$ are disjoint. Thus, we can not state unique existence of solutions. If $A(x)+\Adelta$ and $-Z$ share an eigenvalue, then the associated $\sylop$ is rank-deficient. Then Equation \eqref{eq:epsexistsylvester} has a solution, or better infinitely many solutions, only if the inhomogeneity is consistent. Based on these considerations, we propose two practical approaches to such a solution $E$.
\begin{enumerate}
	\item Solve \eqref{eq:epsexistsylvester} with $\Rdeltapmo =0$. If this fails, then the linear operator $\sylop$ is not invertible and $\Adelta$ is not in the range of $\sylop$. One can try whether for a small second summand $-\Adelta + B\Rdeltapmo B^TP$ is consistent. However, since $B$ typically has only a few columns, this is only a low-rank update which is unlikely to fix the inconsistency in general.
	\item If \eqref{eq:epsexistsylvester} is not solvable, one may solve the perturbed system
	\begin{equation}\label{eq:epsexistsylvpert}
		(A(x)+\Adelta-BR^{-1}B^\trp P)E + E Z = -\Adelta + B\Rdeltapmo B^\trp P,
	\end{equation}
	which is hopefully a slight perturbation, if $E$ is small. If $\Adelta$ is small, then Equation \eqref{eq:epsexistsylvpert} has a unique solution since $BR^{-1}B^\trp P$ was stabilizing $A$ and also $Z$ has only eigenvalues with a negative real part. 
\end{enumerate}
Another issue is the smallness of the update $E$ -- a second crucial ingredient of the approach. If we assume that $\sylop$ is invertible, then the norm of the update is readily estimated by
\begin{equation} \label{eq:nepsisinvprhs}
	\norm{E}_F \leq \norm{\sylop ^{-1}}_2 \norm{\sylrhs} _F.
\end{equation}
Relation \eqref{eq:nepsisinvprhs} is also what the general perturbation estimates given in \cite[Eq. (16.23), (16.25)]{Hig02} reduce to in the considered case.

At a first glance, the smallness of $\sylrhs= -\Adelta + B\Rdeltapmo B^\trp P$ induces a small $E$. The freedom in the choice of $\Rdeltapmo $ can be used to further optimize the solution. Either through minimizing the norm of $\sylrhs$, which is probably not optimal in terms of a minimal norm $E$ but which comes with the a-priori estimate \eqref{eq:nepsisinvprhs}, or through minimizing the solution in an optimization setup. The latter optimization approach may also be be employed if $\sylop$ is not invertible, provided that one can guarantee a consistent right hand side for all considered choices of parameters.

 Estimates for $\norm{\sylop ^{-1}}_2$ may be obtained as follows. The direct approach would be to compute the largest singular value of $\sylop ^{-1}$ that defines the considered spectral norm of $\sylop ^{-1}$ e.g. via the \emph{power method} \cite{GhaL95}. Alternative ways are given by virtue of the equality of the smallest singular value of $\sylop(A_1, A_2)$ to the so called \emph{separation} of $A_1$ and $A_2$:
\begin{equation*}
	\sep(A_1, A_2) = \min_X \frac{\norm{A_1X - XA_2}_F}{\norm{X}_F},
\end{equation*}
cf. \cite{Var79}, e.g., via an algorithm reported in \cite{Bye84} that bases on \emph{Schur} decompositions and that has been implemented, e.g., in the \textsf{SB04OD} subroutine of \emph{SLICOT} \cite{BenMSVV99}.

%% file: chapters/numexa.tex
\section{Numerical Examples}\label{sec:numexa}
We consider the \emph{5D example} that was considered in \cite[Ch. 3.4]{BanLT07} and which writes as an SDC system $\dot \xi = A\xi + Bu$ like

\begin{subequations}\label{eq:sdc5dexa}
	\begin{align}
		\dsvecwu &= 
		\begin{bmatrix}
			0 & 1 & 0 & 0 & 0 \\ 
			0 & 0 & 1 & 0 & 0 \\ 
			0 & 0 & 0 & \svecf^2 & 0 \\ 
			-\sveco & 0 & 0 & \svecf^2 & 0 \\ 
			0 & 0 & 0 & 0 & 0 
		\end{bmatrix}\svecwu
		+
		\begin{bmatrix}
			0 & 0 \\ 0 & 0 \\ 1 & 0 \\ 0 & 0 \\ 0 & 1
		\end{bmatrix}
		u,
		\quad \xi(0)=x_0 \in \mathbb R^{5}.
		\intertext{We add the observation $\eta = C\xi$, defined as}
		\eta &= 
		\begin{bmatrix}
			1 & 0 & 0 & 0 & 0 \\ 
			0 & 0 & 0 & 1 & 0 
		\end{bmatrix}\svecwu.
		\label{eq:5dexa_obs}
	\end{align}
\end{subequations}
Note that with the chosen input and output operators the system is controllable and observable independent of the state $\xi(t)$ so that, in particular, at every state $x$ there exists a feedback that stabilizes the matrix $A(x)$. We compute stabilizing feedbacks by means of the SDRE \eqref{eq:sdre} and the update scheme that was defined through Theorem \ref{thm:updric}.
 
In the first approach, that we will denote by \sdre, we use only the SDRE based feedback which requires the solution of a Riccati equation at every stage of the numerical integration. In the second approach, referred to as \sylvup, we update the initial SDRE feedback according to Theorem \ref{thm:updric}. If the norm of the current update $E$ exceeds a threshold $\epsilon<1$, we reset the base feedback $P$ with the solution of the SDRE at the current state $x$. 

The parameters for the definition of the SDRE feedback and the updates are set to
\begin{equation*}
	R = 10^{-3}I_{2\times 2}, \quad Q=C\trsp C, \quad\text{and}\quad \Rdeltapmo=0.
\end{equation*}

We use \scipy's built-in integrator \textsf{odeint} with the absolute and relative accuracy tolerances set to $10^{-6}$ to integrate the closed loop system on $(0, 3]$, starting from the initial value $$x_0 = \begin{bmatrix} -1.3 & -1.4 & -1.1 & -2.0 & 0.3 \end{bmatrix}^\trp.$$ This initial value is different from the one used in \cite{BanLT07} for which the initial solution of the SDRE applied as a static feedback already stabilizes the trajectory. 

	As illustrated in Figure \ref{fig:stab_5dbnlt}(b), without stabilization, the system blows up in a short time, while with stabilization, the trajectories approach zero. This successful stabilization was achieved for the \sdre~case as well as for the \sylvup~case for varying update thresholds $\epsilon$. In the \sylvup~approach, during the time integration, Sylvester equations are solved in order to update the feedback to bound the variation in $K$, cf. Theorem \ref{thm:updric} and Lemma \ref{lem:epsexistsylvester}, and to keep the decay rate piecewise constant, cf. Figure \ref{fig:stab_5dbnlt}(a). Note that $\epsilon=0$ corresponds to the \sdre~scenario and that the jumps occur where $\norm{E}$ exceeds $\epsilon$ and where the \sylvup~scheme is reinitiated with the current SDRE solution.

	Apart from allowing for application of the theoretical results of Section \ref{sec:extlinstab}, the \sylvup~approach comes with the advantage over \sdre~that mainly Sylvester equations are solved instead of Riccati equations. In the considered five dimensional setup, the solution of the Sylvester equation \eqref{eq:epsexistsylvester} using \textsf{scipy.linalg.solve\_sylvester} takes about $100\mu s$ which is much less time than $182 \mu s$ that is needed by \textsf{scipy.linalg.solve\_continuous\_are} to solve the associated Riccati equation \eqref{eq:sdre}. The additional effort to compute $B^\trp P(I+E)^{-1}$ in each time step is $12\mu s$ and comparatively small. 
	
\begin{figure}[tbp]
  \pgfplotsset{footnotesize}
	\setlength\figureheight{4.5cm}
	\setlength\figurewidth{5.5cm}
	\subfiguretopcaptrue
	 \begin{center}
		\subfigure[]{\input{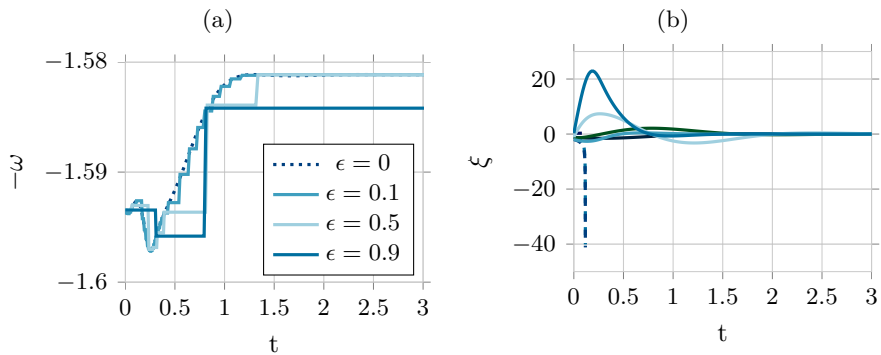}}
		\subfigure[]{\input{pics/soltrajcs.tikz}}
	\end{center}
	\caption{(a): The decay rate $-\omega$ of the closed loop matrix over time $t$ for varying $\epsilon$ and (b): the trajectories $\xi$ of the stabilized (solid lines) and of the uncontrolled (dashed lines) system \eqref{eq:sdc5dexa}.}
	\label{fig:stab_5dbnlt}
\end{figure}

	In terms of the overall computation time, however, the \sdre~approach outperforms the \sylvup~procedure in the presented example. Here, the generally faster computation of the feedback is compensated by the additional number of time steps that was required by the integrator to achieve the same accuracy. We observe that for smaller thresholds $\epsilon$, which cause more sudden changes in the feedback matrix, the integrator needs more function evaluations due to less smoothness in the system, cf. Table \ref{tab:updsdreperf}. Nevertheless, as we show in a second example, for larger systems, for which the differences in the computational complexity between the linear Sylvester and the nonlinear Riccati equation is much more significant, the \sylvup~will be more economic also in the overall costs.

	\begin{table}[tbp]
		\centering
		\begin{tabular}{l|cccc}
			Scheme & $\epsilon$ & \textsf{\#fb-switches} & \textsf{\#f-eva} &  \textsf{comp-time} \\
			\hline
\sdre     & $0$ & ---		& $245$		& $0.054s$\\
\sylvup & $0.1$ & $32$	& $1287$	& $0.271s$\\
\sylvup & $0.5$ & $7$		& $521$		& $0.110s$\\
\sylvup & $0.9$ & $2$		& $374$		& $0.078s$\\
		\end{tabular}
		\caption{Influence of $\epsilon$ on the number of switches \textsf{\#fb-switches} in the feedback definition, on the number of function evaluations \textsf{\#f-eva} in the time integrator, and on the overall computation time \textsf{comp-time} for the simulation of the \emph{5D example} \eqref{eq:sdc5dexa}.}
		\label{tab:updsdreperf}
	\end{table}

 As a second example, we consider the \emph{Chaffee Infante} equation, which is an autonomous PDE. Precisely, for the spatial coordinate $z\in(0,2)$ and time $t\in(0,3]$, we consider
 \begin{subequations}
	 \label{eq:schloegl}
	 \begin{align}
		 \dot \xi& = \partial_{zz}\xi+ 5(1 - \xi^2)\xi
		 \intertext{with boundary conditions}
		 \xi (t) \bigr|_{z=0}=0 \quad&\text{and}\quad \partial_z \xi(t)\bigr|_{z=2}=u(t)
		 \intertext{and the initial value}
	\inix &= 0.2\sin(0.5\pi z). 
	 \end{align}
 \end{subequations}
 It is known that the equilibrium point $\xi=0$ of \eqref{eq:schloegl} is unstable and that the solution for any $\inix\neq 0$ converges to one of two stable equilibria; cf. \cite{Alt14}. We discretize \eqref{eq:schloegl} by a finite-element scheme using \emph{FEniCS} \cite{LogOeRW12} and $N$ equally distributed linear hat functions which leads to an SDC system with $N$ degrees of freedom in the state and a single input. The output matrix $C\in \mathbb R^{5,N}$ is defined to observe the solution at the spatial locations $z=0$, $z=0.5$, $z=1$, $z=1.5$, and $z=2$. The parameters are chosen as $Q=C^\trp C$, $R=10^{-1}$, and $\Rdeltapmo=0$. We use \textsf{scipy.integrate.odeint} to integrate the closed-loop system as in the previous examples. Since one deals with a finite element discretization, one should use the norm induced by the corresponding mass matrix to compare the errors independently of the discretization. We mimic this scaling in the norms by scaling the prescribed tolerances $10^{-6}$ with the inverse of the elements length $2/N$.

 Both the \sylvup~and the \sdre~stabilization successfully force the system into the unstable zero state as illustrated in Figure \ref{fig:stab_schloegl}. As expected, for ever larger $N$, i.e. ever larger system sizes, the advantage of solving linear updates in the \sylvup~scheme over solving nonlinear Riccati equations in the \sdre~scheme becomes increasingly evident; cf. Table \ref{tab:updsdreperfschloegl}.

	\begin{table}[tbp]
		\centering
		\begin{tabular}{l|cccc}
			Scheme & $\epsilon$ & \textsf{\#fb-switches} & \textsf{\#f-eva} &  \textsf{comp-time} \\
			\hline
			&\multicolumn{4}{c}{$N=20$} \\ 
			\hline

\sdre		& $0$		& ---	& $442$		& $1.921s$\\
\sylvup & $0.5$ & $2$ & $838$		& $3.266s$\\
\sylvup & $0.9$ & $0$ & $451$		& $1.756s$\\
			\hline
			&\multicolumn{4}{c}{$N=40$} \\ 
\hline
\sdre	  & $0$		& --- &	$849$		& $6.267s$ \\
\sylvup & $0.5$ & $3$ & $1936$	& $10.000s$\\
\sylvup & $0.9$ & $1$ & $1186$	& $6.140s$ \\

			\hline
			&\multicolumn{4}{c}{$N=60$} \\ 
\hline
\sdre	  & $0$		& --- &	$1194$	& $15.426s$\\
\sylvup & $0.5$ & $4$ & $2240$	& $18.379s$\\
\sylvup & $0.9$ & $2$ & $1770$	& $14.140s$\\

			\hline
			&\multicolumn{4}{c}{$N=80$} \\ 
		\hline
\sdre	  & $0$			& --- & $1589$	& $42.088s$\\
\sylvup & $0.5$		& $6$ & $2953$	& $35.840s$\\
\sylvup & $0.9$		& $3$ & $2096$	& $25.486s$\\
			\hline
			&\multicolumn{4}{c}{$N=100$} \\ 
		\hline
\sdre	  & $0$			& --- & $2106$	& $90.148s$\\
\sylvup & $0.5$		& $7$ & $3778$	& $68.080s$\\
\sylvup & $0.9$		& $4$ & $2423$	& $43.816s$\\

		\end{tabular}
		\caption{Influence of $\epsilon$ on the number of switches \textsf{\#fb-switches} in the feedback definition, on the number of function evaluations \textsf{\#f-eva} in the time integrator, and on the overall computation time \textsf{comp-time} for the simulation of the stabilized \emph{Chaffee Infante} equation \eqref{eq:schloegl} with finite element discretizations on varying mesh sizes $N$.}
		\label{tab:updsdreperfschloegl}
	\end{table}

	\begin{figure}
		\pgfplotsset{width=.4\textwidth}
		\begin{tikzpicture}
			\begin{axis}[
				point meta min=-.7, point meta max=1.,
				zmin=-1, zmax=1.,
				xlabel=$z$,
				ylabel=$t$,
				ytick={1, 2, 3},
				xtick={0, 1, 2},
				xticklabels={2, 1, 0},
				zlabel={$\xi(t, z)$}
				]
				\addplot3[surf,mesh/ordering=y varies]
				table {pics/schloegl-unco.dat};
			\end{axis}
		\end{tikzpicture}
		\begin{tikzpicture}
			\begin{axis}[
				point meta min=-.7, point meta max=1.,
				zmin=-1, zmax=1.,
				xlabel=$z$,
				ylabel=$t$,
				ytick={1, 2, 3},
				xtick={0, 1, 2},
				xticklabels={2, 1, 0},
				zlabel={$\xi(t, z)$}
				]
				\addplot3[surf,mesh/ordering=y varies]
				table {pics/schloegl-sylvupd.dat};
			\end{axis}
		\end{tikzpicture}
		\caption{The uncontrolled (left) and the stabilized (right) evolution of the solution to the \emph{Chaffee Infante} equation \eqref{eq:schloegl}.}
		\label{fig:stab_schloegl}
	\end{figure}
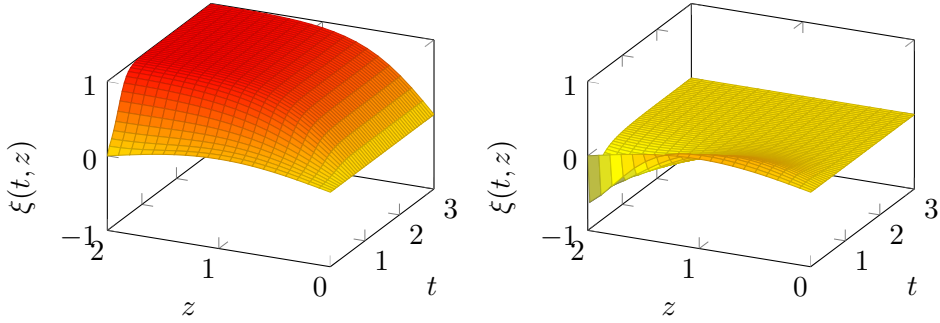

	The code and information on the system architecture used for the tests is available from the public \emph{git} repository \cite{Hei16_ext-lin-stab}.

%% file: pics/soltrajcs.tikz
%
%
%
\begin{tikzpicture}

\definecolor{color1}{rgb}{0.247058823529412,0.623529411764706,0.749019607843137}
\definecolor{color0}{rgb}{0.623529411764706,0.811764705882353,0.874509803921569}
\definecolor{color3}{rgb}{0,0.247058823529412,0.498039215686275}
\definecolor{color2}{rgb}{0,0.435294117647059,0.623529411764706}
\definecolor{color5}{rgb}{0,0.129411764705882,0.247058823529412}
\definecolor{color4}{rgb}{0,0.0588235294117647,0.372549019607843}
\definecolor{color6}{rgb}{0,0.317647058823529,0.12156862745098}

\begin{axis}[
xlabel={t},
ylabel={$\xi$},
xmin=0, xmax=3,
ymin=-50, ymax=30,
width=\figurewidth,
height=\figureheight,
tick align=outside,
xmajorgrids,
x grid style={white!80.0!black},
ymajorgrids,
y grid style={white!80.0!black},
axis line style={white!80.0!black}
]
\addplot [very thick, color0, dashed]
table {%
0 -1.3
0.000581 -1.30081358614478
0.001162 -1.30162754515355
0.001743 -1.30244187854231
0.002324 -1.30325658815214
0.002905 -1.30407167496254
0.003486 -1.30488714098527
0.004067 -1.30570298812057
0.004648 -1.30651921811921
0.005229 -1.3073358324445
0.00581 -1.30815283275524
0.006391 -1.30897022072234
0.006972 -1.30978799809136
0.007553 -1.31060616659731
0.008134 -1.31142472795657
0.008715 -1.31224368395329
0.009296 -1.31306303630685
0.009877 -1.3138827868747
0.010458 -1.31470293741142
0.011039 -1.31552348975117
0.01162 -1.31634444573787
0.012201 -1.3171658072213
0.012782 -1.31798757607362
0.013363 -1.31880975418934
0.013944 -1.31963234347773
0.014525 -1.32045534584775
0.015106 -1.32127876324161
0.015687 -1.32210259762249
0.016268 -1.32292685097322
0.016849 -1.32375152529627
0.01743 -1.3245766226034
0.018011 -1.32540214492279
0.018592 -1.32622809430591
0.019173 -1.32705447282254
0.019754 -1.32788128256074
0.020335 -1.32870852562913
0.020916 -1.32953620416955
0.021497 -1.33036432033728
0.022078 -1.33119287630558
0.022659 -1.33202187426758
0.02324 -1.33285131643631
0.023821 -1.33368120504746
0.024402 -1.33451154238939
0.024983 -1.33534233073427
0.025564 -1.3361735723904
0.026145 -1.33700526968692
0.026726 -1.33783742497494
0.027307 -1.3386700406432
0.027888 -1.33950311910101
0.028469 -1.34033666277849
0.02905 -1.34117067413107
0.029631 -1.34200515564003
0.030212 -1.34284010981304
0.030793 -1.34367553918499
0.031374 -1.34451144632041
0.031955 -1.34534783381129
0.032536 -1.34618470427828
0.033117 -1.34702206037154
0.033698 -1.34785990477134
0.034279 -1.3486982401886
0.03486 -1.34953706936372
0.035441 -1.35037639506244
0.036022 -1.35121622008548
0.036603 -1.35205654726586
0.037184 -1.35289737946895
0.037765 -1.35373871959292
0.038346 -1.35458057056926
0.038927 -1.35542293536555
0.039508 -1.35626581698778
0.040089 -1.35710921847462
0.04067 -1.35795314290077
0.041251 -1.3587975933776
0.041832 -1.35964257305377
0.042413 -1.36048808511573
0.042994 -1.36133413279337
0.043575 -1.36218071935943
0.044156 -1.36302784812364
0.044737 -1.36387552243712
0.045318 -1.36472374569305
0.045899 -1.3655725213273
0.04648 -1.36642185281931
0.047061 -1.36727174370051
0.047642 -1.36812219754857
0.048223 -1.36897321798492
0.048804 -1.36982480867865
0.049385 -1.3706769733473
0.049966 -1.37152971575771
0.050547 -1.37238303972756
0.051128 -1.37323694913751
0.051709 -1.37409144791862
0.05229 -1.37494654005368
0.052871 -1.3758022295809
0.053452 -1.37665852059491
0.054033 -1.37751541724767
0.054614 -1.37837292375139
0.055195 -1.37923104439091
0.055776 -1.38008978350945
0.056357 -1.38094914551311
0.056938 -1.38180913487366
0.057519 -1.38266975612984
0.0581 -1.38353101389251
0.058681 -1.38439291284614
0.059262 -1.38525545774466
0.059843 -1.38611865341577
0.060424 -1.38698250476236
0.061005 -1.38784701676547
0.061586 -1.38871219449794
0.062167 -1.389578043111
0.062748 -1.39044456783817
0.063329 -1.39131177399896
0.06391 -1.39217966700068
0.064491 -1.39304825234842
0.065072 -1.39391753565016
0.065653 -1.39478752260322
0.066234 -1.39565821900374
0.066815 -1.39652963074891
0.067396 -1.39740176384036
0.067977 -1.39827462440156
0.068558 -1.39914821866625
0.069139 -1.40002255298123
0.06972 -1.40089763381188
0.070301 -1.40177346774956
0.070882 -1.40265006151549
0.071463 -1.403527421958
0.072044 -1.40440555605852
0.072625 -1.40528447093813
0.073206 -1.40616417387504
0.073787 -1.4070446722902
0.074368 -1.40792597375564
0.074949 -1.40880808600008
0.07553 -1.40969101692924
0.076111 -1.41057477462086
0.076692 -1.41145936732305
0.077273 -1.41234480346402
0.077854 -1.41323109166423
0.078435 -1.41411824075535
0.079016 -1.4150062597671
0.079597 -1.41589515794067
0.080178 -1.41678494474418
0.080759 -1.41767562987448
0.08134 -1.41856722326306
0.081921 -1.41945973509293
0.082502 -1.42035317581914
0.083083 -1.42124755615816
0.083664 -1.42214288710504
0.084245 -1.42303917996492
0.084826 -1.42393644634864
0.085407 -1.42483469818095
0.085988 -1.42573394772563
0.086569 -1.4266342076151
0.08715 -1.42753549084457
0.087731 -1.42843781080161
0.088312 -1.42934118128367
0.088893 -1.43024561650989
0.089474 -1.43115113116435
0.090055 -1.43205774040526
0.090636 -1.43296545988581
0.091217 -1.43387430580747
0.091798 -1.43478429492209
0.092379 -1.43569544457071
0.09296 -1.43660777274175
0.093541 -1.43752129808455
0.094122 -1.43843603996432
0.094703 -1.43935201849327
0.095284 -1.44026925460275
0.095865 -1.44118777007808
0.096446 -1.44210758763444
0.097027 -1.44302873097343
0.097608 -1.4439512248536
0.098189 -1.44487509518973
0.09877 -1.44580036912754
0.099351 -1.44672707513925
0.099932 -1.4476552431533
0.100513 -1.44858490466192
0.101094 -1.44951609288128
0.101675 -1.45044884289138
0.102256 -1.45138319184174
0.102837 -1.45231917914068
0.103418 -1.4532568467004
0.103999 -1.45419623921101
0.10458 -1.45513740446191
0.105161 -1.45608039371835
0.105742 -1.45702526215906
0.106323 -1.45797206939508
0.106904 -1.45892088010103
0.107485 -1.45987176476743
0.108066 -1.46082480063614
0.108647 -1.46178007283368
0.109228 -1.46273767581592
0.109809 -1.46369771521008
0.11039 -1.46466031019038
0.110971 -1.46562559664936
0.111552 -1.46659373151647
0.112133 -1.46756489880504
0.112714 -1.46853931844378
0.113295 -1.46951725974416
0.113876 -1.47049906325718
0.114457 -1.47148517923348
0.115038 -1.47247624376713
0.115619 -1.47347326067113
0.1162 -1.4744782345295
};
\addplot [very thick, color1, dashed]
table {%
0 -1.4
0.000581 -1.40064045711134
0.001162 -1.40128363658192
0.001743 -1.40192955989842
0.002324 -1.40257825328973
0.002905 -1.40322973092255
0.003486 -1.4038840216363
0.004067 -1.40454115263094
0.004648 -1.40520114923006
0.005229 -1.40586403314768
0.00581 -1.4065298285527
0.006391 -1.40719855980684
0.006972 -1.40787025297737
0.007553 -1.40854493389062
0.008134 -1.40922262794761
0.008715 -1.40990336211682
0.009296 -1.41058716186845
0.009877 -1.41127405586682
0.010458 -1.41196407039905
0.011039 -1.41265723364122
0.01162 -1.41335357399089
0.012201 -1.41405311997486
0.012782 -1.41475590064251
0.013363 -1.41546194556584
0.013944 -1.41617128469963
0.014525 -1.41688394810344
0.015106 -1.41759996656092
0.015687 -1.41831937135171
0.016268 -1.41904219422893
0.016849 -1.41976846741836
0.01743 -1.420498223496
0.018011 -1.4212314954722
0.018592 -1.42196831687397
0.019173 -1.42270872168514
0.019754 -1.42345274434634
0.020335 -1.42420041983384
0.020916 -1.42495178409957
0.021497 -1.42570687338394
0.022078 -1.4264657243738
0.022659 -1.42722837426795
0.02324 -1.42799486077717
0.023821 -1.42876522222667
0.024402 -1.42953949864834
0.024983 -1.43031772927169
0.025564 -1.43109995443148
0.026145 -1.43188621500789
0.026726 -1.4326765524729
0.027307 -1.43347100947933
0.027888 -1.43426962921627
0.028469 -1.43507245541733
0.02905 -1.43587953253002
0.029631 -1.43669090573607
0.030212 -1.43750662097177
0.030793 -1.4383267249656
0.031374 -1.43915126536769
0.031955 -1.43998029061058
0.032536 -1.44081384996597
0.033117 -1.44165199358411
0.033698 -1.44249477251589
0.034279 -1.44334223873498
0.03486 -1.44419444511645
0.035441 -1.44505144534407
0.036022 -1.44591329415348
0.036603 -1.44678004727469
0.037184 -1.44765176144106
0.037765 -1.4485284944091
0.038346 -1.44941030497829
0.038927 -1.45029725313658
0.039508 -1.45118940018407
0.040089 -1.45208680840679
0.04067 -1.4529895412569
0.041251 -1.45389766338181
0.041832 -1.45481124064799
0.042413 -1.45573034016484
0.042994 -1.45665503059344
0.043575 -1.45758538210368
0.044156 -1.45852146604157
0.044737 -1.45946335516678
0.045318 -1.46041112368272
0.045899 -1.46136484726648
0.04648 -1.46232460310721
0.047061 -1.46329047041157
0.047642 -1.46426253004135
0.048223 -1.46524086435219
0.048804 -1.46622555742147
0.049385 -1.46721669508616
0.049966 -1.46821436498078
0.050547 -1.46921865662504
0.051128 -1.4702296621677
0.051709 -1.47124747559011
0.05229 -1.47227219277753
0.052871 -1.47330391173771
0.053452 -1.47434273264952
0.054033 -1.47538875791162
0.054614 -1.4764420923165
0.055195 -1.47750284385235
0.055776 -1.47857112274942
0.056357 -1.47964704176011
0.056938 -1.48073071632802
0.057519 -1.4818222646577
0.0581 -1.48292180808333
0.058681 -1.48402947114954
0.059262 -1.48514538126892
0.059843 -1.48626966901932
0.060424 -1.48740246822608
0.061005 -1.48854391615159
0.061586 -1.48969415446567
0.062167 -1.49085332824724
0.062748 -1.49202158624341
0.063329 -1.49319908111749
0.06391 -1.49438596955511
0.064491 -1.49558241301841
0.065072 -1.49678857811601
0.065653 -1.49800463549752
0.066234 -1.49923076056726
0.066815 -1.50046713362664
0.067396 -1.50171394011355
0.067977 -1.50297137201433
0.068558 -1.50423962685135
0.069139 -1.50551890787116
0.06972 -1.50680942445977
0.070301 -1.50811139277648
0.070882 -1.50942503604724
0.071463 -1.51075058421455
0.072044 -1.51208827443856
0.072625 -1.51343835162302
0.073206 -1.51480106995696
0.073787 -1.51617669152519
0.074368 -1.51756548701098
0.074949 -1.51896773614482
0.07553 -1.52038372962683
0.076111 -1.52181376853717
0.076692 -1.52325816407
0.077273 -1.52471723840332
0.077854 -1.52619132586272
0.078435 -1.52768077479295
0.079016 -1.52918594603895
0.079597 -1.53070721424617
0.080178 -1.53224496949499
0.080759 -1.53379961729673
0.08134 -1.53537157907126
0.081921 -1.53696129392416
0.082502 -1.53856922082097
0.083083 -1.54019583718185
0.083664 -1.54184164068275
0.084245 -1.54350715293899
0.084826 -1.54519291866604
0.085407 -1.54689950644217
0.085988 -1.54862751165762
0.086569 -1.55037756011307
0.08715 -1.55215030686198
0.087731 -1.55394643996674
0.088312 -1.55576668250912
0.088893 -1.55761179373953
0.089474 -1.55948257485777
0.090055 -1.56137986981937
0.090636 -1.5633045678964
0.091217 -1.56525761129953
0.091798 -1.56723999470734
0.092379 -1.56925277073393
0.09296 -1.57129705867492
0.093541 -1.57337404581358
0.094122 -1.57548499601636
0.094703 -1.57763125386093
0.095284 -1.57981425636218
0.095865 -1.58203553789029
0.096446 -1.5842967431259
0.097027 -1.58659963606688
0.097608 -1.58894611219922
0.098189 -1.59133821659221
0.09877 -1.59377815738226
0.099351 -1.59626832349975
0.099932 -1.59881131036701
0.100513 -1.60140994088194
0.101094 -1.6040672994086
0.101675 -1.60678676119929
0.102256 -1.60957203923068
0.102837 -1.61242722794349
0.103418 -1.61535686239274
0.103999 -1.61836598691883
0.10458 -1.6214602395948
0.105161 -1.62464595604345
0.105742 -1.62793029572203
0.106323 -1.63132139860392
0.106904 -1.63482858603096
0.107485 -1.63846261349132
0.108066 -1.64223600545158
0.108647 -1.64616348452214
0.109228 -1.65026255513593
0.109809 -1.65455430478539
0.11039 -1.65906451706995
0.110971 -1.66382528191291
0.111552 -1.6688773991621
0.112133 -1.67427408795248
0.112714 -1.68008703694697
0.113295 -1.68641684982653
0.113876 -1.69341261144941
0.114457 -1.70131268601899
0.115038 -1.71054436328904
0.115619 -1.7220390747932
0.1162 -1.73898016803232
};
\addplot [very thick, color2, dashed]
table {%
0 -1.1
0.000581 -1.10466716786546
0.001162 -1.10937284972568
0.001743 -1.11411751667085
0.002324 -1.11890174699963
0.002905 -1.12372585882882
0.003486 -1.12859049285414
0.004067 -1.13349625226085
0.004648 -1.1384437045956
0.005229 -1.14343334884747
0.00581 -1.14846573063197
0.006391 -1.15354140039757
0.006972 -1.15866096286321
0.007553 -1.16382501508165
0.008134 -1.16903414057152
0.008715 -1.17428897305988
0.009296 -1.17959009828147
0.009877 -1.18493820428351
0.010458 -1.19033390305466
0.011039 -1.19577786874939
0.01162 -1.20127078247578
0.012201 -1.20681332924269
0.012782 -1.21240621098322
0.013363 -1.21805014655482
0.013944 -1.22374586845859
0.014525 -1.22949411631637
0.015106 -1.23529565140144
0.015687 -1.24115125128795
0.016268 -1.24706170932189
0.016849 -1.25302783462289
0.01743 -1.259050452347
0.018011 -1.26513040350619
0.018592 -1.27126854479149
0.019173 -1.27746574870154
0.019754 -1.28372290354265
0.020335 -1.29004091700327
0.020916 -1.29642073610224
0.021497 -1.30286331603447
0.022078 -1.30936962733384
0.022659 -1.31594065884211
0.02324 -1.32257741770898
0.023821 -1.32928093417352
0.024402 -1.3360523125129
0.024983 -1.342892614002
0.025564 -1.34980294589995
0.026145 -1.35678443533448
0.026726 -1.36383823153848
0.027307 -1.37096553451152
0.027888 -1.37816756360598
0.028469 -1.3854455579055
0.02905 -1.39280078444846
0.029631 -1.40023453919561
0.030212 -1.40774814799777
0.030793 -1.41534296874041
0.031374 -1.42302039978216
0.031955 -1.43078186999699
0.032536 -1.43862884223757
0.033117 -1.44656281560252
0.033698 -1.45458532652341
0.034279 -1.46269794985169
0.03486 -1.47090229898004
0.035441 -1.47920002437178
0.036022 -1.48759281957046
0.036603 -1.49608242051294
0.037184 -1.50467060632616
0.037765 -1.51335920036249
0.038346 -1.52215007123517
0.038927 -1.53104514256127
0.039508 -1.54004640119144
0.040089 -1.54915587428452
0.04067 -1.55837564145487
0.041251 -1.56770783644899
0.041832 -1.57715464846363
0.042413 -1.58671832346404
0.042994 -1.5964011848864
0.043575 -1.60620563041528
0.044156 -1.61613410903983
0.044737 -1.62618913691112
0.045318 -1.636373299081
0.045899 -1.64668925123896
0.04648 -1.65713972205778
0.047061 -1.66772754984341
0.047642 -1.6784556554603
0.048223 -1.6893270300176
0.048804 -1.70034475113465
0.049385 -1.71151198525741
0.049966 -1.72283198997488
0.050547 -1.73430812015354
0.051128 -1.74594388443767
0.051709 -1.75774288352314
0.05229 -1.76970881504244
0.052871 -1.78184548974279
0.053452 -1.79415683462878
0.054033 -1.80664689610499
0.054614 -1.81931985334306
0.055195 -1.83218008289497
0.055776 -1.84523208003605
0.056357 -1.85848048079825
0.056938 -1.8719300749307
0.057519 -1.88558581076505
0.0581 -1.89945282642506
0.058681 -1.91353645566413
0.059262 -1.92784219639137
0.059843 -1.94237573559061
0.060424 -1.95714295528291
0.061005 -1.97214994835126
0.061586 -1.98740310610404
0.062167 -2.00290902495798
0.062748 -2.01867452865139
0.063329 -2.03470668943951
0.06391 -2.05101283625046
0.064491 -2.06760062498125
0.065072 -2.08447807194952
0.065653 -2.10165344589077
0.066234 -2.11913533436512
0.066815 -2.13693265539829
0.067396 -2.15505467917486
0.067977 -2.17351117091032
0.068558 -2.19231228319433
0.069139 -2.21146857239948
0.06972 -2.23099103855231
0.070301 -2.25089119091324
0.070882 -2.27118107592875
0.071463 -2.29187323407916
0.072044 -2.31298075078801
0.072625 -2.3345173117049
0.073206 -2.35649737585491
0.073787 -2.37893600864814
0.074368 -2.40184895767417
0.074949 -2.42525269919258
0.07553 -2.44916466306403
0.076111 -2.47360315441008
0.076692 -2.4985873143351
0.077273 -2.52413721777492
0.077854 -2.55027401435715
0.078435 -2.57702016156226
0.079016 -2.60439921569696
0.079597 -2.63243599101974
0.080178 -2.66115677104742
0.080759 -2.69058929210161
0.08134 -2.72076279389404
0.081921 -2.75170825774967
0.082502 -2.78345870425712
0.083083 -2.81604896691946
0.083664 -2.84951593619791
0.084245 -2.88389909451345
0.084826 -2.91924036283141
0.085407 -2.95558419113407
0.085988 -2.99297799891138
0.086569 -3.03147272794135
0.08715 -3.07112261276147
0.087731 -3.11198577348536
0.088312 -3.15412451153237
0.088893 -3.19760545614091
0.089474 -3.24250055431725
0.090055 -3.28888714577514
0.090636 -3.3368483675102
0.091217 -3.3864745218313
0.091798 -3.43786289413239
0.092379 -3.49111873972872
0.09296 -3.54635695357363
0.093541 -3.60370220819561
0.094122 -3.66329063944849
0.094703 -3.72527056104401
0.095284 -3.78980492756218
0.095865 -3.85707221496533
0.096446 -3.92726923415727
0.097027 -4.0006129783387
0.097608 -4.07734335023948
0.098189 -4.15772737763312
0.09877 -4.24206231806646
0.099351 -4.33067987794402
0.099932 -4.4239528246811
0.100513 -4.5223002461417
0.101094 -4.62619675464753
0.101675 -4.73618054338271
0.102256 -4.85286704288598
0.102837 -4.97696181993504
0.103418 -5.10927907461059
0.103999 -5.2507638226205
0.10458 -5.40252034082612
0.105161 -5.56584914800981
0.105742 -5.74229362359004
0.106323 -5.93370106986444
0.106904 -6.14230506790678
0.107485 -6.37083485391925
0.108066 -6.62266932424731
0.108647 -6.9020464314251
0.109228 -7.21436724819436
0.109809 -7.56664492818538
0.11039 -7.96817629886783
0.110971 -8.43159731044983
0.111552 -8.97460809780953
0.112133 -9.62291198681904
0.112714 -10.4155736455515
0.113295 -11.4155092744196
0.113876 -12.732219391865
0.114457 -14.5781833418948
0.115038 -17.4400337178838
0.115619 -22.8080905639363
0.1162 -40.1064067945064
};
\addplot [very thick, color3, dashed]
table {%
0 -2
0.000581 -2.00547537263717
0.001162 -2.01099048929568
0.001743 -2.0165458223406
0.002324 -2.0221419516088
0.002905 -2.02777919602036
0.003486 -2.03345819793915
0.004067 -2.03917956212893
0.004648 -2.04494385757185
0.005229 -2.05075158441737
0.00581 -2.0566032896285
0.006391 -2.06249952500998
0.006972 -2.06844089665612
0.007553 -2.07442800299236
0.008134 -2.08046142890524
0.008715 -2.08654180950661
0.009296 -2.09266973189986
0.009877 -2.09884588553524
0.010458 -2.1050708837788
0.011039 -2.11134540218184
0.01162 -2.11767012325172
0.012201 -2.12404573339808
0.012782 -2.13047293596036
0.013363 -2.13695245120784
0.013944 -2.14348501305673
0.014525 -2.15007136254262
0.015106 -2.15671226236141
0.015687 -2.16340849151492
0.016268 -2.17016084478154
0.016849 -2.17697013271802
0.01743 -2.18383718191853
0.018011 -2.19076283483677
0.018592 -2.19774794961145
0.019173 -2.20479340019316
0.019754 -2.21190007634435
0.020335 -2.21906888721415
0.020916 -2.2263007812883
0.021497 -2.23359671523217
0.022078 -2.24095766105429
0.022659 -2.24838460907552
0.02324 -2.25587856792907
0.023821 -2.26344056934234
0.024402 -2.2710717190895
0.024983 -2.27877307994219
0.025564 -2.28654576066286
0.026145 -2.29439088988709
0.026726 -2.30230961836035
0.027307 -2.31030314760142
0.027888 -2.31837269848646
0.028469 -2.32651951162741
0.02905 -2.33474485559607
0.029631 -2.34305002789185
0.030212 -2.35143635590947
0.030793 -2.35990519908362
0.031374 -2.36845795732763
0.031955 -2.37709606107567
0.032536 -2.38582097474607
0.033117 -2.39463419900881
0.033698 -2.40353727187247
0.034279 -2.4125317697713
0.03486 -2.42161930768633
0.035441 -2.43080153767407
0.036022 -2.44008015487766
0.036603 -2.4494568968395
0.037184 -2.45893354429797
0.037765 -2.46851192222289
0.038346 -2.47819390085095
0.038927 -2.48798140542892
0.039508 -2.49787642444362
0.040089 -2.50788098669605
0.04067 -2.5179971734491
0.041251 -2.52822712010414
0.041832 -2.53857301751922
0.042413 -2.54903711332742
0.042994 -2.55962173263967
0.043575 -2.57032927482217
0.044156 -2.58116219055224
0.044737 -2.59212299767593
0.045318 -2.60321428294695
0.045899 -2.6144387037637
0.04648 -2.62579899051486
0.047061 -2.63729798322997
0.047642 -2.64893860450435
0.048223 -2.66072384718516
0.048804 -2.67265679063718
0.049385 -2.68474060305933
0.049966 -2.69697854380112
0.050547 -2.7093739694973
0.051128 -2.72193039056891
0.051709 -2.73465140949634
0.05229 -2.74754072570444
0.052871 -2.7606021517409
0.053452 -2.77383961641901
0.054033 -2.7872571679603
0.054614 -2.8008589873619
0.055195 -2.81464945301055
0.055776 -2.82863306202478
0.056357 -2.8428144522885
0.056938 -2.85719841541172
0.057519 -2.87178990159607
0.0581 -2.88659405084464
0.058681 -2.90161619879968
0.059262 -2.91686184526843
0.059843 -2.93233667914247
0.060424 -2.94804658436027
0.061005 -2.96399765573216
0.061586 -2.98019628650451
0.062167 -2.99664907504207
0.062748 -3.01336284704138
0.063329 -3.03034467672639
0.06391 -3.04760189500455
0.064491 -3.06514215976337
0.065072 -3.08297348932249
0.065653 -3.10110415443018
0.066234 -3.11954274467074
0.066815 -3.13829818010556
0.067396 -3.15737973296663
0.067977 -3.17679717052945
0.068558 -3.19656064745584
0.069139 -3.21668072220282
0.06972 -3.23716839689398
0.070301 -3.25803518289977
0.070882 -3.27929312878996
0.071463 -3.30095477718132
0.072044 -3.32303321564733
0.072625 -3.3455421320014
0.073206 -3.36849598744709
0.073787 -3.39190984958704
0.074368 -3.41579946821798
0.074949 -3.44018132182144
0.07553 -3.46507284249576
0.076111 -3.49049233761581
0.076692 -3.51645895055485
0.077273 -3.54299275853376
0.077854 -3.57011491348175
0.078435 -3.59784787519844
0.079016 -3.62621520232661
0.079597 -3.65524171147854
0.080178 -3.68495368854305
0.080759 -3.71537887223221
0.08134 -3.74654650466696
0.081921 -3.77848756960091
0.082502 -3.81123509007157
0.083083 -3.84482390205092
0.083664 -3.87929089848976
0.084245 -3.91467556432049
0.084826 -3.95101982304166
0.085407 -3.988368127191
0.085988 -4.02676789883699
0.086569 -4.06627008236036
0.08715 -4.10692891492545
0.087731 -4.14880251929786
0.088312 -4.1919531995747
0.088893 -4.2364475876985
0.089474 -4.28235763340729
0.090055 -4.32976067917469
0.090636 -4.37873986478509
0.091217 -4.42938549536648
0.091798 -4.48179485916269
0.092379 -4.53607321437071
0.09296 -4.59233545885915
0.093541 -4.65070626810589
0.094122 -4.71132178094975
0.094703 -4.77433031412408
0.095284 -4.83989482526899
0.095865 -4.90819379344649
0.096446 -4.97942403270148
0.097027 -5.05380253941873
0.097608 -5.13156921955692
0.098189 -5.21299110416546
0.09877 -5.29836545411587
0.099351 -5.38802397918847
0.099932 -5.48233945022818
0.100513 -5.5817309585848
0.101094 -5.68667312012523
0.101675 -5.79770413164101
0.102256 -5.91543942734438
0.102837 -6.04058457775613
0.103418 -6.17395378677394
0.103999 -6.31649207400116
0.10458 -6.46930372027829
0.105161 -6.63368924845635
0.105742 -6.81119204211839
0.106323 -7.00365940782974
0.106904 -7.21332493104349
0.107485 -7.44291785246204
0.108066 -7.69581707306314
0.108647 -7.97626055015782
0.109228 -8.28964936142432
0.109809 -8.64299666560692
0.11039 -9.04559929548805
0.110971 -9.51009320681356
0.111552 -10.0541785402569
0.112133 -10.7035586277843
0.112714 -11.4972981439161
0.113295 -12.4983132959378
0.113876 -13.8161046096927
0.114457 -15.6631514372702
0.115038 -18.5260863810553
0.115619 -23.8952294955413
0.1162 -41.1946337079931
};
\addplot [very thick, color4, dashed]
table {%
0 0.3
0.000581 0.3
0.001162 0.3
0.001743 0.3
0.002324 0.3
0.002905 0.3
0.003486 0.3
0.004067 0.3
0.004648 0.3
0.005229 0.3
0.00581 0.3
0.006391 0.3
0.006972 0.3
0.007553 0.3
0.008134 0.3
0.008715 0.3
0.009296 0.3
0.009877 0.3
0.010458 0.3
0.011039 0.3
0.01162 0.3
0.012201 0.3
0.012782 0.3
0.013363 0.3
0.013944 0.3
0.014525 0.3
0.015106 0.3
0.015687 0.3
0.016268 0.3
0.016849 0.3
0.01743 0.3
0.018011 0.3
0.018592 0.3
0.019173 0.3
0.019754 0.3
0.020335 0.3
0.020916 0.3
0.021497 0.3
0.022078 0.3
0.022659 0.3
0.02324 0.3
0.023821 0.3
0.024402 0.3
0.024983 0.3
0.025564 0.3
0.026145 0.3
0.026726 0.3
0.027307 0.3
0.027888 0.3
0.028469 0.3
0.02905 0.3
0.029631 0.3
0.030212 0.3
0.030793 0.3
0.031374 0.3
0.031955 0.3
0.032536 0.3
0.033117 0.3
0.033698 0.3
0.034279 0.3
0.03486 0.3
0.035441 0.3
0.036022 0.3
0.036603 0.3
0.037184 0.3
0.037765 0.3
0.038346 0.3
0.038927 0.3
0.039508 0.3
0.040089 0.3
0.04067 0.3
0.041251 0.3
0.041832 0.3
0.042413 0.3
0.042994 0.3
0.043575 0.3
0.044156 0.3
0.044737 0.3
0.045318 0.3
0.045899 0.3
0.04648 0.3
0.047061 0.3
0.047642 0.3
0.048223 0.3
0.048804 0.3
0.049385 0.3
0.049966 0.3
0.050547 0.3
0.051128 0.3
0.051709 0.3
0.05229 0.3
0.052871 0.3
0.053452 0.3
0.054033 0.3
0.054614 0.3
0.055195 0.3
0.055776 0.3
0.056357 0.3
0.056938 0.3
0.057519 0.3
0.0581 0.3
0.058681 0.3
0.059262 0.3
0.059843 0.3
0.060424 0.3
0.061005 0.3
0.061586 0.3
0.062167 0.3
0.062748 0.3
0.063329 0.3
0.06391 0.3
0.064491 0.3
0.065072 0.3
0.065653 0.3
0.066234 0.3
0.066815 0.3
0.067396 0.3
0.067977 0.3
0.068558 0.3
0.069139 0.3
0.06972 0.3
0.070301 0.3
0.070882 0.3
0.071463 0.3
0.072044 0.3
0.072625 0.3
0.073206 0.3
0.073787 0.3
0.074368 0.3
0.074949 0.3
0.07553 0.3
0.076111 0.3
0.076692 0.3
0.077273 0.3
0.077854 0.3
0.078435 0.3
0.079016 0.3
0.079597 0.3
0.080178 0.3
0.080759 0.3
0.08134 0.3
0.081921 0.3
0.082502 0.3
0.083083 0.3
0.083664 0.3
0.084245 0.3
0.084826 0.3
0.085407 0.3
0.085988 0.3
0.086569 0.3
0.08715 0.3
0.087731 0.3
0.088312 0.3
0.088893 0.3
0.089474 0.3
0.090055 0.3
0.090636 0.3
0.091217 0.3
0.091798 0.3
0.092379 0.3
0.09296 0.3
0.093541 0.3
0.094122 0.3
0.094703 0.3
0.095284 0.3
0.095865 0.3
0.096446 0.3
0.097027 0.3
0.097608 0.3
0.098189 0.3
0.09877 0.3
0.099351 0.3
0.099932 0.3
0.100513 0.3
0.101094 0.3
0.101675 0.3
0.102256 0.3
0.102837 0.3
0.103418 0.3
0.103999 0.3
0.10458 0.3
0.105161 0.3
0.105742 0.3
0.106323 0.3
0.106904 0.3
0.107485 0.3
0.108066 0.3
0.108647 0.3
0.109228 0.3
0.109809 0.3
0.11039 0.3
0.110971 0.3
0.111552 0.3
0.112133 0.3
0.112714 0.3
0.113295 0.3
0.113876 0.3
0.114457 0.3
0.115038 0.3
0.115619 0.3
0.1162 0.3
};
\addplot [very thick, color5]
table {%
0 -1.3
0.015 -1.32108047574918
0.03 -1.34215618491985
0.045 -1.36299488791344
0.06 -1.38338548777137
0.075 -1.40313722710852
0.09 -1.42207889938879
0.105 -1.44005803964507
0.12 -1.45694014139945
0.135 -1.47260780294625
0.15 -1.4869598559223
0.165 -1.49991047804178
0.18 -1.51138822322554
0.195 -1.52133519036643
0.21 -1.52970614016751
0.225 -1.53646711756587
0.24 -1.54159433106043
0.255 -1.54507289509847
0.27 -1.5468960537308
0.285 -1.54706353885935
0.3 -1.54558096817846
0.315 -1.54245935699865
0.33 -1.53771424693265
0.345 -1.53136561239051
0.36 -1.52343759990611
0.375 -1.51395841630348
0.39 -1.50296018343717
0.405 -1.4904789085014
0.42 -1.47655404914565
0.435 -1.46122857057363
0.45 -1.44454862400358
0.465 -1.42656324516552
0.48 -1.40732412077712
0.495 -1.38688530321693
0.51 -1.36530291614711
0.525 -1.34263486297381
0.54 -1.31894054278912
0.555 -1.2942805701581
0.57 -1.26871650250425
0.585 -1.24231052968461
0.6 -1.21512533969942
0.615 -1.18722389970412
0.63 -1.15866923744411
0.645 -1.12952424262099
0.66 -1.09985147808917
0.675 -1.06971298271767
0.69 -1.03917013927425
0.705 -1.00828352718952
0.72 -0.977112783302069
0.735 -0.945716467907192
0.75 -0.914151946950452
0.765 -0.882475277491955
0.78 -0.85074111395246
0.795 -0.819002606198769
0.81 -0.787311321916745
0.825 -0.755717174829793
0.84 -0.724268357120801
0.855 -0.693011274533661
0.87 -0.661990490416349
0.885 -0.63124868589958
0.9 -0.600826626424761
0.915 -0.570763134334739
0.93 -0.541095067365556
0.945 -0.511857289014399
0.96 -0.483082657967956
0.975 -0.454802032250002
0.99 -0.427044260370424
1.005 -0.399836211813418
1.02 -0.373202760655622
1.035 -0.347166805798926
1.05 -0.321749292288418
1.065 -0.296969236699729
1.08 -0.272843895691284
1.095 -0.249388258174777
1.11 -0.226615913704514
1.125 -0.204538486637002
1.14 -0.183165874792789
1.155 -0.162506262004902
1.17 -0.142566174691988
1.185 -0.123350599403073
1.2 -0.104862740448015
1.215 -0.0871045502673025
1.23 -0.0700764194382887
1.245 -0.0537773553791235
1.26 -0.0382050221936521
1.275 -0.0233557981389388
1.29 -0.00922483215079375
1.305 0.00419389598886629
1.32 0.0169075121333739
1.335 0.0289241904599909
1.35 0.0402530972614023
1.365 0.0509043449187441
1.38 0.060888900134326
1.395 0.0702185552008815
1.41 0.0789058682617589
1.425 0.0869640918965827
1.44 0.0944071329518499
1.455 0.101249490526536
1.47 0.107506203297291
1.485 0.113192796060542
1.5 0.11832522775205
1.515 0.122919841213864
1.53 0.12699331383995
1.545 0.130562609167958
1.56 0.133644931007053
1.575 0.136257677863861
1.59 0.138418399324499
1.605 0.140144754721275
1.62 0.14145446930951
1.635 0.142365299754505
1.65 0.142894996004793
1.665 0.143061262668239
1.68 0.142881726841305
1.695 0.142373905448899
1.71 0.141555173807929
1.725 0.140442736898986
1.74 0.139053601382394
1.755 0.137404549643694
1.77 0.13551211552889
1.785 0.133392561386647
1.8 0.131061856749818
1.815 0.128535658485966
1.83 0.12582929279842
1.845 0.122957738311028
1.86 0.119935610890722
1.875 0.116777149751625
1.89 0.113496205226511
1.905 0.110106227047922
1.92 0.106620255559995
1.935 0.103050912154898
1.95 0.0994103948342926
1.965 0.0957104681964567
1.98 0.0919624606071997
1.995 0.0881772615702193
2.01 0.084365316731044
2.025 0.0805366290344935
2.04 0.07670075570805
2.055 0.0728668106949564
2.07 0.0690434641276767
2.085 0.0652389432485873
2.1 0.0614610434768798
2.115 0.0577171227938376
2.13 0.0540141122813647
2.145 0.0503585166305204
2.16 0.0467564250172646
2.175 0.0432135131320201
2.19 0.0397350542916587
2.205 0.036325923965402
2.22 0.0329906097050265
2.235 0.0297332198978253
2.25 0.0265574915931455
2.265 0.0234668002274471
2.28 0.0204641700023227
2.295 0.0175522834259949
2.31 0.0147334915959878
2.325 0.0120098250118198
2.34 0.00938300402365662
2.355 0.006854449964959
2.37 0.00442529633069416
2.385 0.00209639924933643
2.4 -0.000131651408501451
2.415 -0.00225851888037501
2.43 -0.00428411015417793
2.445 -0.00620856426215861
2.46 -0.00803224022807816
2.475 -0.00975570659884489
2.49 -0.0113797301768905
2.505 -0.0129052649204226
2.52 -0.014333440473872
2.535 -0.0156655523351677
2.55 -0.0169030506213556
2.565 -0.0180475286871222
2.58 -0.0191007149645709
2.595 -0.0200644609255216
2.61 -0.020940730844875
2.625 -0.021731594908151
2.64 -0.0224392163440888
2.655 -0.0230658431979345
2.67 -0.023613801845157
2.685 -0.0240854839149983
2.7 -0.0244833406325681
2.715 -0.0248098753859947
2.73 -0.0250676319423976
2.745 -0.0252591911059703
2.76 -0.0253871616469438
2.775 -0.02545417103428
2.79 -0.0254628602120121
2.805 -0.0254158794244921
2.82 -0.0253158798829151
2.835 -0.0251655091030556
2.85 -0.0249674017947314
2.865 -0.0247241767938815
2.88 -0.0244384331564323
2.895 -0.0241127441841678
2.91 -0.0237496519142214
2.925 -0.0233516644674234
2.94 -0.0229212515750659
2.955 -0.0224608405132436
2.97 -0.0219728129329826
2.985 -0.0214595017397331
3 -0.0209231881691948
};
\addplot [very thick, color6]
table {%
0 -1.4
0.015 -1.40790637565468
0.03 -1.3996082201959
0.045 -1.37654088226767
0.06 -1.34008720418175
0.075 -1.29157688408696
0.09 -1.23228624091214
0.105 -1.16343887968993
0.12 -1.08620461994918
0.135 -1.00169394352858
0.15 -0.910960679687785
0.165 -0.814996123424869
0.18 -0.714728847716216
0.195 -0.611037935620371
0.21 -0.504719304722994
0.225 -0.396489031473844
0.24 -0.286986206021921
0.255 -0.176770985586448
0.27 -0.0663279935167739
0.285 0.0439232764193056
0.3 0.153614995032601
0.315 0.262418471333915
0.33 0.370032846559515
0.345 0.476174875380767
0.36 0.580570523543127
0.375 0.682953814475104
0.39 0.783067120601909
0.405 0.880663317375175
0.42 0.975506876467647
0.435 1.06737678156087
0.45 1.1560710162441
0.465 1.24140486925371
0.48 1.32321080914984
0.495 1.40134094673241
0.51 1.47566697866063
0.525 1.54607982082183
0.54 1.61248901273963
0.555 1.67482218346915
0.57 1.73302892333771
0.585 1.78707050429461
0.6 1.8369211580649
0.615 1.8825697862643
0.63 1.9240190414484
0.645 1.96128445762734
0.66 1.99439459753085
0.675 2.02338919609467
0.69 2.04831672066671
0.705 2.06923547870529
0.72 2.0862129129471
0.735 2.09932487483148
0.75 2.10865493127885
0.765 2.11429353674973
0.78 2.11633770387869
0.795 2.11489035552794
0.81 2.11005959853603
0.825 2.10195780841034
0.84 2.09070217168342
0.855 2.07641370612784
0.87 2.05921660339411
0.885 2.0392377649039
0.9 2.01660565539968
0.915 1.99145042997463
0.93 1.96390443682748
0.945 1.93410096167321
0.96 1.90217374693496
0.975 1.8682563511769
0.99 1.83248142756342
1.005 1.79498204463472
1.02 1.75589044994949
1.035 1.71533764323955
1.05 1.67345307701138
1.065 1.63036438528986
1.08 1.58619640097839
1.095 1.5410720964252
1.11 1.49511223292198
1.125 1.44843503745236
1.14 1.40115585391479
1.155 1.35338694461428
1.17 1.30523734472753
1.185 1.2568125572291
1.2 1.20821419420227
1.215 1.15954099388798
1.23 1.11088807636746
1.245 1.06234684724385
1.26 1.0140049067782
1.275 0.965946000278505
1.29 0.918249988608962
1.305 0.870992818355311
1.32 0.824246512150005
1.335 0.778079164425739
1.35 0.732554832023027
1.365 0.68773352430265
1.38 0.643671727949742
1.395 0.600422087005636
1.41 0.558033434786558
1.425 0.516550850237993
1.44 0.476015710479588
1.455 0.436465753354711
1.47 0.397935147797714
1.485 0.360454568365592
1.5 0.324051272310256
1.515 0.288749186616752
1.53 0.254568994294683
1.545 0.221528225247945
1.56 0.189641353407821
1.575 0.15891989296992
1.59 0.129372499172933
1.605 0.101005071910592
1.62 0.073820855431548
1.635 0.047820550094406
1.65 0.0230024223059841
1.665 -0.00063759855820357
1.68 -0.023105804867005
1.695 -0.0444106030413339
1.71 -0.0645624059243359
1.725 -0.0835735224324159
1.74 -0.101458051974155
1.755 -0.118231774712135
1.77 -0.133912048652233
1.785 -0.148517701484388
1.8 -0.162068929624638
1.815 -0.174587193535541
1.83 -0.186095119451961
1.845 -0.196616399919503
1.86 -0.206175696904796
1.875 -0.214798547392264
1.89 -0.222511272561582
1.905 -0.229340880965754
1.92 -0.235314994059904
1.935 -0.240461743902895
1.95 -0.24480972266381
1.965 -0.248387865628463
1.98 -0.251225384326614
1.995 -0.253351719711861
2.01 -0.25479642513783
2.025 -0.255589136294524
2.04 -0.255759491020507
2.055 -0.255337085290671
2.07 -0.254351358035451
2.085 -0.252831552502172
2.1 -0.250806747190267
2.115 -0.248305703488751
2.13 -0.245356865787765
2.145 -0.241988293127238
2.16 -0.238227637262013
2.175 -0.23410208508908
2.19 -0.229638337173319
2.205 -0.22486255987979
2.22 -0.219800363318829
2.235 -0.214476772756851
2.25 -0.208916196623832
2.265 -0.203142403943328
2.28 -0.197178506219767
2.295 -0.191046936254468
2.31 -0.184769430797599
2.325 -0.178367018427025
2.34 -0.171860004843806
2.355 -0.165267964714881
2.37 -0.158609736567794
2.385 -0.151903406790444
2.4 -0.145166303372777
2.415 -0.138415021179725
2.43 -0.131665388148762
2.445 -0.124932474912811
2.46 -0.118230619839789
2.475 -0.111573398351795
2.49 -0.104973648158122
2.505 -0.0984434832956718
2.52 -0.0919942766888203
2.535 -0.0856366933558827
2.55 -0.0793806935550598
2.565 -0.0732355318105011
2.58 -0.0672097910050755
2.595 -0.0613113806731555
2.61 -0.0555475487919426
2.625 -0.0499249132330453
2.64 -0.0444494602809492
2.655 -0.0391265653625892
2.67 -0.0339610195304856
2.685 -0.0289570322131058
2.7 -0.0241182570331423
2.715 -0.0194478122600905
2.73 -0.0149482907092665
2.745 -0.0106217859987619
2.76 -0.00646990931432347
2.775 -0.00249380594973057
2.79 0.00130582366941922
2.805 0.00492870449385874
2.82 0.00837496414577651
2.835 0.0116451120091123
2.85 0.0147400344195119
2.865 0.017660966164018
2.88 0.020409466570789
2.895 0.0229874054697567
2.91 0.0253969486909506
2.925 0.0276405321470158
2.94 0.0297208455690227
2.955 0.0316408185607814
2.97 0.0334035996658004
2.985 0.0350125382873033
3 0.0364711697969224
};
\addplot [very thick, color0]
table {%
0 -1.1
0.015 0.0292951455014155
0.03 1.06117637011408
0.045 1.99910210996405
0.06 2.84660880892788
0.075 3.60724097322095
0.09 4.2845737158897
0.105 4.88206733868006
0.12 5.40352947801299
0.135 5.85284053993647
0.15 6.2339073169424
0.165 6.55110690023153
0.18 6.80777288617405
0.195 7.00870919280063
0.21 7.15924583085611
0.225 7.26450228639923
0.24 7.32968538454003
0.255 7.36029283481447
0.27 7.36077982746213
0.285 7.33529051941988
0.3 7.28663786291431
0.315 7.21715337383959
0.33 7.12829012435751
0.345 7.02091356254436
0.36 6.8955412864382
0.375 6.75270421493706
0.39 6.59300295940373
0.405 6.41720113032653
0.42 6.22614854573793
0.435 6.02091497786574
0.45 5.80290622561952
0.465 5.57307910036696
0.48 5.33273650237499
0.495 5.08319007534803
0.51 4.82573619711673
0.525 4.56163834932858
0.54 4.29211511411104
0.555 4.01845896668946
0.57 3.74198153823772
0.585 3.46327794865179
0.6 3.18335548936518
0.615 2.90316097150854
0.63 2.62358211169604
0.645 2.34544937748956
0.66 2.06967665194165
0.675 1.79680931077661
0.69 1.52750908305838
0.705 1.26240868157441
0.72 1.00209287206909
0.735 0.74710039935027
0.75 0.497912824883264
0.765 0.254980409453499
0.78 0.0187190684772087
0.795 -0.210495488704264
0.81 -0.432362346561774
0.825 -0.646572478591026
0.84 -0.852833471023727
0.855 -1.05091520826527
0.87 -1.24062144870456
0.885 -1.42178861038726
0.9 -1.59437487128073
0.915 -1.75818382855028
0.93 -1.91313327850553
0.945 -2.05917749392222
0.96 -2.19629701067825
0.975 -2.32455720601171
0.99 -2.44395221684875
1.005 -2.55449528556432
1.02 -2.65626011393886
1.035 -2.74934035100877
1.05 -2.8338482795731
1.065 -2.90991355540375
1.08 -2.97776006489844
1.095 -3.03746817483251
1.11 -3.08919994825877
1.125 -3.13314376565092
1.14 -3.16949923896062
1.155 -3.1984763870968
1.17 -3.22029445939133
1.185 -3.2352188571105
1.2 -3.24346298612807
1.215 -3.24524068536624
1.23 -3.24079938409723
1.245 -3.23039084963163
1.26 -3.21427042630393
1.275 -3.19269620406246
1.29 -3.16592826726875
1.305 -3.13422790252428
1.32 -3.09785684071828
1.335 -3.05707658931086
1.35 -3.0121703062973
1.365 -2.9633763991817
1.38 -2.91094500012785
1.395 -2.85513146040279
1.41 -2.7961880956687
1.425 -2.73436356802683
1.44 -2.66990261068415
1.455 -2.60304553821639
1.47 -2.53402788992831
1.485 -2.46308007774361
1.5 -2.39042705982974
1.515 -2.31628806760592
1.53 -2.24087633928428
1.545 -2.16439887965718
1.56 -2.08705626883732
1.575 -2.00904247102491
1.59 -1.93054468004497
1.605 -1.85174319917416
1.62 -1.7728112807295
1.635 -1.69391510728295
1.65 -1.61521372785804
1.665 -1.53685894357185
1.68 -1.45899532624606
1.695 -1.38176020555317
1.71 -1.30528366902094
1.725 -1.22968860072156
1.74 -1.15509071164376
1.755 -1.08159860595705
1.77 -1.00931384214068
1.785 -0.938331026346025
1.8 -0.868737899093226
1.815 -0.800615450084608
1.83 -0.734038031531535
1.845 -0.669073487221626
1.86 -0.605783291596572
1.875 -0.544222694289698
1.89 -0.48444087109922
1.905 -0.426481098419589
1.92 -0.370380889721167
1.935 -0.316172208735861
1.95 -0.263881577474516
1.965 -0.213530329361431
1.98 -0.165134783344005
1.995 -0.118706359149541
2.01 -0.0742518813626214
2.025 -0.0317736820871563
2.04 0.00873018661995745
2.055 0.0472657951639418
2.07 0.0838429930940818
2.085 0.118475282919282
2.1 0.151179802709468
2.115 0.181976921531997
2.13 0.210890134807377
2.145 0.237945852202402
2.16 0.263173223329576
2.175 0.286603943156656
2.19 0.308272076712138
2.205 0.328213869195563
2.22 0.346467581198059
2.235 0.363073315545349
2.25 0.37807284295457
2.265 0.391509437697842
2.28 0.403427720430765
2.295 0.413873500007548
2.31 0.422893619658603
2.325 0.430535813998105
2.34 0.43684856170615
2.355 0.441880951562902
2.37 0.445682557473861
2.385 0.448303286907608
2.4 0.449793253046588
2.415 0.450202723873646
2.43 0.449581926793826
2.445 0.447980963856848
2.46 0.445449771129314
2.475 0.442037921816455
2.49 0.437794594102976
2.505 0.432768506100116
2.52 0.427007749834634
2.535 0.420559798024186
2.55 0.413471404020885
2.565 0.405788487563347
2.58 0.397556160956802
2.595 0.388818606016887
2.61 0.379619010654031
2.625 0.369999599480211
2.64 0.360001503214079
2.655 0.349664742864043
2.67 0.339028248979594
2.685 0.328129742087897
2.7 0.317005758086275
2.715 0.305691639428671
2.73 0.294221450518151
2.745 0.282628027370611
2.76 0.270942941152321
2.775 0.259196459079477
2.79 0.247417559214823
2.805 0.235633950861007
2.82 0.223872063858223
2.835 0.212157068570134
2.85 0.200512801870761
2.865 0.188961826481199
2.88 0.177525467755449
2.895 0.166223799313707
2.91 0.155075634665145
2.925 0.144098582683842
2.94 0.1333090549695
2.955 0.122722269947322
2.97 0.11235228769418
2.985 0.102212036317543
3 0.0923133303325705
};
\addplot [very thick, color1]
table {%
0 -2
0.015 -2.13288536842758
0.03 -2.25030820428665
0.045 -2.35271878844116
0.06 -2.44025326093538
0.075 -2.51269396090066
0.09 -2.56954476485668
0.105 -2.61017583288827
0.12 -2.63379541043898
0.135 -2.63943306634317
0.15 -2.62616999464938
0.165 -2.59323261639557
0.18 -2.54014847068027
0.195 -2.46680104079552
0.21 -2.37361299311247
0.225 -2.26167342387257
0.24 -2.13274252511006
0.255 -1.98926708308898
0.27 -1.83430001700726
0.285 -1.6712879397217
0.3 -1.50381513637679
0.315 -1.33536787541592
0.33 -1.1691466597783
0.345 -1.00789878769327
0.36 -0.853754609286697
0.375 -0.708255662685313
0.39 -0.572409175536643
0.405 -0.446761365794823
0.42 -0.331495342386639
0.435 -0.226548467032156
0.45 -0.131598884022128
0.465 -0.0462116208002013
0.48 0.0300806550386142
0.495 0.0977908501144526
0.51 0.157450952588863
0.525 0.209594315926599
0.54 0.254744692577734
0.555 0.293412832509116
0.57 0.326160736823453
0.585 0.35346776960434
0.6 0.375757565696237
0.615 0.393435806695191
0.63 0.406890813062272
0.645 0.416494223767895
0.66 0.42261827532133
0.675 0.425619473286871
0.69 0.425809412614942
0.705 0.423486017154258
0.72 0.418933737708021
0.735 0.412423245501894
0.75 0.404207676185411
0.765 0.394522785950678
0.78 0.383591032356442
0.795 0.371621408873884
0.81 0.358803933428787
0.825 0.345295085349109
0.84 0.331258224163162
0.855 0.31684660898199
0.87 0.302200512586964
0.885 0.287447399044623
0.9 0.272675098978782
0.915 0.257951589029677
0.93 0.243371699343168
0.945 0.229019507398052
0.96 0.214968611652048
0.975 0.201268524273181
0.99 0.187922918302753
1.005 0.174978564179987
1.02 0.162479244075163
1.035 0.150460914432724
1.05 0.138952265045466
1.065 0.127975269341853
1.08 0.117512671204112
1.095 0.107528870782311
1.11 0.0980383252905458
1.125 0.0890508666597011
1.14 0.0805718263962997
1.155 0.0726023675376793
1.17 0.0651399531924683
1.185 0.0581653338997867
1.2 0.0516201277676578
1.215 0.0454967222436438
1.23 0.0397905051557229
1.245 0.0344944628122212
1.26 0.0295994409278508
1.275 0.0250944357095646
1.29 0.0209668541735065
1.305 0.0172027854004502
1.32 0.0137872511084532
1.335 0.0107044257353581
1.35 0.00792207754563346
1.365 0.00539086641859797
1.38 0.00309999584969402
1.395 0.00103951365649746
1.41 -0.000801138756068394
1.425 -0.00243294073553457
1.44 -0.00386721404918871
1.455 -0.00511547895639905
1.47 -0.00618935370341358
1.485 -0.0071004563362198
1.5 -0.00786031488149184
1.515 -0.00848029602358861
1.53 -0.00897153567283509
1.545 -0.00934487834842953
1.56 -0.00961083242003197
1.575 -0.00977952462471053
1.59 -0.00986066626466
1.605 -0.00986352916802348
1.62 -0.00979690966988893
1.635 -0.00966913287836797
1.65 -0.00948804593795121
1.665 -0.00926098894975588
1.68 -0.0089948100799501
1.695 -0.00869586679779931
1.71 -0.0083700297940806
1.725 -0.00802269720869535
1.74 -0.00765880232647029
1.755 -0.0072828329392538
1.77 -0.00689884334506002
1.785 -0.00651047714171551
1.8 -0.00612098194035369
1.815 -0.00573323356986576
1.83 -0.00534975418268901
1.845 -0.00497273488246304
1.86 -0.00460405779844344
1.875 -0.00424531721973751
1.89 -0.00389784091539086
1.905 -0.00356271293618621
1.92 -0.00324079119218857
1.935 -0.00293273187741602
1.95 -0.00263900124857711
1.965 -0.00235990319854629
1.98 -0.00209559471292057
1.995 -0.00184609664261593
2.01 -0.00161132040004975
2.025 -0.00139107388871683
2.04 -0.0011850795956812
2.055 -0.000992983124731128
2.07 -0.000814377873984153
2.085 -0.000648812808191986
2.1 -0.000495781847132732
2.115 -0.000354758137067412
2.13 -0.000225189742343158
2.145 -0.000106514117292618
2.16 1.84144365209897e-06
2.175 0.000100448412617156
2.19 0.000189876361605595
2.205 0.000270684631756727
2.22 0.00034342116517753
2.235 0.000408619493041115
2.25 0.000466794311648538
2.265 0.000518439893687223
2.28 0.000564029299763961
2.295 0.000604012573289744
2.31 0.000638815896090776
2.325 0.000668841508929902
2.34 0.000694467159005975
2.355 0.000716046552103106
2.37 0.00073390996158838
2.385 0.000748363646399215
2.4 0.000759690376709227
2.415 0.000768153578549982
2.43 0.000773994715669522
2.445 0.000777435200451169
2.46 0.000778679418029791
2.475 0.000777913390772448
2.49 0.000775307046537554
2.505 0.000771015757891552
2.52 0.000765180658150441
2.535 0.000757930461455395
2.55 0.000749382573196386
2.565 0.000739644058868495
2.58 0.000728812448229372
2.595 0.00071697711723787
2.61 0.000704220291269425
2.625 0.000690617224203863
2.64 0.000676237899040086
2.655 0.000661147486741853
2.67 0.000645406296767845
2.685 0.000629071620154487
2.7 0.000612197496329566
2.715 0.000594834973843286
2.73 0.000577033534740854
2.745 0.000558840387732768
2.76 0.000540301157558891
2.775 0.000521460619047845
2.79 0.000502362568846993
2.805 0.000483049641750597
2.82 0.000463563551913231
2.835 0.00044394489870878
2.85 0.000424234411879418
2.865 0.000404472140021061
2.88 0.000384696985346051
2.895 0.000364947183655697
2.91 0.000345260823752707
2.925 0.000325674833528539
2.94 0.000306225189115334
2.955 0.000286947307813133
2.97 0.000267875617197219
2.985 0.000249043409758516
3 0.00023048276622639
};
\addplot [very thick, color2]
table {%
0 0.3
0.015 3.11293534685437
0.03 5.87238888499493
0.045 8.5474933880044
0.06 11.1078863487935
0.075 13.5155031354426
0.09 15.7310896315521
0.105 17.7050753528903
0.12 19.4071506082042
0.135 20.8008988130965
0.15 21.8559402425368
0.165 22.5528317703723
0.18 22.891136886343
0.195 22.8846649798855
0.21 22.5652467231881
0.225 21.9808492067964
0.24 21.1890727250078
0.255 20.2433343858566
0.27 19.2044271449232
0.285 18.1172523162394
0.3 17.024810347078
0.315 15.9494803867561
0.33 14.9059227597539
0.345 13.9028458364975
0.36 12.9456802911883
0.375 12.0344604839118
0.39 11.1658812542564
0.405 10.3418757270232
0.42 9.55506664866741
0.435 8.80468179098829
0.45 8.0954004165022
0.465 7.41898169558314
0.48 6.77476220015548
0.495 6.16224146648335
0.51 5.58103947529059
0.525 5.03084751916926
0.54 4.511381600065
0.555 4.0241759543169
0.57 3.57196733976603
0.585 3.14779452648455
0.6 2.75133259923302
0.615 2.38215352934169
0.63 2.03972370850385
0.645 1.72340618237315
0.66 1.43482311564852
0.675 1.17017068853901
0.69 0.928318712971804
0.705 0.708413260159924
0.72 0.509537222015879
0.735 0.330665602375766
0.75 0.170405439447816
0.765 0.0277368431364636
0.78 -0.0983081252069369
0.795 -0.208710100521828
0.81 -0.305699621901154
0.825 -0.390026282980958
0.84 -0.461698105853544
0.855 -0.521621708212518
0.87 -0.57068673978512
0.885 -0.609761725280103
0.9 -0.643585760922873
0.915 -0.669167782571656
0.93 -0.686908796521982
0.945 -0.697544704712807
0.96 -0.701781498702159
0.975 -0.703280257544266
0.99 -0.701133462722941
1.005 -0.694053458362593
1.02 -0.682623574422153
1.035 -0.667396205163696
1.05 -0.648892335746636
1.065 -0.62760135153372
1.08 -0.608903172421194
1.095 -0.588331466700695
1.11 -0.565606127576072
1.125 -0.541122724297569
1.14 -0.51524884736892
1.155 -0.488323994297689
1.17 -0.460660217329994
1.185 -0.43544413123841
1.2 -0.411392882718219
1.215 -0.386764021856084
1.23 -0.361804775176903
1.245 -0.336739435056403
1.26 -0.311769840931214
1.275 -0.287076206027072
1.29 -0.262817992973349
1.305 -0.239134865160908
1.32 -0.216147771388465
1.335 -0.193959936349394
1.35 -0.175348124332114
1.365 -0.15803826701678
1.38 -0.141223717858779
1.395 -0.124987131151411
1.41 -0.109397730654665
1.425 -0.0945123763760194
1.44 -0.0803762189922051
1.455 -0.0670236564767739
1.47 -0.0544791856736803
1.485 -0.042758247957927
1.5 -0.0318680636119481
1.515 -0.0218084395378755
1.53 -0.0125725461980297
1.545 -0.00414768140922727
1.56 0.00348402381598063
1.575 0.0103449103089175
1.59 0.0164611400323287
1.605 0.0218621028552107
1.62 0.0265797622173594
1.635 0.0306482316523056
1.65 0.0341032561968369
1.665 0.0369816888267329
1.68 0.0393211192727866
1.695 0.0411594895834087
1.71 0.0425347406986877
1.725 0.0434845143909248
1.74 0.0440458658519046
1.755 0.0442550355550113
1.77 0.0441472393288647
1.785 0.0437564974361268
1.8 0.0431154667134197
1.815 0.042255332481973
1.83 0.0412057037247615
1.845 0.0399945381898359
1.86 0.0386480932770546
1.875 0.03719088924602
1.89 0.0356456939859845
1.905 0.0340335257876518
1.92 0.0323736585583451
1.935 0.0306836639857883
1.95 0.0289794338557109
1.965 0.0272752372542884
1.98 0.0255837850615983
1.995 0.0239162667521413
2.01 0.0222824590261739
2.025 0.0206907725752579
2.04 0.0191483290629231
2.055 0.0176610331511465
2.07 0.0162337112253485
2.085 0.0148701600475834
2.1 0.0135731800833371
2.115 0.0123447210785502
2.13 0.0111859355593734
2.145 0.0100972658101338
2.16 0.00907851046107647
2.175 0.00812890818583521
2.19 0.00724719995365667
2.205 0.00643170805557892
2.22 0.00568038637205776
2.235 0.00499088047800138
2.25 0.0043605905629799
2.265 0.00378672198997042
2.28 0.00326632894640055
2.295 0.00279636074221925
2.31 0.0023737031356963
2.325 0.00199521043863554
2.34 0.00165774201945097
2.355 0.00135818767345721
2.37 0.00109348765416469
2.385 0.00086067144107209
2.4 0.000656879698539859
2.415 0.000479331725091165
2.43 0.000325395081735088
2.445 0.00019258099142964
2.46 7.85143995725706e-05
2.475 -1.90036941380179e-05
2.49 -0.00010202436095381
2.505 -0.00017246073919416
2.52 -0.000232049369218977
2.535 -0.000282387009689037
2.55 -0.000324923810306014
2.565 -0.000360950305829637
2.58 -0.000391630093159831
2.595 -0.00041798793585398
2.61 -0.000440913595439816
2.625 -0.000461186374796218
2.64 -0.000479466101594637
2.655 -0.000496304867591211
2.67 -0.000512162485867167
2.685 -0.00052740479743176
2.7 -0.000542316202144217
2.715 -0.000557108996825897
2.73 -0.000571928460506468
2.745 -0.000586860949537095
2.76 -0.000601941992469947
2.775 -0.000617164512726433
2.79 -0.000632484594378777
2.805 -0.000647823745039201
2.82 -0.000663082180069386
2.835 -0.000678143620921964
2.85 -0.000692871888032189
2.865 -0.000707120725616055
2.88 -0.000720740564871115
2.895 -0.000733580067671442
2.91 -0.00074548821443156
2.925 -0.000756319498383775
2.94 -0.000765935889960674
2.955 -0.000774209094042588
2.97 -0.000781022431535986
2.985 -0.000786272327263789
3 -0.000789868934697112
};
\end{axis}

\end{tikzpicture}

%% file: chapters/conclusion.tex
\section{Conclusion}
We analysed the stability of trajectories $\xi$ of an SDC system like \eqref{eq:extlinsys} based on properties of the spectrum of $A(\xi(t))$. The straight-forward adaptation of known sufficient conditions for linear time-varying systems came with strong global assumptions that are unlikely to be fulfilled. Taking into account that the coefficient function $\xi(t) \mapsto A(\xi(t))$ is stabilized together with the trajectory, we derived sufficient conditions for stability that can be checked locally. In view of using the obtained theoretical results for feedback stabilization, we developed an update scheme that ensures uniform decay rates and bounds on the transient behavior of the closed-loop SDC system matrix. The usability of the sufficient conditions and the efficiency of the approach to stabilization via updating an initial feedback was illustrated in numerical examples.

By now, in the numerical examples as well as in the theoretical investigations, we have not considered the potentials for optimization within the derived approaches. For example, the freedom in the choice of the weighting matrix perturbation $\Rdeltapmo$ may well be used to optimize the feedback update $E$. Additionally, it might be worth investigating whether structural assumptions on the changes $\Delta$ in the coefficient matrices can be exploited to provide feedback updates of, e.g., low-rank.